\documentclass[a4paper, 11pt, reqno]{amsart}

\usepackage[margin=2cm]{geometry}
\usepackage{graphicx}
\usepackage{amsmath}
\usepackage{amsthm,amsfonts,amssymb,mathrsfs,amscd,amstext,amsbsy}
\usepackage{epic,eepic}
\usepackage{enumerate}
\usepackage[all]{xy}
\usepackage{tikz}
\usepackage{color}
\usepackage{verbatim}

\usepackage[english]{babel}
\usepackage{enumitem}
\usepackage{hyperref}

\usepackage{tikz}

\usetikzlibrary{patterns}
\usetikzlibrary{decorations.markings}
\usetikzlibrary{arrows}
\usetikzlibrary{calc}

\tikzset{middlearrow/.style={
        decoration={markings,
            mark= at position 0.6 with {\arrow{#1}} ,
        },
        postaction={decorate}
    }
}

\tikzset{middleupuparrow/.style={
        decoration={markings,
            mark= at position 0.9 with {\arrow{#1}} ,
        },
        postaction={decorate}
    }
}
\tikzset{middleuparrow/.style={
        decoration={markings,
            mark= at position 0.65 with {\arrow{#1}} ,
        },
        postaction={decorate}
    }
}

\hypersetup{
    pdftitle=   {Sparse highly connected spanning subgraphs in dense directed graphs},
   pdfauthor=  {Dong Yeap Kang}
}

\makeatletter
\def\blfootnote{\gdef\@thefnmark{}\@footnotetext}
\makeatother
\newtheorem{THM}{Theorem}[section]
\newtheorem*{THM*}{Theorem~\ref{main}}
\newtheorem{LEM}[THM]{Lemma}
\newtheorem{OBS}[THM]{Observation}

\newtheorem{COR}[THM]{Corollary}
\newtheorem{PROP}[THM]{Proposition}
\newtheorem{CONJ}[THM]{Conjecture}
\newtheorem{CLAIM}{Claim}

\theoremstyle{remark}

\newtheorem*{REM}{Remark}

\theoremstyle{definition}
\newtheorem{DEFN}[THM]{Definition}

\numberwithin{equation}{section}

\title{Sparse highly connected spanning subgraphs in dense directed graphs}
\author{Dong Yeap Kang}
\address{Department of Mathematical Sciences, KAIST, 291 Daehak-ro
  Yuseong-gu Daejeon, 34141 South Korea}
\thanks{Supported by the National Research Foundation of Korea (NRF) grant funded by the Korea government (MSIT) (No. NRF-2017R1A2B4005020) and also by TJ Park Science Fellowship of POSCO TJ Park Foundation.}
\email{dyk90@kaist.ac.kr}

\date{\today}

\begin{document}
\begin{abstract}
Mader proved that every strongly $k$-connected $n$-vertex digraph contains a strongly $k$-connected spanning subgraph with at most $2kn - 2k^2$ edges, where the equality holds for the complete bipartite digraph ${DK}_{k,n-k}$. For dense strongly $k$-connected digraphs, this upper bound can be significantly improved. More precisely, we prove that every strongly $k$-connected $n$-vertex digraph $D$ contains a strongly $k$-connected spanning subgraph with at most $kn + 800k(k+\overline{\Delta}(D))$ edges, where $\overline{\Delta}(D)$ denotes the maximum degree of the complement of the underlying undirected graph of a digraph $D$. Here, the additional term $800k(k+\overline{\Delta}(D))$ is tight up to multiplicative and additive constants. As a corollary, this implies that every strongly $k$-connected $n$-vertex semicomplete digraph contains a strongly $k$-connected spanning subgraph with at most $kn + 800k^2$ edges, which is essentially optimal since $800k^2$ cannot be reduced to the number less than $k(k-1)/2$. 

We also prove an analogous result for strongly $k$-arc-connected directed multigraphs. Both proofs yield polynomial-time algorithms.\\

\textup{2010} \textit{Mathematics Subject Classification}:\:\:Primary: 05C20; Secondary: 05C40
\end{abstract}
%\blfootnote{\textup{2010} \textit{Mathematics Subject Classification}:05C20, 05C40}

%\keywords{Extremal graph theory, directed graphs, oriented graphs}
\maketitle

\section{Introduction}\label{sec:intro}

Given a strongly connected digraph, what is the minimum number of edges of a strongly connected spanning subgraph? This minimum spanning strongly connected subgraph problem (or {\em MSSS}) is NP-hard, since it generalises the Hamiltonian cycle problem. The problem is closely related to both extremal graph theory and combinatorial optimization in perspective of studying the properties of extremal graphs and algorithmic aspects, and especially to industry, in order to build well-connected road systems with minimal cost. Even though the problem is NP-hard, it is known that the problem is polynomial-time solvable for various classes of digraphs~\cite{bang2003, bang2001}, and there are algorithms that approximate the minimum number of edges of a strongly connected spanning subgraph~\cite{bang2004spanning, vetta2001}.

One of the natural generalisations of the MSSS problem is the problem of determining the minimum number of edges in a strongly $k$-connected (or $k$-arc-connected) spanning subgraph of a strongly $k$-connected (or $k$-arc-connected, respectively) digraph. Even though the problem is known to be NP-hard~\cite{GJ1979}, there are algorithms that approximate the minimum number of edges of a strongly $k$-connected (or $k$-arc-connected) spanning subgraph~\cite{cheriyan2000}. For more on algorithmic aspects of both problems and their variants, the readers are referred to~\cite{bang2009problems},~\cite[Chapter 12]{bang2008digraphs} and the recent survey~\cite{bang2018} on tournaments and semicomplete digraphs.

We investigate an upper bound of the minimum number of edges in a strongly $k$-connected spanning subgraph and a strongly $k$-arc-connected spanning subgraph. The following are well-known results for general digraphs and directed multigraphs.

\begin{itemize} 
\item[$(\rm 1)$] (Mader~\cite{mader1985}) For integers $k \geq 1$ and $n \geq 4k+3$, every strongly $k$-connected $n$-vertex digraph contains a strongly $k$-connected spanning subgraph with at most $2k(n-k)$ edges.

\item[$(\rm 2)$] (Dalmazzo~\cite{dalmazzo1977}) For integers $k,n \geq 1$, every strongly $k$-arc-connected $n$-vertex directed multigraph contains a strongly $k$-arc-connected spanning subgraph with at most $2k(n-1)$ edges.

\item[$(\rm 3)$] (Berg and Jord\'{a}n~\cite{berg2005}) There exists a function $h(k)$ such that for integers $k \geq 1$ and $n \geq h(k)$, every strongly $k$-arc-connected $n$-vertex digraph contains a strongly $k$-arc-connected spanning subgraph with at most $2k(n-k)$ edges.
\end{itemize}

%%%%%%%%%%%%%%%%%%%%%%%%%%%%%
\begin{figure}[h]
\centering
\begin{tikzpicture}[scale=0.8]

\draw[fill=none] (0,0) ellipse [x radius=1,y radius=3];
\draw[fill=none] (4,0) ellipse [x radius=1,y radius=3];

%% DK_{4,4} %%
\def \n {4}

\foreach \s in {1,...,\n}{
  \filldraw[fill=black] (0, -3 + \s*1.2) circle (2pt);
}

\foreach \s in {1,...,\n}{
  \filldraw[fill=black] (4, -3 + \s*1.2) circle (2pt);
}

\foreach \i in {1,...,\n}{
	\foreach \j in {1,...,\n}{
		\draw[middleuparrow={latex}] (0, -3 + \i*1.2) to[bend left=5] (4, -3 + \j*1.2) ;
		\draw[middleuparrow={latex}] (4, -3 + \i*1.2) to[bend left=5] (0, -3 + \j*1.2) ;
	}
}

\filldraw[fill=black] (11, 3) circle (2pt);
\filldraw[fill=black] (9, 0) circle (2pt);
\filldraw[fill=black] (13, 0) circle (2pt);
\filldraw[fill=black] (8, -3) circle (2pt);
\filldraw[fill=black] (10, -3) circle (2pt);
\filldraw[fill=black] (12, -3) circle (2pt);
\filldraw[fill=black] (14, -3) circle (2pt);

\draw[middleuparrow={latex}] (11,3) to[bend left=10] (9,0) ;
\draw[middleuparrow={latex}] (9,0) to[bend left=10] (11,3) ;
\draw[middleuparrow={latex}] (11,3) to[bend left=20] (9,0) ;
\draw[middleuparrow={latex}] (9,0) to[bend left=20] (11,3) ;

\draw[middleuparrow={latex}] (11,3) to[bend left=10] (13,0) ;
\draw[middleuparrow={latex}] (13,0) to[bend left=10] (11,3) ;
\draw[middleuparrow={latex}] (11,3) to[bend left=20] (13,0) ;
\draw[middleuparrow={latex}] (13,0) to[bend left=20] (11,3) ;

\draw[middleuparrow={latex}] (9,0) to[bend left=10] (8,-3) ;
\draw[middleuparrow={latex}] (8,-3) to[bend left=10] (9,0) ;
\draw[middleuparrow={latex}] (9,0) to[bend left=20] (8,-3) ;
\draw[middleuparrow={latex}] (8,-3) to[bend left=20] (9,0) ;

\draw[middleuparrow={latex}] (9,0) to[bend left=10] (10,-3) ;
\draw[middleuparrow={latex}] (10,-3) to[bend left=10] (9,0) ;
\draw[middleuparrow={latex}] (9,0) to[bend left=20] (10,-3) ;
\draw[middleuparrow={latex}] (10,-3) to[bend left=20] (9,0) ;

\draw[middleuparrow={latex}] (13,0) to[bend left=10] (12,-3) ;
\draw[middleuparrow={latex}] (12,-3) to[bend left=10] (13,0) ;
\draw[middleuparrow={latex}] (13,0) to[bend left=20] (12,-3) ;
\draw[middleuparrow={latex}] (12,-3) to[bend left=20] (13,0) ;

\draw[middleuparrow={latex}] (13,0) to[bend left=10] (14,-3) ;
\draw[middleuparrow={latex}] (14,-3) to[bend left=10] (13,0) ;
\draw[middleuparrow={latex}] (13,0) to[bend left=20] (14,-3) ;
\draw[middleuparrow={latex}] (14,-3) to[bend left=20] (13,0) ;
\end{tikzpicture}
\caption{$DK_{4,4}$ and the directed multigraph obtained from the 7-vertex tree whose edges are replaced by 2 directed 2-cycles.}
\end{figure}
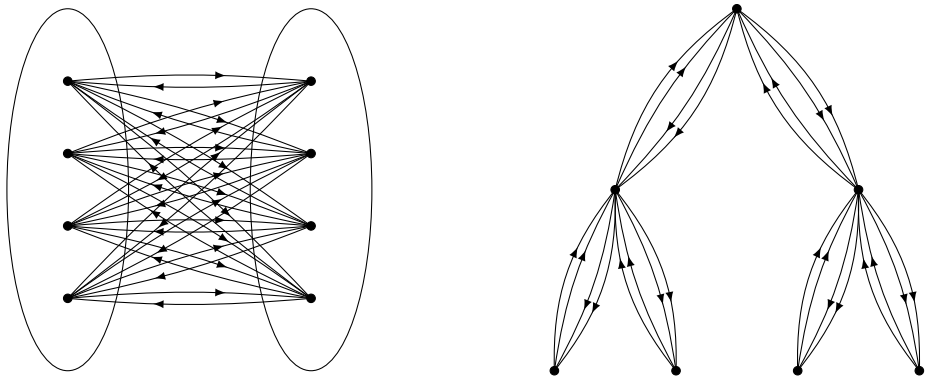

The upper bounds for these three cases are best possible; the digraph ${DK}_{k,n-k}$ obtained from $K_{k,n-k}$\footnote{An undirected graph $K_{k,n-k}$ is a complete bipartite graph with two independent sets of  size $k$ and size $n-k$, respectively.} by replacing each edge with a directed 2-cycle shows that the upper bounds given in $(\rm 1)$ and $(\rm 3)$ are tight, and a directed multigraph obtained from an $n$-vertex tree by replacing each edge with $k$ directed 2-cycles shows that $(\rm 2)$ cannot be improved.

%In fact, Mader~\cite{mader1985} proved that for $k \geq 2$\footnote{For $k=1$, then there are numerous examples of strongly connected $n$-vertex digraphs with exactly $2n-2$ edges, for example, a digraph obtained from a tree by replacing each edge with a directed 2-cycle.} and $n \geq 4k+\max(5-k , 0)$, every minimally strongly $k$-connected $n$-vertex digraph has at most $2kn - 2k^2$ edges, and exactly $2kn - 2k^2$ edges if and only if the digraph is isomorphic to ${DK}_{k,n-k}$.

Nevertheless, one may ask whether those upper bounds given in $(\rm 1)$--$(\rm 3)$ can be improved for dense digraphs, because all of these extremal examples are sparse. As a starting point, Bang-Jensen, Huang, and Yeo~\cite{bang2004spanning} proved the following result that improves the result of Berg and Jord\'{a}n for tournaments. 

\begin{THM}[Bang-Jensen, Huang, and Yeo~\cite{bang2004spanning}]\label{thm:old1}
For all integers $k,n \geq 1$, every strongly $k$-arc-connected $n$-vertex tournament contains a strongly $k$-arc-connected spanning subgraph with at most $kn + 136k^2$ edges.
\end{THM}

They also proved that the number $136k^2$ of additional edges cannot be reduced to the number less than $\frac{k(k-1)}{2}$, so the result is essentially best possible. In 2009, Bang-Jensen~\cite{bang2009problems} asked whether there is a function $g(k)$ such that every strongly $k$-connected $n$-vertex tournament contains a strongly $k$-connected spanning subgraph with at most $kn + g(k)$ edges. Recently, Kim, Kim, Suh and the author~\cite{kang2016sparse} answered the question affirmatively.

\begin{THM}[Kang, Kim, Kim, and Suh~\cite{kang2016sparse}]\label{thm:old2}
For all integers $k,n \geq 1$, every strongly $k$-connected $n$-vertex tournament contains a strongly $k$-connected spanning subgraph with at most $kn + 750 k^2 \log_2 (k+1)$ edges.
\end{THM}

In particular, they answered the question of Bang-Jensen with $g(k) = 750 k^2 \log_2 (k+1)$. Since an example of Bang-Jensen, Huang, and Yeo~\cite{bang2004spanning} shows that $g(k) \geq \frac{k(k-1)}{2}$, there is a gap between the lower bound $\frac{k(k-1)}{2}$ and the upper bound $750 k^2 \log_2 (k+1)$ of $g(k)$. We close this gap by showing that $g(k) = \Theta(k^2)$ and generalise both Theorems~\ref{thm:old1} and~\ref{thm:old2} to a larger class of directed digraphs and directed multigraphs, respectively.

Before stating the results, let us begin with some terminology. Let $UG(D)$ be an \emph{underlying graph} of a directed multigraph $D$, a simple undirected graph obtained from $D$ by removing orientations of edges and multiple edges. Let $\overline{\Delta}(D)$ be the maximum degree of the complement of $UG(D)$, which is equal to $\max_{v \in V(D)}\lvert \left \{w \in V(D) \setminus \left \{v \right \} : (v,w),(w,v) \notin E(D) \right \} \rvert$. A directed multigraph $D$ is \emph{semicomplete} if $\overline{\Delta}(D) = 0$.

Bang-Jensen, Huang, and Yeo~\cite[Theorem 8.3]{bang2004spanning} proved that every strongly connected digraph $D$ contains a strongly connected spanning subgraph with at most $n + \overline{\Delta}(D)$ edges. We generalise this to strongly $k$-connected digraphs and strongly $k$-arc-connected directed multigraphs as follows.

\begin{THM}\label{thm:main} For integers $k,n \geq 1$, the following hold.
\begin{itemize} 
\item[$(\rm 1)$] Every strongly $k$-connected $n$-vertex \emph{digraph} $D$ contains a strongly $k$-connected spanning subgraph with at most $kn + 800 k \overline{\Delta}(D) + 800 k^2$ edges.

\item[$(\rm 2)$] Every strongly $k$-arc-connected $n$-vertex \emph{directed multigraph} $D$ contains a strongly $k$-arc-connected spanning subgraph with at most $kn + 670 k \overline{\Delta}(D) + 670 k^2$ edges. 
\end{itemize}
\end{THM}

\begin{REM}~\
\begin{enumerate}
\item Theorem~\ref{thm:main} gives the better result for ``dense'' digraphs and directed multigraphs. Given any $0 < \varepsilon < 1$, Theorem~\ref{thm:main} $(\rm 1)$ implies that any strongly $k$-connected $n$-vertex digraph $D$ with $\overline{\Delta}(D) < (1-\varepsilon)n / 800$ has a strongly $k$-connected spanning subgraph of $D$ with at most $(2 - \varepsilon)kn + 800k^2$ edges, improving the result of Mader~\cite{mader1985} for these dense digraphs. Similarly, the result of Dalmazzo~\cite{dalmazzo1977} is also improved for strongly $k$-arc-connected $n$-vertex directed multigraphs with $\overline{\Delta}(D) < (1-\varepsilon)n / 670$.

\item Both additional terms $800 k(k+\overline{\Delta}(D))$ and $670 k(k + \overline{\Delta}(D))$ are optimal up to multiplicative and additive constants.  In Section~\ref{sec:lowerbdd}, it is proved that for all integers $k \geq 1$, $\overline{\Delta} \geq 0$ and $n \geq \max(5k+2 , \: 4k+\overline{\Delta}+3)$, there is a strongly $k$-connected $n$-vertex oriented graph $G$ with $\overline{\Delta}(G) \leq \overline{\Delta}$ such that every spanning subgraph $D$ with $\delta^+(D),\delta^-(D) \geq k$ contains at least $kn + \max \left ( \frac{k(k-1)}{2} , k \overline{\Delta} \right )$ edges. 
\end{enumerate}
\end{REM}

Note that the class of tournaments is a subclass of the class of semicomplete digraphs. Theorem~\ref{thm:main} proves that $g(k) = O(k^2)$ suffices, which improves Theorem~\ref{thm:old2} and provides a function that is asymptotically sharp for the question of Bang-Jensen. Moreover, Theorem~\ref{thm:main} extends Theorems~\ref{thm:old1} and~\ref{thm:old2} to semicomplete directed multigraphs.

\begin{COR}\label{cor:tournaments} For all integers $k,n \geq 1$, the following hold.
\begin{itemize} 
\item[$(\rm 1)$] Every strongly $k$-connected $n$-vertex semicomplete digraph $D$ contains a strongly $k$-connected spanning subgraph with at most $kn + 800k^2$ edges.

\item[$(\rm 2)$] Every strongly $k$-arc-connected $n$-vertex semicomplete directed multigraph $D$ contains a strongly $k$-arc-connected spanning subgraph with at most $kn + 670k^2$ edges.
\end{itemize}
\end{COR}

One of the main ideas of the proof is the use of transitive subtournaments that dominate almost all vertices in order to link the vertices, which builds on the recent methods (see~\cite{kang2016sparse, kim2016bipartitions, kuhn2014proof, kuhn2016cycle, pokrovskiy2015highly, pokrovskiy2016edge}). Another main idea of the proof is called a \emph{sparse linkage structure}, which is introduced in~\cite{kang2016sparse} and will be discussed in Section~\ref{sec:prelim}. With some new ingredients, both ideas are extensively used in the proof of Theorem~\ref{thm:main}. 

The proof of Theorem~\ref{thm:main} is constructive so that there is a polynomial-time algorithm which, given a strongly $k$-connected digraph (strongly $k$-arc-connected directed multigraph) $D$ with $\overline{\Delta}(D) \leq \overline{\Delta}$, outputs a strongly $k$-connected (strongly $k$-arc-connected, respectively) spanning subgraph with at most $kn + 800k\overline{\Delta} + 800k^2$ ($kn + 670k\overline{\Delta} + 670k^2$, respectively) edges. Since every strongly $k$-arc-connected $n$-vertex directed multigraph has at least $kn$ edges, the algorithm approximates the minimum number of edges of a strongly $k$-connected (or strongly $k$-arc-connected) spanning subgraph of $G$ within an additive error $O(k(k+\overline{\Delta}))$.\\

{\bf Organization of the paper.} We introduce terminology and tools used in the proof in Section~\ref{sec:prelim}. We discuss a lower bound on the minimum number of edges in a strongly $k$-connected subgraph and a strongly $k$-arc-connected subgraph in Section~\ref{sec:lowerbdd}. We briefly sketch the proof of the main theorems in Section~\ref{sec:proofsketch}. Before the proof of the main results, we introduce some basic objects and notions for the construction of sparse highly connected subgraphs in Section~\ref{sec:object}. The main theorems are proved in Section~\ref{sec:proof}, and we discuss questions related to the main results in Section~\ref{sec:conclude}.

%%%%%%%%%%%%%%% need to mention other research related to tournaments / this topic. %%%%%%%%%%%%%%%%

\section{Preliminaries}\label{sec:prelim}
\subsection{Basic notions and lemmas} 
We begin with some basic definitions.
\begin{itemize}
\item[$(\rm 1)$] {\bf Sets and orderings}. For any integer $N \geq 0$, let $[N]$ denote the set $\left \{1 , \dots , N \right \}$ if $N \geq 1$, $\emptyset$ otherwise. For any $m$-element finite set $S = \left \{s_1 , \dots , s_m \right \}$, a \emph{linear ordering} $\sigma = (s_1 , \dots , s_m)$ is a map from $[m]$ to $S$ such that $\sigma(i) := s_i$ for $1 \leq i \leq m$. For two integers $p$ and $q$, $\sigma(p,q) := \sigma(\left \{p,\dots,q \right \} \cap [m])$ if $p \leq q$, and $\emptyset$ otherwise.

\item[$\rm (2)$] {\bf Directed graphs, Directed multigraphs, Oriented graphs}. A \emph{directed graph} (or \emph{digraph}) $D$ is a pair $(V,E)$ with a finite set $V$ of vertices and a set of $E$ edges in $(V \times V) \setminus \left \{(v,v) \: : \: v \in V \right \}$. A \emph{directed multigraph} $D$ is a pair $(V,E)$ with a finite set $V$ of vertices and a multiset $E$ of edges in $(V \times V) \setminus \left \{(v,v) \: : \: v \in V \right \}$. For simplicity, $uv$ denotes any edge $(u,v) \in E(D)$ for $u,v \in V(D)$. For two directed multigraphs $D_1=(V_1,E_1)$ and $D_2=(V_2,E_2)$, its union $D_1 \cup D_2$ is a directed multigraph $(V_1 \cup V_2 , E_1 \cup E_2)$. For a set $S \subseteq V(D)$, $D[S]$ denotes the subgraph of $D$ induced by $S$. An \emph{underlying graph} $UG(D)$ of a directed multigraph $D$ is a simple undirected graph obtained from $D$ by removing its orientation and multiple edges.

An \emph{oriented graph} is a digraph obtained from an undirected graph by orienting each edge. An oriented graph $G$ is \emph{transitive} if $uv,vw \in E(G)$ then $uv \in E(G)$. For a vertex $v \in D$, a set $N_D^+(v)$ is the set of \emph{out-neighbours} of $v$, and $N_D^-(v)$ is the set of \emph{in-neighbours} of $v$. A set $\delta_D^{+}(v)$ is the multiset of edges out of $v$, and $\delta_D^{-}(v)$ is the multiset of edges into $v$. Let $d_D^+(v) := |\delta_D^+(v)|$ and $d_D^-(v) := |\delta_D^-(v)|$ be \emph{out-degree} and \emph{in-degree} of $v$, respectively. Let $\delta^+(D)$ and $\delta^-(D)$ be the \emph{minimum out-degree} and the \emph{minimum in-degree} of any vertex in $D$, respectively. For two sets $X,Y \subseteq V(D)$, let $E_D(X,Y)$ be the multiset of edges from $X$ to $Y$, and $e_D(X,Y) := |E_D(X,Y)|$. A vertex $v \in V(D)$ is a \emph{source} if the in-degree of $v$ is 0, and a vertex $v$ is a \emph{sink} if the out-degree of $v$ is 0. A vertex $w$ is a \emph{non-neighbour} of $v$ if $w$ is neither $v$ nor an in-neighbour of $v$ nor an out-neighbour of $v$. Let \emph{$\overline{\Delta}(D)$} be the maximum number of non-neighbours of any vertex in $D$, equivalently, the maximum degree of the complement of $UG(D)$. A digraph or a directed multigraph $D$ is \emph{semicomplete} if $\overline{\Delta}(D) = 0$, and a semicomplete oriented graph is called a \emph{tournament}. We frequently use the following fact that $\overline{\Delta}(D') \leq \overline{\Delta}(D)$ for every induced subgraph $D'$ of a multigraph $D$.

For any integer $k \geq 1$, a directed multigraph $D$ is \emph{$k$-regular} if for every $v \in V(D)$, $d_D^{+}(v) = d_D^{-}(v) = k$. A set $A \subseteq V(D)$ \emph{in-dominates} a vertex $v \in V(D)$ if $v \in A$ or there exists $a \in A$ with $va \in E(D)$. A set $B \subseteq V(D)$ \emph{out-dominates} a vertex $u \in V(D)$ if $u \in B$ or there exists $b \in B$ with $bu \in E(D)$.

\item[$(\rm 3)$] {\bf Paths and fans.} A \emph{path} $P = (v_1 , \dots , v_s)$ is a digraph $P$ with the set $V(P) := \left \{v_1 , \dots , v_s \right \}$ of $s$ distinct vertices and the set $E(P) := \left \{v_i v_{i+1} \colon 1 \leq i \leq s-1 \right \}$ of edges. The set of \emph{endvertices} of $P$ is $\left \{v_1 , v_s \right \}$, and the set ${\rm Int}(P)$ of \emph{internal vertices} is $V(P) \setminus \left \{v_1 , v_s \right \}$. A path $P = (v_1 , \dots , v_s)$ in a directed multigraph $D$ is \emph{minimal} if $v_i v_j \notin E(D)$ for $2 \leq i+1 < j \leq s$.

Let $k \geq 1$ be an integer and $S \subseteq V(D)$. For a vertex $v \in V(D) \setminus S$, a \emph{$k$-fan} from $v$ to $S$ (from $S$ to $v$) is a collection of $k$ paths from $v$ to vertices in $S$ (from vertices in $S$ to $v$, respectively) such that each of them contains exactly one vertex in $S$, and any two of them have only the vertex $v$ in common. A \emph{$k$-arc-fan} from $v$ to $S$ (from $S$ to $v$) is a collection of $k$ paths from $v$ to vertices in $S$ (from vertices in $S$ to $v$, respectively) such that each of them contains exactly one vertex in $S$, and any two of them have no edge in common. 

\item[$(\rm 4)$] {\bf Connectivity}. A directed multigraph $D$ is \emph{strongly connected} if for every $u,v \in V(D)$, there is a path from $u$ to $v$. For any integer $k \geq 1$, a directed graph $D$ is \emph{strongly $k$-connected} if $|V(D)| \geq k+1$ and for every $S \subseteq V(D)$ of $|S| \leq k-1$, the directed multigraph $D-S$ is strongly connected. A directed multigraph $D$ is \emph{strongly $k$-arc-connected} if for every $T \subseteq E(D)$ with $|T| \leq k-1$, the directed multigraph $D-T$ remains strongly connected. A directed multigraph $D$ is \emph{minimally strongly $k$-connected} (\emph{minimally strongly $k$-arc-connected}) if $D$ is strongly $k$-connected (strongly $k$-arc-connected, respectively) and $D- \left \{e \right \}$ is not strongly $k$-connected (strongly $k$-arc-connected, respectively) for every $e \in E(D)$.
\end{itemize}

We often use the following well-known facts easily deduced from Menger's theorem.
\begin{PROP}\label{prop:fan}
Let $k \geq 1$ be an integer and $D$ be a directed multigraph and $\emptyset \ne S \subseteq V(D)$.
\begin{itemize}
\item[$(\rm 1)$] If $D$ is strongly $k$-connected and $|S| \geq k$, then for every $v \in V(D) \setminus S$, there are a $k$-fan from $v$ to $S$ and a $k$-fan from $S$ to $v$.
\item[$(\rm 2)$] If $D$ is strongly $k$-arc-connected, then for every $v \in V(D) \setminus S$, there are a $k$-arc-fan from $v$ to $S$ and a $k$-arc-fan from $S$ to $v$.
\item[$(\rm 3)$] If $D$ is strongly $k$-connected and $a_1 , \dots , a_k , b_1 , \dots , b_k \in V(D)$ are $2k$ distinct vertices of $D$, then there are $k$ vertex-disjoint paths $P_1 , \dots , P_k$ such that there is a permutation $\sigma : [k] \to [k]$ and for $i \in [k]$, $P_i$ is a path from $a_i$ to $b_{\sigma(i)}$.
\item[$(\rm 4)$] If $D$ is strongly $k$-arc-connected and $a_1 , \dots , a_k , b_1 , \dots , b_k \in V(D)$ are $2k$ distinct vertices of $D$, then there are $k$ edge-disjoint paths $P_1 , \dots , P_k$ such that there is a permutation $\sigma : [k] \to [k]$ and for $i \in [k]$, $P_i$ is a path from $a_i$ to $b_{\sigma(i)}$.
\end{itemize}
\end{PROP}

Now we prove the following elementary lemma, which extends~\cite[Lemma 2.1]{kang2016sparse} to dense directed multigraphs.
\begin{LEM}\label{lem:manydeg}
For integers $k \geq 1$, $n \geq 2$ , $\overline{\Delta} \geq 0$ with $n \geq k$, let $D$ be an $n$-vertex directed multigraph with $\overline{\Delta}(D) \leq \overline{\Delta}$. Then $D$ has $k$ vertices having at least $(n-k-\overline{\Delta})/2$ in-neighbours in $D$ and $k$ vertices having at least $(n-k-\overline{\Delta})/2$ out-neighbours in $D$.
\end{LEM}
\begin{proof}
Let $x_1 , \dots , x_k$ be $k$ vertices such that $|N_D^-(x_1)| \geq \dots \geq |N_D^-(x_k)|$ and $|N_D^-(x_i)| \geq |N_D^-(v)|$ for every $v \in V(D) \setminus \left \{x_1 , \dots , x_k \right \}$ and $1 \leq i \leq k$. Since $D' = D - \left \{x_1 , \dots , x_{k-1} \right \}$ contains $n-k+1$ vertices and $\overline{\Delta}(D') \leq \overline{\Delta}$,
$$\sum_{x \in V(D')}{|N_{D'}^{-}(x)|} = |E(D')| \geq |E(UG(D'))| \geq \frac{1}{2} |V(D')|(n-k-\overline{\Delta})$$ 
and there is $x \in V(D')$ such that $|N_{D'}^{-}(x)| \geq \frac{n-k-\overline{\Delta}}{2}$ since $|V(D')| \geq 1$. Therefore, for $1 \leq i \leq k$,
$$|N_{D}^-(x_i)| \geq |N_D^{-}(x_k)| \geq |N_{D}^{-}(x)| \geq |N_{D'}^{-}(x)| \geq \frac{n-k-\overline{\Delta}}{2}.$$

Similarly, there are $k$ vertices having at least $\frac{n-k-\overline{\Delta}}{2}$ out-neighbours in $D$.
\end{proof}

\subsection{Sparse linkage structures}
We need some notions introduced in~\cite[Section 3]{kang2016sparse}. For any $n$-vertex digraph $D$ and a linear ordering $\sigma = (v_1 , \dots , v_n)$ of $V(D)$, a digraph $D$ is \emph{$(\sigma, k, t)$-good} for positive integers $k$ and $t$, if the following hold.
\begin{itemize} 
\setlength{\itemindent}{.25in}
\item[$(a)$] If $v_i v_j \in E(D)$ for $1 \leq i,j \leq n$, then $i<j$.
\item[$(b)$] Every vertex $v_j$ for $1 \leq j \leq n-t$ has out-degree at least $k$ in $D$.
\item[$(c)$] Every vertex $v_j$ for $t+1 \leq j \leq n$ has in-degree at least $k$ in $D$.
\end{itemize}

The following lemma easily follows from the definition of $(\sigma,k,t)$-good digraphs. Note that (1) of the lemma follows by~\cite[Claim 3.1]{kang2016sparse}, and (2) is easily deduced from (1).

\begin{LEM}\label{lem:link}
For integers $n \geq 1$, $t \geq k \geq 1$ and a $(\sigma, k, t)$-good $n$-vertex digraph $D$, the following hold.
\begin{itemize} 
\item [$(\rm 1)$] Let $S \subseteq V(D)$ be a set of at most $k-1$ vertices. For every $u \in V(D)\setminus S$, there are vertices $v \in \sigma(1,t)$ and $w \in \sigma(n-t+1,n)$ such that $D-S$ contains a path from $v$ to $u$ and a path from $u$ to $w$.

\item [$(\rm 2)$] Let $F \subseteq E(D)$ be a set of at most $k-1$ edges. For every $u \in V(D)$, there are vertices $v \in \sigma(1,t)$ and $w \in \sigma(n-t+1,n)$ such that $D-F$ contains a path from $v$ to $u$ and a path from $u$ to $w$.
\end{itemize}
\end{LEM}

The following proposition, the heart of the proof of Theorem~\ref{thm:main}, asserts that if $D$ is dense, then we can always find a sparse linkage structure (see~\cite[Lemma 3.4]{kang2016sparse}).

\begin{PROP}[Kang, Kim, Kim, and Suh~\cite{kang2016sparse}]\label{prop:order}
For integers $k,n \geq 1$ and $\overline{\Delta} \geq 0$, let $D$ be an $n$-vertex directed multigraph with $\overline{\Delta}(D) \leq \overline{\Delta}$. There is a linear ordering $\sigma$ of $V(D)$ and a $(\sigma, k, 2k+\overline{\Delta}-1)$-good digraph $D'$, where $D'$ is a spanning subgraph of $D$ with at most $kn - k + k\overline{\Delta}$ edges.
\end{PROP}
%\begin{proof}
%Proposition~\cite[Lemma 3.4]{kang2016sparse} considers oriented graphs $D$ with $\overline{\Delta}(D) \leq \overline{\Delta}$, but it can be extended to digraphs $D$ with $\overline{\Delta}(D) \leq \overline{\Delta}$ by removing edges. By Proposition~\cite[Lemma 3.4]{kang2016sparse}, there are a linear ordering $\sigma$ of $V(D)$ and a $(\sigma,k, 2k+\overline{\Delta}-1)$-good spanning subgraph $D'$ of $D$ that satisfy the proposition. Now we analyze the running time of finding both $\sigma$ and $D'$. The proof of Proposition~\cite[Lemma 3.4]{kang2016sparse} uses~\cite[Claim 3.3]{kang2016sparse} which runs in $O(n^3)$, and find $k$ matchings on an auxiliary bipartite graph $k$ times, where each matching can be found in $O(n^{2.5})$ by~\cite{hopcroft1973}. The rest of the proof runs in $O(n^2)$. In total, both $\sigma$ and $D'$ can be found in $O(n^3 + kn^{2.5})$ as desired.
%\end{proof}

Indeed, the proof of~\cite[Lemma 3.4]{kang2016sparse} yields a polynomial-time algorithm that outputs $D'$ in time $O(n^3 + kn^{2.5})$ using the algorithm of Hopcroft and Karp~\cite{hopcroft1973} that finds a maximum matching in a bipartite graph.

We also need the following applications of Lemma~\ref{lem:link} and Proposition~\ref{prop:order}.

\begin{LEM}\label{lem:digraph1'}
For integers $k,n \geq 1$ and $\overline{\Delta} \geq 0$, let $D$ be a digraph with $\overline{\Delta}(D) \leq \overline{\Delta}$. Let $U$ be a nonempty subset of $V(D)$. Then there are a spanning subgraph $D'$ of $D[U]$, and subsets $U_i, U_o \subseteq U$ satisfying the following.
\begin{itemize} 
\item[$(1)$] $|E(D')| \leq k|U| - k + k \overline{\Delta}$.
\item[$(2)$] $|U_i|,|U_o| \leq 2k+\overline{\Delta}-1$.
\item[$(3)$] For every $S \subseteq V(D)$ with $|S| \leq k-1$ and for every $u,v \in U \setminus S$, the digraph $D'-S$ has a path from $u$ to a vertex in $U_o \setminus S$, and a path from a vertex in $U_i \setminus S$ to $v$.
\end{itemize}
%Moreover, the running time is $O(n^3 + kn^{2.5})$.
\end{LEM}
\begin{proof}
The proof is immediate from Lemma~\ref{lem:link} and Proposition~\ref{prop:order}.
\end{proof}

\begin{LEM}\label{lem:digraph2'}
For integers $k,n \geq 1$ and $\overline{\Delta} \geq 0$, let $D$ be a digraph with $\overline{\Delta}(D) \leq \overline{\Delta}$, and $\left \{P_1 , \dots , P_k \right \}$ be a collection of $k$ vertex-disjoint minimal paths in $D$ such that $P_i$ is a path from $a_i \in V(D)$ to $b_i \in V(D)$.

For every nonempty $U \subseteq \bigcup_{i=1}^{k}{{\rm Int}(P_i)}$, there are a spanning subgraph $D'$ of $D[U] - \bigcup_{i=1}^{k}{E(P_i)}$, and subsets $U_i, U_o \subseteq U$ satisfying the following.
\begin{itemize} 
\item[$(1)$] $|E(D')| \leq (k-1)|U| + (k-1)(\overline{\Delta} + 1)$.
\item[$(2)$] $|U_i|,|U_o| \leq 2k+\overline{\Delta}-1$.
\item[$(3)$] For every $S \subseteq V(D)$ with $|S| \leq k-1$ and for every $u,v \in U \setminus S$, the subgraph $D-S$ has a path from $u$ to a vertex in $(U_o \cup \left \{b_1 , \dots , b_k \right \}) \setminus S$ using only edges in $E(D') \cup \bigcup_{i=1}^{k}E(P_i)$, and a path from a vertex in $(U_i \cup \left \{a_1 , \dots , a_k \right \})\setminus S$ to $v$ only using edges in $E(D') \cup \bigcup_{i=1}^{k}E(P_i)$.
\end{itemize}
%Moreover, the running time is $O(n^3 + kn^{2.5})$.
\end{LEM}
\begin{proof}
Let $E_\textrm{path} := \bigcup_{i=1}^{k}E(P_i)$. Since $\overline{\Delta}(D) \leq \overline{\Delta}$ and every vertex intersects at most one path in $\left \{P_1 , \dots , P_k \right \}$, we have $\overline{\Delta}(D[U] - E_{\rm path}) \leq \overline{\Delta} + 2$. By Proposition~\ref{prop:order}, there are a linear ordering $\sigma$ of $U$ and a $(\sigma , k-1 , 2k+\overline{\Delta}-1)$-good spanning subgraph $D'$ of $D[U] - E_\textrm{path}$ that satisfies (1). Let $U_i := \sigma(1,2k+\overline{\Delta}-1)$ and $U_o := \sigma(|U|-2k-\overline{\Delta}+2,|U|)$. Then $|U_i|,|U_o| \leq 2k+\overline{\Delta}-1$, satisfying (2).

Now it remains to prove (3). Let $S \subseteq V(D)$ with $|S| \leq k-1$ and $u \in U \setminus S$. We aim to prove that there is a path $P$ in $D-S$ from $u$ to a vertex in $(U_o \cup \left \{b_1 , \dots , b_k \right \}) \setminus S$ with $E(P) \subseteq E(D') \cup E_\textrm{path}$. Let us write $\sigma = (v_1 , \dots , v_{|U|})$ and $i$ be the maximum index such that $u$ can reach to $v_i$ by a directed path in $D'-S$. 

If $i \geq |U|-2k-\overline{\Delta}+2$, then $v_i \in U_o$. Let $P$ be a directed path in $D'-S$ from $u$ to $v_i$ and we are done. We may assume that $i \leq |U|-2k-\overline{\Delta}+1$. By the maximality of $i$, we have $S = N_{D'}^+ (v_i)$ since every vertex in $\sigma(1,|U|-2k-\overline{\Delta}+1)$ has out-degree at least $k-1$ in $D'$ and $|S| \leq k-1$. From the definition of $U$, there is $t \in [k]$ such that $v_i \in V^\textrm{int}(P_t)$, where $P_t$ is a minimal path in $D$ from $a_t$ to $b_t$. Let $Q$ be the subpath of $P_t$ from $v_i$ to $b_t$, and $w_i$ be the out-neighbour of $v_i$ in $Q$. Since $P_t$ is a minimal path in $D$, we have $V(Q) \cap N_D^{+}(v_i) = \left \{ w_i \right \}$. Hence it follows that $V(Q) \cap N_{D'}^+(v_i) = V(Q) \cap S = \emptyset$ since $E(D') \cap E(Q) \subseteq E(D') \cap E_\textrm{path} = \emptyset$. Therefore, there is a path $P$ in $D-S$ from $u$ to $b_t$ with $E(P) \subseteq E(D') \cup E_\textrm{path}$, as desired. Similarly, for every $v \in U \setminus S$, there is a path $P'$ in $D-S$ from a vertex in $U_i \cup \left \{a_1 , \dots , a_k \right \}$ to $v$ with $E(P') \subseteq E(D') \cup E_\textrm{path}$. %A digraph $D'$ and the subsets $U_i,U_o \subseteq U$ can be found in $O(n^3 + kn^{2.5})$ by Proposition~\ref{prop:order}.
\end{proof}

Both Lemmas~\ref{lem:digraph1'} and~\ref{lem:digraph2'} have the following variations with the identical proofs. When applying Proposition~\ref{prop:order}, we may assume that $D$ is a digraph by removing multiple edges.

\begin{LEM}\label{lem:digraph3'}
For integers $k,n \geq 1$ and $\overline{\Delta} \geq 0$, let $D$ be a directed multigraph with $\overline{\Delta}(D) \leq \overline{\Delta}$. Let $U$ be a nonempty subset of $V(D)$. Then there are a spanning subgraph $D'$ of $D[U]$, and subsets $U_i, U_o \subseteq U$ satisfying the following.
\begin{itemize} 
\item[$(1)$] $|E(D')| \leq k|U| - k + k \overline{\Delta}$.
\item[$(2)$] $|U_i|,|U_o| \leq 2k+\overline{\Delta}-1$.
\item[$(3)$] For every $F \subseteq E(D)$ with $|F| \leq k-1$ and for every $u,v \in U$, the digraph $D'-F$ has a path from $u$ to a vertex in $U_o$, and a path from a vertex in $U_i$ to $v$.
\end{itemize}
%Moreover, the running time is $O(n^3 + kn^{2.5})$.
\end{LEM}

\begin{LEM}\label{lem:digraph4'}
For integers $k,n \geq 1$ and $\overline{\Delta} \geq 0$, let $D$ be a directed multigraph with $\overline{\Delta}(D) \leq \overline{\Delta}$ and $\left \{P_1 , \dots , P_k \right \}$ be a collection of $k$ edge-disjoint paths in $D$ such that for $i \in [k]$,  $P_i$ is a path from $a_i \in V(D)$ to $b_i \in V(D)$. 

For every nonempty $U \subseteq \bigcup_{i=1}^{k}{{\rm Int}(P_i)}$, there are a spanning subgraph $D'$ of $D[U] - \bigcup_{i=1}^{k}{E(P_i)}$, and subsets $U_i, U_o \subseteq U$ satisfying the following.
\begin{itemize} 
\item[$(1)$] $|E(D')| \leq (k-1)|U| + (k-1)(\overline{\Delta} + 2k - 1)$.
\item[$(2)$] $|U_i|,|U_o| \leq 4k + \overline{\Delta} - 3$.
\item[$(3)$] For every $F \subseteq E(D)$ with $|F| \leq k-1$ and for every $u,v \in U$, a subgraph $D-F$ has a path from $u$ to a vertex in $U_o \cup \left \{b_1 , \dots , b_k \right \}$ using only edges in $E(D') \cup \bigcup_{i=1}^{k}E(P_i)$, and a path from a vertex in $U_i \cup \left \{a_1 , \dots , a_k \right \}$ to $v$ using only edges in $E(D') \cup \bigcup_{i=1}^{k}E(P_i)$.
\end{itemize}
\end{LEM}
\begin{proof}
Let $E_\textrm{path} := \bigcup_{i=1}^{k}E(P_i)$. Since $P_i$ intersects every vertex at most once for $i \in [k]$ and $\overline{\Delta}(D) \leq \overline{\Delta}$, we have $\overline{\Delta}(D[U] - E_{\rm path}) \leq \overline{\Delta} + 2k$. By Proposition~\ref{prop:order}, there are a linear ordering $\sigma$ of $U$ and a $(\sigma , k-1 , 4k+\overline{\Delta}-3)$-good digraph $D'$, where $D'$ is a spanning subgraph of $D[U] - E_\textrm{path}$ that satisfies (1). Let $U_i := \sigma(1,4k+\overline{\Delta}-3)$ and $U_o := \sigma(|U|-4k-\overline{\Delta}+4,|U|)$. Then $|U_i|,|U_o| \leq 4k+\overline{\Delta}-3$, satisfying (2).

Now it remains to prove (3). Let $F \subseteq E(D)$ with $|F| \leq k-1$ and $u \in U$. We aim to prove that there is a path $P$ in $D-F$ from $u$ to a vertex in $U_o \cup \left \{b_1 , \dots , b_k \right \}$ with $E(P) \subseteq E(D') \cup E_\textrm{path}$. Let us write $\sigma = (v_1 , \dots , v_{|U|})$ and $i$ be the maximum index such that $u$ can reach to $v_i$ by a directed path in $D'-F$. 

If $i \geq |U|-4k-\overline{\Delta}+4$, then $v_i \in U_o$. Let $P$ be a directed path in $D'-F$ from $u$ to $v_i$ and we are done. We may assume that $i \leq |U|-4k-\overline{\Delta}+3$. By the maximality of $i$, we have $F = \delta_{D'}^+ (v_i)$ since every vertex in $\sigma(1,|U|-4k-\overline{\Delta}+4)$ has out-degree at least $k-1$ in $D'$ and $|F| \leq k-1$. From the definition of $U$, there is $t \in [k]$ such that $v_i \in V^\textrm{int}(P_t)$, where $P_t$ is a path in $D$ from $a_t$ to $b_t$. Let $Q$ be a subpath of $P_t$ from $v_i$ to $b_t$. Since $E(D') \cap E_{\rm path} = \emptyset$, it follows that $E(Q) \cap F = \emptyset$. Therefore, there is a path $P$ in $D-F$ from $u$ to $b_t$ with $E(P) \subseteq E(D') \cup E_\textrm{path}$, as desired. Similarly, for every $v \in U$, there is a path $P'$ in $D-F$ from a vertex in $U_i \cup \left \{a_1 , \dots , a_k \right \}$ to $v$ with $E(P') \subseteq E(D') \cup E_\textrm{path}$.
\end{proof}

\subsection{Minimally strongly $k$-connected digraphs}

For any undirected graph $G$, a subgraph $C = (v_1 , \dots , v_t)$ is a \emph{circuit} in $G$ if $v_1 , \dots , v_t \in V(G)$ and $v_i v_{i+1} \in E(G)$ for $1 \leq i \leq t$, where we define $v_{t+1} = v_1$ and these $t$ edges are distinct. Note that the vertices $v_1 , \dots , v_t$ are not necessarily distinct, and we regard a circuit $C$ as a subgraph of $G$, such that $V(C) := \left \{v_1 , \dots , v_t \right \}$ and $E(C) := \left \{v_i v_{i+1} \: : \: 1 \leq i \leq t \right \}$.

For a digraph $D$, a subgraph $C = (v_1 , \dots , v_{2m})$ is an \emph{anti-directed trail} in $D$ if $v_1 , \dots , v_{2m} \in V(D)$, $v_{2i-1} v_{2i} \in E(D)$ and $v_{2i+1} v_{2i} \in E(D)$ for $1 \leq i \leq m$, where we define $v_{2m+1} = v_1$ and these $2m$ edges are distinct. Note that the vertices $v_1 , \dots , v_{2m}$ are not necessarily distinct, and we regard an anti-directed trail $C$ as a subgraph of $D$, such that $V(C) := \left \{v_1 , \dots , v_{2m} \right \}$ and $E(C) := \bigcup_{i=1}^{m} \left \{v_{2i-1}v_{2i} , v_{2i+1}v_{2i} \right \}$.

For a digraph $D = (V,E)$, let $V' := \left \{v' \: : \: v \in V \right \}$ and $V'' := \left \{v'' \: : \: v \in V \right \}$ be two disjoint copies of $V$. A \emph{bipartite representation} $BG(D)$ of $D$ be an undirected bipartite graph with $V(BG(D)) := V' \cup V''$ and $E(BG(D)) := \left \{ \left \{x',y''\right \} \: : \: (x,y) \in E(D) \right \}$.

It is easy to see that a subgraph $D'$ of $D$ is an anti-directed trail if and only if its bipartite representation $BG(D')$ is a circuit in $BG(D)$. Therefore, $D$ has no anti-directed trail then 
$$|E(D)| = |E(BG(D))| \leq |V(BG(D))|-1 = 2|V(D)|-1,$$ 
since $BG(D)$ is a forest. This proves the following proposition (see~\cite[Lemma 2]{mader1985}) that characterizes digraphs without anti-directed trails. 

\begin{PROP}\label{prop:forest}
A digraph $D$ does not contain an anti-directed trail if and only if $BG(D)$ is a forest. In particular, $|E(D)| \leq 2|V(D)|-1$ if $D$ has no anti-directed trail.
\end{PROP}

For a directed multigraph $D=(V,E)$ and a vertex $u \in V$, a spanning subgraph $T$ is an \emph{out-branching} (\emph{in-branching}) of $D$ rooted at $u$ if $T$ is an oriented graph obtained from a tree by orienting edges and $u$ is the only vertex with in-degree (out-degree, respectively) zero in $T$. We make the use of the following theorem (see~\cite{edmonds1973} or~\cite[Theorem 9.3.1]{bang2008digraphs}).

\begin{THM}[Edmonds~\cite{edmonds1973}]\label{thm:branching}
Let $D=(V,E)$ be a directed multigraph with a vertex $u \in V(D)$. Then the following hold.
\begin{itemize} 
\item[$(\rm 1)$] $D$ contains $k$ edge-disjoint out-branchings rooted at $u$ if and only if for every $\emptyset \ne S \subseteq V(D) \setminus \left \{u \right \}$, $e_D (V(D) \setminus S , S) \geq k$. 
\item[$(\rm 2)$] $D$ contains $k$ edge-disjoint in-branchings rooted at $u$ if and only if for every $\emptyset \ne S \subseteq V(D) \setminus \left \{u \right \}$, $e_D (S , V(D) \setminus S) \geq k$.
\end{itemize}
\end{THM}

Theorem~\ref{thm:branching} has the following corollary, which extends the result of Dalmazzo~\cite{dalmazzo1977} that every strongly $k$-arc-connected $n$-vertex directed multigraph contains a strongly $k$-arc-connected subgraph with at most $2k(n-1)$ edges (see~\cite[Theorem 5.6.1]{bang2008digraphs}).

\begin{COR}\label{cor:minimal_arc}
Let $k \geq 1$ be an integer and $D$ be a minimally strongly $k$-arc-connected directed multigraph and $\emptyset \ne U \subseteq V(D)$. Then $|E(D[U])| \leq 2k(|U|-1)$.
\end{COR}
\begin{proof}
Fix a vertex $u \in U$. By Theorem~\ref{thm:branching}, there are $k$ edge-disjoint out-branchings $T_1^+ , \dots , T_{k}^+$ rooted at $u$, and $k$ edge-disjoint in-branchings $T_1^- , \dots , T_{k}^-$ rooted at $u$. Since $\bigcup_{i=1}^{k} T_i^+ \cup \bigcup_{i=1}^{k} T_i^-$ is a strongly $k$-arc-connected spanning subgraph of $D$, we have $D = \bigcup_{i=1}^{k} T_i^+ \cup \bigcup_{i=1}^{k} T_i^-$. As $|E(T_i^+ [U])| \leq |U|-1$ and $|E(T_i^- [U])| \leq |U|-1$ for every $i \in [k]$, we have
$$|E(D[U])| \leq \sum_{i=1}^{k} |E(T_i^+ [U])| + \sum_{i=1}^{k} |E(T_i^- [U])| \leq 2k(|U|-1)$$
as desired.
\end{proof}

We use the following theorem by Mader (see~\cite{mader1996} or~\cite[Corollary 5.6.20]{bang2008digraphs}).

\begin{THM}[Mader~\cite{mader1996}]\label{thm:mader_minimal}
For any integer $k \geq 2$ and a minimally strongly $k$-connected digraph $D = (V,E)$, let $D' = (V,E')$ be a strongly $(k-1)$-connected spanning subgraph of $D$. Then the digraph $(V,E \setminus E')$ contains no anti-directed trail.
\end{THM}

The following proposition proves that, if a digraph $D$ is minimally strongly $k$-connected, then for any $U \subseteq V(D)$, the induced subgraph $D[U]$ contains only few edges. This also proves that every strongly $k$-connected digraph $D$ contains a strongly $k$-connected spanning subgraph with at most $2k|V(D)|$ edges, which is slightly weaker than the result of Mader~\cite{mader1985}.

\begin{PROP}\label{prop:fewedges}
For any integer $k \geq 1$, let $D$ be a minimally strongly $k$-connected digraph and $\emptyset \ne U \subseteq V(D)$. Then $|E(D[U])| \leq 2k|U| - k - 1$.
\end{PROP}
\begin{proof}
We prove by induction on $k$. If $k=1$, the proposition follows from Corollary~\ref{cor:minimal_arc}, as $D$ is minimally strongly 1-arc-connected. Now we may assume that $k \geq 2$. Let $D'$ be a minimally strongly $(k-1)$-connected spanning subgraph of $D$. By the induction hypothesis, $|E(D'[U])| \leq 2(k-1)|U| - k$.

By Theorem~\ref{thm:mader_minimal}, the digraph $D'' := (V,E \setminus E')$ has no anti-directed trail by Theorem~\ref{thm:mader_minimal}. As its induced subgraph $D''[U]$ also has no anti-directed trail, it has at most $2|U|-1$ edges by Proposition~\ref{prop:forest}. Hence
$$|E(D[U])| = |E(D'[U])| + |E(D''[U])| \leq (2(k-1)|U| - k) + (2|U|-1) \leq 2k|U| - k - 1.$$
as desired.
\end{proof}

\section{Lower bounds}\label{sec:lowerbdd}
Inspired by the construction of $\mathcal{T}_{n, k}$ in~\cite[Section 2]{bang2004spanning}, we define a strongly $k$-connected $(n_1 + n_2 + \overline{\Delta} + 1)$-vertex oriented graph $G_{n_1 , n_2 , k, \overline{\Delta}}$ for integers $n_1 , n_2 \geq 2k+1$ as follows. Let $G_1$ be an $(\overline{\Delta}+1)$-vertex digraph with no edges. Let $T_2$ be an $n_1$-vertex tournament obtained from an $\lfloor \frac{n_1 - 1}{2} \rfloor$-th power\footnote{A $k$th power of a digraph $D$ is a digraph that has the vertex-set $V(D)$ and $(u,v)\in E(D)$ when the distance from $u$ to $v$ is at most $k$ in $D$.} of a directed cycle of length $n_1$ by adding arbitrary edges to ensure that $T_2$ is a tournament. Since $\lfloor \frac{n_1 - 1}{2} \rfloor \geq k$, the tournament $T_2$ is strongly $k$-connected and $\delta^+(T_2),\delta^-(T_2) \geq \lfloor \frac{n_1 - 1}{2} \rfloor$. Similarly, let $T_3$ be an $n_2$-vertex tournament obtained from an $\lfloor \frac{n_2 - 1}{2} \rfloor$-th power of a directed cycle of length $n_2$ by adding arbitrary edges. Since $\lfloor \frac{n_2 - 1}{2} \rfloor \geq k$, the tournament $T_3$ is strongly $k$-connected and $\delta^+(T_3),\delta^-(T_3) \geq \lfloor \frac{n_2 - 1}{2} \rfloor$. We may assume that $V(G_1), V(T_2)$, and $V(T_3)$ are disjoint. Let $a_1 , \dots , a_k \in V(T_2)$ and $b_1 , \dots , b_k \in V(T_3)$ be $2k$ distinct vertices and define
\begin{align*}
V(G_{n_1 , n_2 , k,\overline{\Delta}}) &:= V(G_1) \cup V(T_2) \cup V(T_3)\\
E(G_{n_1 , n_2 , k,\overline{\Delta}}) &:= (V(G_1) \times V(T_3)) \cup (V(T_2) \times V(G_1)) \\
&\qquad\cup ((V(T_2) \times V(T_3)) \setminus \left \{a_i b_i \: : \: 1 \leq i \leq k \right \}) \cup \left \{b_i a_i \: : \: 1 \leq i \leq k \right \}
\end{align*}

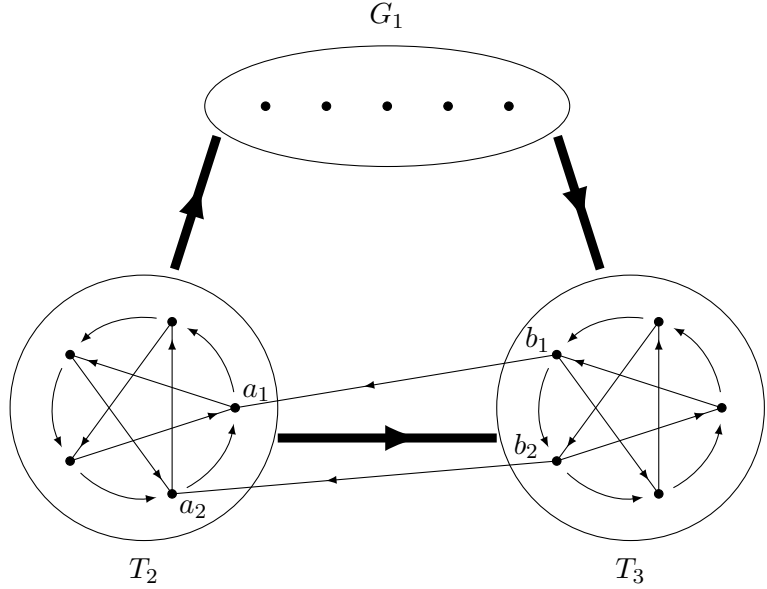
\begin{figure}[h]
\centering
\begin{tikzpicture}[scale=0.8]

\draw[fill=none] (4,0) ellipse [x radius=3,y radius=1];
\node at (4,1.5) {$G_1$};
\draw[fill=none] (0,-5) ellipse [x radius=2.2,y radius=2.2];
\node at (0,-7.7) {$T_2$};
 \draw[fill=none] (8,-5) ellipse [x radius=2.2,y radius=2.2];
\node at (8,-7.7) {$T_3$};

%% G_1 %%
\filldraw[fill=black] (2,0.0) circle (2pt);
\filldraw[fill=black] (3,0.0) circle (2pt);
\filldraw[fill=black] (4,0.0) circle (2pt);
\filldraw[fill=black] (5,0.0) circle (2pt);
\filldraw[fill=black] (6,0.0) circle (2pt);

%% transitive tournament for A_1 %%
\def \n {5}
\def \radius {1.5cm}
\def \margin {10} % margin in angles, depends on the radius
\foreach \s in {1,...,\n}
{
 %\pgfmathtruncatemacro{\x}{\cos(360/\n * (\s - 1)};
 %\pgfmathtruncatemacro{\y}{\sin(360/\n * (\s - 1)};
  \filldraw[fill=black] (0,-5) + ({360/\n * (\s - 1)}:\radius)  circle (2pt);
  \draw[->, >=latex] (0,-5) + ({360/\n * (\s - 1)+\margin}:\radius)
    arc ({360/\n * (\s - 1)+\margin}:{360/\n * (\s)-\margin}:\radius); 
  \draw[middleupuparrow={latex}, yshift=-5cm] ({360/\n * (\s - 1)}:\radius) -- ({360/\n * (\s + 1)}:\radius) ;
}

\node[xshift=0.28cm, yshift=-3.8cm] at ({360/\n * 0}:\radius) {$a_1$};
\node[xshift=0.28cm, yshift=-4.2cm] at ({360/\n * 4}:\radius) {$a_2$};

\foreach \s in {1,...,\n}
{
 %\pgfmathtruncatemacro{\x}{\cos(360/\n * (\s - 1)};
 %\pgfmathtruncatemacro{\y}{\sin(360/\n * (\s - 1)};
  \filldraw[fill=black] (8,-5) + ({360/\n * (\s - 1)}:\radius)  circle (2pt);
  \draw[->, >=latex] (8,-5) + ({360/\n * (\s - 1)+\margin}:\radius)
    arc ({360/\n * (\s - 1)+\margin}:{360/\n * (\s)-\margin}:\radius); 
  \draw[middleupuparrow={latex}, xshift=8cm, yshift=-5cm] ({360/\n * (\s - 1)}:\radius) -- ({360/\n * (\s + 1)}:\radius) ;
}

\node[xshift=6.15cm, yshift=-3.8cm] at ({360/\n * 2}:\radius) {$b_1$};
\node[xshift=6.00cm, yshift=-3.8cm] at ({360/\n * 3}:\radius) {$b_2$};

\draw[middlearrow={latex}, xshift=8cm, yshift=-5cm] ({360/\n * 2}:\radius) -- ($(-8,0)+({360/\n * 0}:\radius)$);
\draw[middlearrow={latex}, xshift=8cm, yshift=-5cm] ({360/\n * 3}:\radius) -- ($(-8,0)+({360/\n * 4}:\radius)$);

\draw[middleuparrow={latex}, line width=1.2mm] (0.5,-2.7) -- (1.2,-0.5);
\draw[middleuparrow={latex}, line width=1.2mm] (6.8,-0.5) -- (7.5,-2.7);
\draw[middleuparrow={latex}, line width=1.2mm] (2.2,-5.5) -- (5.8,-5.5);
\end{tikzpicture}
\caption{The oriented graph $G_{5,5,2,4}$. }
\end{figure}

Note that $G_{n_1 , n_2 , k , \overline{\Delta}}$ has the following properties.
\begin{itemize} 
\item $G_{n_1 , n_2 , k, \overline{\Delta}}$ is strongly $k$-connected. 
\item $\overline{\Delta}(G_{n_1 , n_2 , k, \overline{\Delta}}) \leq \overline{\Delta}$.
\item The minimum in-degree and the minimum out-degree are at least $\min (\lfloor \frac{n_1 - 1}{2} \rfloor , \lfloor \frac{n_2 - 1}{2} \rfloor)$.
\end{itemize}

If $n = n_1 + n_2 + \overline{\Delta} + 1$ and $|n_1 - n_2| \leq 1$, then $\min(n_1 , n_2) \geq \frac{n-\overline{\Delta}-2}{2}$ and $\min \left (\lfloor \frac{n_1 - 1}{2} \rfloor , \lfloor \frac{n_2 - 1}{2} \rfloor \right ) \geq \lfloor \frac{n - \overline{\Delta}}{4} \rfloor - 1$.

Let $D$ be a spanning subgraph of $G_{n_1 , n_2 , k,\overline{\Delta}}$ with $\delta^{+}(D),\delta^{-}(D) \geq k$. Since every vertex in $G_1$ has in-degree at least $k$ in $D$,
\begin{align*}
\sum_{v \in V(T_2)}d_D^+(v) - \sum_{w \in V(T_2)}d_D^-(w) & \geq e_D(V(T_2),V(G_1)) - e_D(V(T_3),V(T_2)) \geq k(\overline{\Delta}+1) - k
\end{align*}
and thus $\sum_{v \in V(T_2)}d_D^+(v) \geq \sum_{w \in V(T_2)}d_D^-(w) + k \overline{\Delta} \geq kn_1 + k\overline{\Delta}$. Hence
%$$\sum_{v \in V(T_2)}d_D^+(v) = e_D (V(T_2),V(G_1)) - e_D(V(T_3),V(T_2)) + \sum_{w \in V(T_2)}d_D^-(w) \geq k(\overline{\Delta}+1) - k + kn_1$$
\begin{align*}
|E(D)| & = \sum_{u \in V(G_1)}d_D^{+}(u) + \sum_{v \in V(T_2)}d_D^{+}(v) + \sum_{w \in V(T_3)}d_D^{+}(w) \\
& \geq k|V(G_1)| + \sum_{v \in V(T_2)}d_D^{+}(v) + k|V(T_3)| \geq k(n_1 + n_2 + \overline{\Delta} + 1) + k\overline{\Delta}.
\end{align*}

Let us define $T_{n_1 , n_2 , k}$ be an $(n_1 + n_2 + k)$-vertex tournament obtained from an $(n_1 + n_2 + k)$-vertex oriented graph $G_{n_1 , n_2 , k , k-1}$ by replacing $G_1$ with a $k$-vertex transitive tournament $T_1$. Note that $T_{n_1 , n_2 , k}$ has the following properties.
\begin{itemize} 
\item $T_{n_1 , n_2 , k}$ is strongly $k$-connected. 
\item The minimum in-degree and the minimum out-degree are at least $\min (\lfloor \frac{n_1 - 1}{2} \rfloor , \lfloor \frac{n_2 - 1}{2} \rfloor)$.
\end{itemize}
Let $D$ be a spanning subgraph of $T_{n_1 , n_2 , k}$ with $\delta^+ (D), \delta^- (D) \geq k$. Let $\sigma = (v_1 , \dots , v_t)$ be a transitive ordering of the transitive tournament $T_1$. Since $d_D^-(v_i) \geq k$ for $1 \leq i \leq k$, we have $e_D (V(T_2) , v_i) + e_D(V(T_1) , v_i) \geq k$. In particular, $e_D(V(T_2) , v_i) \geq k - i + 1$, and thus $e_D (V(T_2) , V(T_1)) \geq \sum_{i=1}^{k}(k-i+1) = \frac{k(k+1)}{2}.$ Hence
$$\sum_{v \in V(T_2)}d_D^+(v) - \sum_{w \in V(T_2)}d_D^-(w) \geq e_D (V(T_2),V(T_1)) - e_D(V(T_3),V(T_2))  \geq \frac{k(k+1)}{2} - k$$
and thus 
\begin{align*}
|E(D)| & = \sum_{u \in V(G_1)}d_D^{+}(u) + \sum_{v \in V(T_2)}d_D^{+}(v) + \sum_{w \in V(T_3)}d_D^{+}(w) \\
& \geq k|V(G_1)| + \sum_{v \in V(T_2)}d_D^{+}(v) + k|V(T_3)| \geq k(n_1 + n_2 + k) + \frac{k(k-1)}{2}.
\end{align*}

If $n = n_1 + n_2 + k$ and $|n_1 - n_2| \leq 1$, then $\min(n_1 , n_2) \geq \frac{n-k-1}{2}$ and $\min \left (\lfloor \frac{n_1 - 1}{2} \rfloor , \lfloor \frac{n_2 - 1}{2} \rfloor \right ) \geq \lfloor \frac{n - k - 3}{4} \rfloor$.

The construction above proves the following proposition.

\begin{PROP}\label{prop:lowerbdd}
Let $k \geq 1$ and $\overline{\Delta} \geq 0$ be integers.
\begin{itemize} 
\item[$(\rm 1)$] For any integer $n \geq 4k+\overline{\Delta}+3$, there is a strongly $k$-connected $n$-vertex oriented graph $G$ with $\overline{\Delta}(G) \leq \overline{\Delta}$ and $\delta^+(G),\delta^-(G) \geq \lfloor \frac{n-\overline{\Delta}}{4} \rfloor - 1$, such that every spanning subgraph $D$ with $\delta^+(D),\delta^-(D) \geq k$ contains at least $kn + k\overline{\Delta}$ edges.

\item[$(\rm 2)$] For any integer $n \geq 5k+2$, there is a strongly $k$-connected $n$-vertex tournament $T$ with $\delta^+(T),\delta^-(T) \geq \lfloor \frac{n-k-3}{4} \rfloor $, such that every spanning subgraph $D$ with $\delta^+(D),\delta^-(D) \geq k$ contains at least $kn + \frac{k(k-1)}{2}$ edges.
\end{itemize}
\end{PROP}

%\begin{REM} 
%Proposition~\ref{prop:lowerbdd} provides a lower bound on the minimum number of edges in a strongly $k$-connected (strongly $k$-arc-connected) spanning subgraph of almost semicomplete digraphs (directed multigraphs, respectively). Indeed, Proposition~\ref{prop:lowerbdd} gives more information on a variation of Theorem~\ref{thm:main}: even though we relax the assumptions in Theorem~\ref{thm:main} to oriented graphs and for example, either we only consider spanning subgraphs with minimum in/out-degree at least $k$, or we additionally assume that both  the minimum in-degree and the minimum out-degree of $D$ are large, $kn + \Omega(k(k+\overline{\Delta}))$ edges are necessary for spanning subgraphs in general.
%\end{REM}

\section{Brief idea	 of the proof of Theorem~\ref{thm:main}}\label{sec:proofsketch}
Before introducing tools used in the proof, we illustrate the brief idea of the proof of $(\rm 1)$ of Theorem~\ref{thm:main} for $\overline\Delta = 0$, where the given digraph $D$ is semicomplete.

In order to provide enough intuition, we assume the simplest case. First, let us assume that we have $3k$ disjoint sets $A_1 , \dots , A_{3k} \subseteq V(D)$ and $3k$ disjoint sets $B_1 , \dots , B_{3k} \subseteq V(D) \setminus \bigcup_{i=1}^{3k}A_i$ such that 
\begin{itemize}
\item $|A_i| = |B_i| = 5$ for $1 \leq i \leq 3k$.
\item $D[A_i]$ contains a spanning transitive tournament $T[A_i]$ with a sink $a_i$ and $D[B_i]$ contains a spanning transitive tournament $T[B_i]$ with a source $b_i$ for $1 \leq i \leq 3k$.
\item Every vertex $v \in V(D) \setminus \left ( \bigcup_{i=1}^{3k}A_i \cup \bigcup_{i=1}^{3k}B_i \right )$ is in-dominated by $A_i$ and out-dominated by $B_i$ for $1 \leq i \leq 3k$.
\end{itemize}

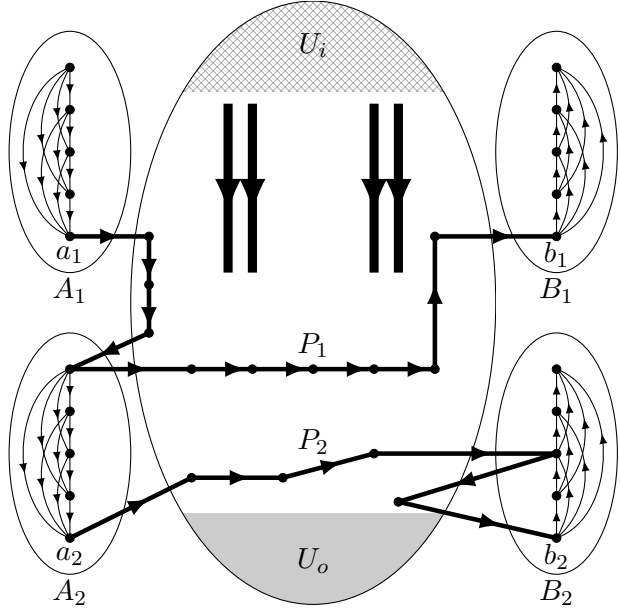
\begin{figure}[h]
\centering
\begin{tikzpicture}[scale=0.8]
\begin{scope}
\clip
  (1,1.0) rectangle (7,2.5);
  \draw[preaction={fill=none}, pattern=crosshatch, pattern color=black!30](4,-2.5) ellipse [x radius=3,y radius=5];
\end{scope}

\begin{scope}
\clip
  (1,-6) rectangle (7,1.0);
  \draw[fill=none] (4,-2.5) ellipse [x radius=3,y radius=5];
\end{scope}

\begin{scope}
\clip
  (1,-7.5) rectangle (7,-6);
  \draw[fill=black!20] (4,-2.5) ellipse [x radius=3,y radius=5];
\end{scope}

\node at (4,1.77) {$U_i$};
\draw[middleuparrow={latex}, line width=1.2mm] (2.6,0.8) -- (2.6,-2);
\draw[middleuparrow={latex}, line width=1.2mm] (3,0.8) -- (3,-2);
\draw[middleuparrow={latex}, line width=1.2mm] (5,0.8) -- (5,-2);
\draw[middleuparrow={latex}, line width=1.2mm] (5.4,0.8) -- (5.4,-2);

\node at (4,-6.77) {$U_o$};

\draw[fill=none] (0,0) ellipse [x radius=1,y radius=2];
\node at (0,-2.3) {$A_1$};
 \draw[fill=none] (0,-5) ellipse [x radius=1,y radius=2];
\node at (0,-7.3) {$A_2$};

\draw[fill=none] (8,0) ellipse [x radius=1,y radius=2];
\node at (8,-2.3) {$B_1$};
 \draw[fill=none] (8,-5) ellipse [x radius=1,y radius=2];
\node at (8,-7.3) {$B_2$};

%% transitive tournament for A_1 %%
\filldraw[fill=black] (0,1.4) circle (2pt);
\filldraw[fill=black] (0,0.7) circle (2pt);
\filldraw[fill=black] (0,0.0) circle (2pt);
\filldraw[fill=black] (0,-0.7) circle (2pt);
\filldraw[fill=black] (0,-1.4) circle (2pt);

\draw[middlearrow={latex}] (0,1.4) to[bend right=70] (0,-1.4) ;
\draw[middleuparrow={latex}] (0,1.4) -- (0,0.7) ;
\draw[middleuparrow={latex}] (0,0.7) -- (0,0.0) ;
\draw[middleuparrow={latex}] (0,0.0) -- (0,-0.7) ;
\draw[middleuparrow={latex}] (0,-0.7) -- (0,-1.4) ;
\draw[middlearrow={latex}] (0,1.4) to[bend right=50] (0,-0.7) ;
\draw[middlearrow={latex}] (0,0.7) to[bend right=50](0,-1.4) ;
\draw[middlearrow={latex}] (0,1.4) to[bend right=30] (0,0.0) ;
\draw[middlearrow={latex}] (0,0.7) to[bend right=30](0,-0.7) ;
\draw[middlearrow={latex}] (0,0.0) to[bend right=30](0,-1.4) ;

\node at (0,-1.7) {$a_1$};

%% transitive tournament for A_2 %%
\filldraw[fill=black] (0,-3.6) circle (2pt);
\filldraw[fill=black] (0,-4.3) circle (2pt);
\filldraw[fill=black] (0,-5.0) circle (2pt);
\filldraw[fill=black] (0,-5.7) circle (2pt);
\filldraw[fill=black] (0,-6.4) circle (2pt);

\draw[middlearrow={latex}] (0,-3.6) to[bend right=70] (0,-6.4) ;
\draw[middleuparrow={latex}] (0,-3.6) -- (0,-4.3) ;
\draw[middleuparrow={latex}] (0,-4.3) -- (0,-5.0) ;
\draw[middleuparrow={latex}] (0,-5.0) -- (0,-5.7) ;
\draw[middleuparrow={latex}] (0,-5.7) -- (0,-6.4) ;
\draw[middlearrow={latex}] (0,-3.6) to[bend right=50] (0,-5.7) ;
\draw[middlearrow={latex}] (0,-4.3) to[bend right=50](0,-6.4) ;
\draw[middlearrow={latex}] (0,-3.6) to[bend right=30] (0,-5.0) ;
\draw[middlearrow={latex}] (0,-4.3) to[bend right=30](0,-5.7) ;
\draw[middlearrow={latex}] (0,-5.0) to[bend right=30](0,-6.4) ;

\node at (0,-6.7) {$a_2$};

%% transitive tournament for B_1 %%
\filldraw[fill=black] (8,1.4) circle (2pt);
\filldraw[fill=black] (8,0.7) circle (2pt);
\filldraw[fill=black] (8,0.0) circle (2pt);
\filldraw[fill=black] (8,-0.7) circle (2pt);
\filldraw[fill=black] (8,-1.4) circle (2pt);

\draw[middlearrow={latex}] (8,-1.4) to[bend right=70] (8,1.4) ;
\draw[middleuparrow={latex}] (8,-1.4) -- (8,-0.7) ;
\draw[middleuparrow={latex}] (8,-0.7) -- (8,0.0) ;
\draw[middleuparrow={latex}] (8,0.0) -- (8,0.7) ;
\draw[middleuparrow={latex}] (8,0.7) -- (8,1.4) ;
\draw[middlearrow={latex}] (8,-1.4) to[bend right=50] (8,0.7) ;
\draw[middlearrow={latex}] (8,-0.7) to[bend right=50](8,1.4) ;
\draw[middlearrow={latex}] (8,-1.4) to[bend right=30] (8,0.0) ;
\draw[middlearrow={latex}] (8,-0.7) to[bend right=30](8,0.7) ;
\draw[middlearrow={latex}] (8,0.0) to[bend right=30](8,1.4) ;

\node at (8,-1.7) {$b_1$};

%% transitive tournament for B_2 %%
\filldraw[fill=black] (8,-3.6) circle (2pt);
\filldraw[fill=black] (8,-4.3) circle (2pt);
\filldraw[fill=black] (8,-5.0) circle (2pt);
\filldraw[fill=black] (8,-5.7) circle (2pt);
\filldraw[fill=black] (8,-6.4) circle (2pt);

\draw[middlearrow={latex}] (8,-6.4) to[bend right=70] (8,-3.6) ;
\draw[middleuparrow={latex}] (8,-6.4) -- (8,-5.7) ;
\draw[middleuparrow={latex}] (8,-5.7) -- (8,-5.0) ;
\draw[middleuparrow={latex}] (8,-5.0) -- (8,-4.3) ;
\draw[middleuparrow={latex}] (8,-4.3) -- (8,-3.6) ;
\draw[middlearrow={latex}] (8,-6.4) to[bend right=50] (8,-4.3) ;
\draw[middlearrow={latex}] (8,-5.7) to[bend right=50](8,-3.6) ;
\draw[middlearrow={latex}] (8,-6.4) to[bend right=30] (8,-5.0) ;
\draw[middlearrow={latex}] (8,-5.7) to[bend right=30](8,-4.3) ;
\draw[middlearrow={latex}] (8,-5.0) to[bend right=30](8,-3.6) ;
\node at (8,-6.7) {$b_2$};

%% edges of path P_1 %%
\filldraw[fill=black] (1.3,-1.4) circle (2pt);
\filldraw[fill=black] (1.3,-2.2) circle (2pt);
\filldraw[fill=black] (1.3,-3.0) circle (2pt);

\draw[middleuparrow={latex}, line width=0.6mm] (0,-1.4) -- (1.3,-1.4);
\draw[middleupuparrow={latex}, line width=0.6mm] (1.3,-1.4) -- (1.3,-2.2);
\draw[middleupuparrow={latex}, line width=0.6mm] (1.3,-2.2) -- (1.3,-3.0);
\draw[middleuparrow={latex}, line width=0.6mm] (1.3,-3.0) -- (0,-3.6);

\filldraw[fill=black] (2,-3.6) circle (2pt);
\filldraw[fill=black] (3,-3.6) circle (2pt);
\filldraw[fill=black] (4,-3.6) circle (2pt);
\filldraw[fill=black] (5,-3.6) circle (2pt);
\filldraw[fill=black] (6,-3.6) circle (2pt);
\filldraw[fill=black] (6,-1.4) circle (2pt);

\draw[middleuparrow={latex}, line width=0.6mm] (0,-3.6) -- (2,-3.6);
\draw[middleupuparrow={latex}, line width=0.6mm] (2,-3.6) -- (3,-3.6);
\draw[middleupuparrow={latex}, line width=0.6mm] (3,-3.6) -- (4,-3.6);
\draw[middleupuparrow={latex}, line width=0.6mm] (4,-3.6) -- (5,-3.6);
\draw[middleupuparrow={latex}, line width=0.6mm] (5,-3.6) -- (6,-3.6);
\draw[middleuparrow={latex}, line width=0.6mm] (6,-3.6) -- (6,-1.4);
\draw[middleuparrow={latex}, line width=0.6mm] (6,-1.4) -- (8,-1.4);
\node at (4,-3.2) {{\bf $P_1$}};

%% edges of path P_2 %%
\filldraw[fill=black] (2,-5.4) circle (2pt);
\filldraw[fill=black] (3.5,-5.4) circle (2pt);
\filldraw[fill=black] (5,-5.0) circle (2pt);
\filldraw[fill=black] (5.4,-5.8) circle (2pt);
\draw[middleuparrow={latex}, line width=0.6mm] (0,-6.4) -- (2,-5.4);
\draw[middleuparrow={latex}, line width=0.6mm] (2,-5.4) -- (3.5,-5.4);
\draw[middleuparrow={latex}, line width=0.6mm] (3.5,-5.4) -- (5,-5.0);
\draw[middleuparrow={latex}, line width=0.6mm] (5,-5.0) -- (8,-5.0);
\draw[middleuparrow={latex}, line width=0.6mm] (8,-5.0) -- (5.4,-5.8);
\draw[middleuparrow={latex}, line width=0.6mm] (5.4,-5.8) -- (8,-6.4);
\node at (4,-4.8) {{\bf $P_2$}};
\end{tikzpicture}
\caption{In-dominating sets $A_1,A_2$ and out-dominating sets $B_1,B_2$ with two paths $P_1$ and $P_2$ connecting pairs of vertices $(a_1,b_1)$ and $(a_2,b_2)$, respectively. The paths $P_1$ and $P_2$ may intersect other vertices in $A \cup B$. The thick lines depict that after removing one  vertex in $V(D)$, each remaining vertex in $V(D)\setminus (A \cup B)$ can be reached from a vertex in $U_i$ and can reach to a vertex in $U_o$ via sparse linkage structure.}
\end{figure}

We may assume that $d_{D[\left \{a_1 , \dots , a_{3k} \right \}]}^-(a_1) \geq \dots \geq d_{D[\left \{a_1 , \dots , a_{3k} \right \}]}^-(a_{3k})$ and $d_{D[\left \{b_1 , \dots , b_{3k} \right \}]}^+ (b_1) \geq \dots \geq d_{D[\left \{b_1 , \dots , b_{3k} \right \}]}^+ (b_{3k})$ by permuting indices in $[3k]$. By Lemma~\ref{lem:manydeg}, it follows that $d_{D[\left \{a_1 , \dots , a_{3k} \right \}]}^- (a_i) \geq k$ and $d_{D[\left \{b_1 , \dots , b_{3k} \right \}]}^+ (b_i) \geq k$ for $1 \leq i \leq k$.

Since $D$ is strongly $k$-connected, we can use Menger's theorem. There exists a permutation $\sigma : [k] \to [k]$ such that for $1 \leq i \leq k$, there exists a path $P_i$ from $a_i$ to $b_{\sigma(i)}$ in $D$. We may assume that $\sigma$ is an identity map by permuting indices in $[k]$. As we only permute indices in $[k]$ here, it is still preserved that $d_{D[\left \{a_1 , \dots , a_{3k} \right \}]}^- (a_i) \geq k$ and $d_{D[\left \{b_1 , \dots , b_{3k} \right \}]}^+ (b_i) \geq k$ for $1 \leq i \leq k$.

Let $A = \bigcup_{i=1}^{3k}A_i$ and $B = \bigcup_{i=1}^{3k}B_i$. Using \emph{escapers} (see Lemma~\ref{lem:escape}), there exist a set $E_{\rm escape} \subseteq E(D)$ of edges and a set $V_{\rm out} \subseteq V(D) \setminus (A \cup B)$ of vertices such that $|E_{\rm escape}| = O(k^2)$ and $|V_{\rm out}| = O(k^2)$, where they allow vertices in $A \cup B$ can easily escape from $A \cup B$ using these edges, in the following sense.
\begin{itemize}
\item[$(\rm A4.1)$] For any $S \subseteq V(D)$ with $|S| \leq k-1$ and $u \in (A \cup B) \setminus S$, there is a path from $u$ to a vertex in $V_{\rm out}$ in $D-S$ using only edges in $E_{\rm escape}$.
\item[$(\rm A4.2)$] For any $S \subseteq V(D)$ with $|S| \leq k-1$ and $u \in (A \cup B) \setminus S$, there is a path from a vertex in $V_{\rm out}$ to $u$ in $D-S$ using only edges in $E_{\rm escape}$.
\end{itemize}

Now we use the sparse linkage structure introduced in Section 2. Let us apply Lemma~\ref{lem:digraph1'} to $D[V_{\rm out}]$ and $D[V(D) \setminus (A \cup B \cup \bigcup_{i=1}^{k}V^{\rm int}(P_i) \cup V_{\rm out})]$, where we get a spanning subgraph $D'$ of $D[V_{\rm out}]$, $U_i',U_o' \subseteq V_{\rm out}$, a spanning subgraph $D''$ of $D[V(D) \setminus (A \cup B \cup \bigcup_{i=1}^{k}V^{\rm int}(P_i) \cup V_{\rm out})]$ and $U_i'' , U_o'' \subseteq V(D) \setminus (A \cup B \cup \bigcup_{i=1}^{k}V^{\rm int}(P_i) \cup V_{\rm out})$. Similarly, let us apply Lemma~\ref{lem:digraph2'} to $D[\bigcup_{i=1}^{k}V^{\rm int}(P_i) \setminus (A \cup B \cup V_{\rm out})]$, where we get a spanning subgraph $D'''$ of $D[\bigcup_{i=1}^{k}V^{\rm int}(P_i) \setminus (A \cup B \cup V_{\rm out})]$ and $U_i''',U_o''' \subseteq \bigcup_{i=1}^{k}V^{\rm int}(P_i) \setminus (A \cup B \cup V_{\rm out})$. Given any $S \subseteq V(D)$ with $|S| \leq k-1$, they satisfy the following.
\begin{itemize}
\item $|E(D')| \leq k|V_{\rm out}| - k$.
\item $|E(D'')| \leq k|V(D) \setminus (A \cup B \cup \bigcup_{i=1}^{k}V^{\rm int}(P_i) \cup V_{\rm out})| - k$.
\item $|E(D''')| \leq (k-1)|\bigcup_{i=1}^{k}V^{\rm int}(P_i) \setminus (A \cup B \cup V_{\rm out})| + (k-1)$.
\item $|U_i'|, |U_o'|,|U_i''|, |U_o''|,|U_i'''|,|U_o'''| \leq 2k-1$.	
\item[($\rm B4.1$)] For any vertex $w \in V(D') \setminus S$, there exist a path from $w$ to a vertex in $U_o' \setminus S$ in $D'-S$ and a path from a vertex in $U_i' \setminus S$ to $w$ in $D'-S$.
\item[($\rm B4.2$)] For any vertex $w \in V(D'') \setminus S$, there exist a path from $w$ to a vertex in $U_o'' \setminus S$ in $D''-S$ and a path from a vertex in $U_i'' \setminus S$ to $w$ in $D''-S$.
\item[($\rm B4.3$)] For any vertex $w \in V(D''') \setminus S$, there exist a path from $w$ to a vertex in $(U_o''' \cup \left \{b_1 , \dots , b_k \right \}) \setminus S$ in $D'''-S$ and a path from a vertex in $(U_i''' \cup \left \{a_1 , \dots , a_k \right \}) \setminus S$ to $w$ in $D'''-S$.
\end{itemize}

In the following section, an object \emph{absorber} will be related to these properties above. Let 
$$U_o := U_o' \cup U_o'' \cup U_o''' \:\:,\:\:U_i := U_i' \cup U_i'' \cup U_i'''.$$

For any $u \in U_o$ and $1 \leq i \leq k$, as $u \in V(D) \setminus (A \cup B)$, $u$ is in-dominated by $A_i$ and there exists a path $P_{u,i}$ of length at most two from $u$ to $a_i$, since $D[A_i]$ contains a spanning transitive subtournament with a sink $a_i$. Similarly, for any $v \in U_i$ and $1 \leq i \leq k$, $v$ is out-dominated by $B_i$ and there exists a path $Q_{v,i}$ of length at most two from $b_i$ to $v$, since $D[B_i]$ contains a spanning transitive subtournament with a source $b_i$.

Let us define 
$$E' := E(D[A \cup B]) \cup \bigcup_{u \in U_o}\bigcup_{i=1}^{3k} E(P_{u,i}) \cup \bigcup_{v \in U_i}\bigcup_{i=1}^{3k} E(Q_{v,i}).$$ 
Then $|E'| = O(k^2)$, as $|U_o| \leq 6k$ and $|U_i| \leq 6k$.

Let $S \subseteq V(D)$ with $|S| \leq k-1$ and $u \in U_o \setminus S$. For any $1 \leq t \leq k$ with $a_t \notin S$, we claim that there exists a path from $u$ to $a_t$ in $D$ only using edges in $E'$. Indeed, let $i \in N_D^-(a_t)$ be an index with $A_i \cap S = \emptyset$, which is guaranteed by $d_{D[\left \{a_1 , \dots , a_{3k} \right \}]}^-(a_t) \geq k$ and the disjointness of $A_1 , \dots , A_{3t}$. Since $A_i \cap S = \emptyset$ and $a_t \notin S$, the path $P_{u,t}^* := P_{u,i} \cup (a_i , a_t)$ does not intersect $S$ and is from $u$ to $a_t$ only using edges in $E'$.  Similarly, for any $v \in U_i \setminus S$ and $b_t \notin S$ with $1 \leq t \leq k$, there exists a path from $b_t$ to $v$ in $D$ only using edges in $E'$. In summary, 
\begin{itemize}
\item[($\rm C4.1$)] For any $S \subseteq V(D)$ with $|S| \leq k-1$, $u \in U_o \setminus S$ and $a_t \notin S$ with $1 \leq t \leq k$, there exists a path from $u$ to $a_t$ in $D$ only using edges in $E'$.
\item[($\rm C4.2$)] For any $S \subseteq V(D)$ with $|S| \leq k-1$, $v \in U_i \setminus S$ and $b_t \notin S$ with $1 \leq t \leq k$, there exists a path from $b_t$ to $v$ in $D$ only using edges in $E'$
\end{itemize}

In the following section, an object \emph{hub} will attain these properties above. Now, let $D_{\rm sparse}$ be a spanning subgraph of $D$ with the edge set
$$ \bigcup_{i=1}^{k}E(P_i) \cup E_{\rm escape} \cup E(D') \cup E(D'') \cup E(D''') \cup E'$$

Then it is straightforward to see that $|E(D_{\rm sparse})| = k|V(D)| + O(k^2)$. Now we claim that $D_{\rm sparse}$ is strongly $k$-connected. Let $S \subseteq V(D)$ with $|S| \leq k-1$ and $u,v \in V(D) \setminus S$. We aim to find a path from $u$ to $v$ in $D_{\rm sparse}-S$. Let $i \in [k]$ be an index such that $V(P_i) \cap S = \emptyset$. 

Now it suffices to find a path from $u$ to $u^* \in U_o \setminus S$ in $D_{\rm sparse}-S$ and a path from $v^* \in U_i \setminus S$ to $v$ in $D_{\rm sparse}-S$. Indeed, by $(\rm C4.1)$ and $(\rm C4.2)$ we have a path from $u^*$ to $a_i$ and a path from $b_i$ to $v^*$. Together with the path $P_i$, there exists a path from $u$ to $v$ in $D_{\rm sparse}-S$ as desired.
\begin{itemize}
\item If $u \in A \cup B$, then by $(\rm A4.1)$, there exists a path from $u$ to $u' \in V_{\rm out}$ in $D_{\rm sparse}-S$. By $(\rm B4.1)$, there exists a path from $u'$ to $u^* \in U_o \setminus S$ in $D_{\rm sparse}-S$.

\item If $u \in \bigcup_{i=1}^{k} V^{\rm int}(P_i) \setminus V_{\rm out}$, then by $(\rm B4.3)$ there is a path $P$ from $u$ to a vertex $w \in U_o \cup \left \{b_1 , \dots , b_k \right \}$ in $D'''-S$. If $w \in U_o$, then let $u^* := w$. Otherwise, $w \in \left \{b_1 , \dots , b_k \right \} \subseteq A \cup B$, where this case has been already considered above.

\item If $u \in V(D) \setminus (A \cup B \cup \bigcup_{i=1}^{k} V^{\rm int}(P_i))$, then by $(\rm B4.1)$ and $(\rm B4.2)$ there is a path from $u$ to a vertex $u^* \in U_o \setminus S$ in $D_{\rm sparse}-S$.
\end{itemize}

Similarly, one can find a path from a vertex $v^* \in U_i \setminus S$ to $v$. This proves that $D_{\rm sparse}$ is strongly $k$-connected.

Note that this proof only works when for $1 \leq i \leq 3k$, every vertex in $V(D) \setminus (A \cup B)$ is in-dominated by $A_i$, is also out-dominated by $B_i$, and $A_i$ and $B_i$ have the small size.  As we cannot guarantee the existence of these subsets of vertices, this ideal situation might not happen. Nevertheless, we are able to force all vertices in $V(D) \setminus (A \cup B)$ to satisfy the conditions close to the ideal one as follows.

Using Lemma~\ref{lem:dom}, we choose 5-indominators $A_1 , \dots , A_{5k}$ and 5-outdominators $B_1 , \dots , B_{5k}$ (see Definitions~\ref{def:indom} and~\ref{def:outdom}). Each of these 5-indominators $A_i$ (5-outdominators $B_i$) would in-dominate (out-dominate, respectively) all vertices in $V(D) \setminus (A \cup B)$ but a few exceptional vertices $U_i^+$ ($U_i^-$, respectively). As the size of $U_i^+$ or $U_i^-$ could be $\Omega(n)$, we utilise the following two observations to reduce the size. First, we do not need to force all vertices to in/out-dominated by all $5k$ 5-in/outdominators. Second, if the size of $U_i^+$ ($U_i^-$) is big enough, then all vertices in $U_i^+$ ($U_i^-$, respectively) have large out-degree (in-degree, respectively) so they can easily escape from $U_i^+$ ($U_i^-$, respectively).

Hence, we regard any vertex $v \in V(D) \setminus (A \cup B)$ as an exceptional vertex only when there are more than $k$ indices $i \in [5k]$ such that $v$ is not in-dominated by $A_i$ (not out-dominated by $B_i$) and $|U_i^+|$ ($|U_i^-|$, respectively) is not big enough. Let $O^+$ ($O^-$, respectively) be the set of all these vertices in $V(D) \setminus (A \cup B)$ and $O^* := O^+ \cup O^-$ be the set of exceptional vertices. In summary, it can be shown that 5-indominators $A_1 , \dots , A_{5k}$, 5-outdominators $B_1 , \dots , B_{5k}$ and $O^*$ attain the following properties (see Lemma~\ref{lem:trioexist}).
\begin{itemize}
\item $|O^*| = O(k)$.
\item For any vertex $w \in V(D) \setminus (A \cup B \cup O^*)$, there exist a path of length at most two from $w$ to a vertex in $A_i$ for at least $4k$ indices $i \in [5k]$, and a path of length at most two from a vertex in $B_i$ to $w$ for at least $4k$ indices $i \in [5k]$. 
\end{itemize}

Indeed, as every vertex in $V(D) \setminus (A \cup B)$ is not in/out-dominated by all 5-in/outdominators, we cannot simply follow the proof illustrated in this section and it is required to develop more ideas. In the following section, we introduce the objects according to the modification discussed as above.

\section{Basic objects in the construction}\label{sec:object}
As the proof of the main result consists of many technical parts, we divide the proof into statements constructing objects called \emph{dominators}, \emph{trios}, \emph{escapers}, \emph{hubs}, and \emph{absorbers}. Dominators are the most basic objects, very simple but useful in controlling the length of many disjoint paths. A collection of many dominators with many good properties are called a trio, which is  our main interest when involving collections of many dominators. Based on trios, we construct hubs and absorbers, and combine them into a highly connected spanning subgraph with few edges to prove Theorem~\ref{thm:main}.

\subsection{Dominators}
In this subsection, we define indominators and outdominators in digraphs, which are the most basic objects in constructing a sparse highly connected spanning subgraph.

\begin{DEFN}\label{def:indom}
Let $t \geq 1$ be an integer. A \emph{$t$-indominator} is a quadtuple $(D, A,x,a)$ such that $D$ is a directed multigraph, $A$ is a subset of $V(D)$ with at most $t$ vertices, and $x,a \in A$ satisfying the following.
\begin{itemize} 
\item[$(\rm ID1)$] $D[A]$ contains a spanning transitive tournament with a source $x$ and a sink $a$.
\item[$(\rm ID2)$] $x$ has at least $2^{t-1} |U^+|$ out-neighbours in $D$, where $U^+ := \bigcap_{v \in A}{N_D^+(v)} \setminus \bigcup_{v \in A}{N_D^-(v)}$.
\end{itemize}
\end{DEFN}

\begin{DEFN}\label{def:outdom}
Let $t \geq 1$ be an integer. A \emph{$t$-outdominator} is a quadtuple $(D , B,x',b)$ such that $D$ is a directed multigraph, $B$ is a subset of $V(D)$ with at most $t$ vertices, and $x',b \in B$ satisfying the following.
\begin{itemize} 
\item[$(\rm OD1)$] $D[B]$ contains a spanning transitive tournament with a source $b$ and a sink $x'$.
\item[$(\rm OD2)$] $x'$ has at least $2^{t-1} |U^-|$ in-neighbours in $D$, where $U^- := \bigcap_{v \in B}{N_D^-(v)} \setminus \bigcup_{v \in B}{N_D^+(v)}$.
\end{itemize}
\end{DEFN}

The following lemma guarantees the existence of a $t$-in/outdominator in directed multigraphs. This is a variation of~\cite[Lemma 2.3]{pokrovskiy2015highly} proved for tournaments.

\begin{LEM}\label{lem:dom}
Let $t\geq 1$ be an integer. For each vertex $x$ of a directed multigraph $D$, there exist $A \subseteq V(D)$ and $a \in A$ such that $(D , A,x,a)$ is a $t$-indominator, and $B \subseteq V(D)$ and $b \in B$ such that $(D,B,x,b)$ is a $t$-outdominator.
\end{LEM}
\begin{proof}
We only prove that there exist $A \subseteq V(D)$ and $a \in A$ such that $(D,A,x,a)$ is a $t$-indominator. The rest of the proof follows by reversing orientations of all edges.

Let $G$ be an oriented graph obtained from $D$ by removing multiple edges and exactly one edge from each directed 2-cycle. Let $V_1 := N_{G}^{+}(x)$ and $v_1 := x$. Let $s$ be the maximum integer that satisfies $1 \leq s \leq t$ and $v_1 , \dots , v_s \in V(D)$ and $V_1 , \dots, V_s \subseteq V(D)$ satisfying the following properties.

\begin{itemize}[noitemsep, nolistsep]
\item[$(\rm i)$] For $1 \leq i < j \leq s$, $v_j \in N_{G}^{+}(v_i)$.
\item[$(\rm ii)$] For $1 \leq i \leq s$, $V_i := \bigcap_{k=1}^{i}{N_{G}^{+}(v_k)}$.
\item[$(\rm iii)$] For $1 \leq i < s$, $|V_{i+1}| \leq \frac{1}{2}|V_i|$.
\end{itemize}

Note that such $s$ exists as $(\rm i)$, $(\rm ii)$, and $(\rm iii)$ hold for $s=1$. We claim that $V_s = \emptyset$ or $s=t$. Otherwise, let $v_{s+1} \in V_s$ with $d_{G[V_s]}^{+}(v_{s+1}) \leq \frac{|V_s|}{2}$. Indeed, since $G$ is an oriented graph, $G[V_s]$ contains at most $\frac{|V_s|(|V_s|-1)}{2}$ edges, proving that there is a vertex in $V_s$ with out-degree at most $\frac{|V_s|-1}{2}$. Let us define $V_{s+1} := N_{G[V_s]}^+(v_{s+1}) = V_s \cap N_G^{+}(v_{s+1})$, then $|V_{s+1}| \leq \frac{|V_s|}{2}$, contradicting the maximality of $s$.

Therefore, $V_s = \emptyset$ or $s=t$. Let us define $A := \left \{v_1 , \dots , v_s \right \}$ with $a := v_s$. Then $G[A]$ is a transitive tournament with a source $x$ and a sink $a$. Let $V^+ := V_s = \bigcap_{k=1}^{s}{N_G^{+}(v_k)}$. Since $|V^+| \leq 2^{-t+1} |V_1|$ by (iii), implying $|N_G^{+}(x)| = |V_1| \geq 2^{t-1} |V^+|$. Now we claim that
$$\bigcap_{v \in A}N_D^{+}(v) \setminus \bigcup_{v \in A}N_D^{-}(v) \subseteq \bigcap_{v \in A}N_G^+(v).$$

For every $w \in \bigcap_{v \in A}N_D^{+}(v) \setminus \bigcup_{v \in A}N_D^{-}(v)$, we have $w \in \bigcap_{v \in A}{N_G^+(v)}$ otherwise there exists $v \in A$ such that $wv,vw \in E(D)$, implying that $w \in \bigcup_{v \in A}{N_D^{-}(v)}$ and contradicting the assumption on $w$. Therefore, $|V^+| \geq |U^+|$ and we have

$$|N_D^+(x)| \geq |N_G^+(x)| \geq 2^{t-1} |V^+| \geq 2^{t-1} |U^+|,$$
where $U^+ := \bigcap_{v \in A}{N_D^+(v)} \setminus \bigcup_{v \in A}{N_D^-(v)}$. This proves that $(D, A,x,a)$ is a $t$-indominator.
\end{proof}

Throughout the proof,  it is worth noting that $t$ will be always 5 when regarding $t$-indominators and $t$-outdominators.

%{\bf Remark}. Lemma~\ref{lem:in-dom} is one of the biggest differences from the methods in~\cite{kang2016sparse}, as they use the sets in-dominating all other vertices. They defined $A$ to in-dominate every other vertex $v \notin A$ to ensure that there is a short path from $v$ to $A$. However, the size of $A$ may be very large (for example, if $v$ has outdegree $\Omega(f(n))$ then $|A| = \Omega(\log_2 f(n))$), and it contributes a logarithmic factor on $g(k)$ in Theorem~\ref{thm:old}. In the proof of Lemma~\ref{lem:in-dom}, we stop the procedure if the size $A$ reaches 10; this returns a set $A$ with bounded number of vertices, however, there may be many vertices outside $A$ not in-dominated by $A$; but if we choose $v$ to be a vertex of the minimum out-degree in $T$ then every vertex $u \notin U^+$ has out-degree $2^9 |U^+|$, since $u$ has out-degree at least $d_T^+(v)$ in $T$. Hence there are many out-neighbours of $v$ in-dominated by $A$ even if $|U^+|$ is large, so there are many disjoint short paths from $v$ to $A$ even if $A$ does not in-dominate many vertices outside $A$.

\subsection{Trios}
In Section~\ref{sec:proofsketch}, we sketched the proof provided that every vertex in $V(D) \setminus (A \cup B)$ is in-dominated by $A_1 , \dots , A_{3k}$ and out-dominated by $B_1 , \dots , B_{3k}$. However, we cannot guarantee these sets in/out-dominating all other vertices, but the sets in/out-dominating almost all other vertices by Lemma~\ref{lem:dom}. In this subsection, we introduce the object called a \emph{trio}, allowing that most of the  vertices can reach to many 5-indominators and can be reached from many 5-outdominators by paths of length at most two. The other subsections will introduce other objects to follow the sketched proof in Section~\ref{sec:proofsketch} according to this modification.

\begin{DEFN}\label{def:trio}
Let $d,k,t_1,t_2 \geq 1$, $m \geq k$, $\overline{\Delta} \geq 0$ be integers, and $u > 0$ be a real number. Let $D$ be a directed multigraph with $\overline{\Delta}(D) \leq \overline{\Delta}$. A 3-tuple $(\mathcal{A} , \mathcal{B} , O^*)$ is called a \emph{$(t_1 , t_2 , d , m , u)$-{trio}} in $D$ 
if $\mathcal{A}$ is a collection of $m$ distinct 5-indominators $\left \{(D_i , A_i , x_i , a_i) \right \}_{i=1}^{m}$, and $\mathcal{B}$ is a collection of $m$ distinct 5-outdominators $\left \{(D_i' , B_i , x_i' , b_i ) \right \}_{i=1}^{m}$, and a subset $O^* \subseteq V(D)$ of vertices satisfying the following properties, where $U_i^+ := \bigcap_{w \in A_i}N_{D_i}^+(w) \setminus \bigcup_{w \in A_i}N_{D_i}^-(w)$ and $U_i^- := \bigcap_{w \in B_i}N_{D_i'}^-(w) \setminus \bigcup_{w \in B_i}N_{D_i'}^+(w)$.

\begin{itemize} 
\item[$(\rm T1)$] For every $i \in [m]$, $D_i$ is a subgraph of $D$, and contains $D - (\bigcup_{i=1}^{m} A_i \cup \bigcup_{i=1}^{m} B_i)$ as a subgraph.

\item[$(\rm T2)$] For every $i \in [m]$, $D_i'$ is a subgraph of $D - \bigcup_{i=1}^{m}A_i$, and contains $D - (\bigcup_{i=1}^{m} A_i \cup \bigcup_{i=1}^{m} B_i)$ as a subgraph.

\item[$(\rm T3)$] $A_1 , \dots , A_{m}$, $B_1 , \dots , B_{m}$ are disjoint subsets.

\item[$(\rm T4)$] For every $i \in [k]$, $|N_{D[\left \{a_1 , \dots , a_m \right \}]}^- (a_i)| \geq \frac{m-k-\overline{\Delta}}{2}$ and $|N_{D[\left \{b_1 , \dots , b_m \right \}]}^+ (b_i)| \geq \frac{m-k-\overline{\Delta}}{2}$.

\item[$(\rm T5)$] For every $v \in V(D) \setminus (\bigcup_{i=1}^{m}A_i \cup \bigcup_{i=1}^{m}B_i \cup O^*)$, there are at least $m - t_1 - t_2$ indices $i \in [m]$ such that either $v$ is in-dominated by $A_i$, or $v$ is in $U_i^+$ with $|U_i^+| \geq u$.

\item[$(\rm T6)$] For every $v \in V(D) \setminus (\bigcup_{i=1}^{m}A_i \cup \bigcup_{i=1}^{m}B_i \cup O^*)$, there are at least $m - t_1 - t_2$ indices $i \in [m]$ such that either $v$ is out-dominated by $B_i$, or $v$ is in $U_i^-$ with $|U_i^-| \geq u$.

\item[$(\rm T7)$] For every $u \in U_i^+$ with $i \in [m]$ and $\left | U_i^+ \right | \geq u$, the vertex $u$ has at least $d + \left | U_i^+ \right |$ out-neighbours in $D_i$.

\item[$(\rm T8)$] For every $u \in U_i^-$ with $i \in [m]$ and $\left | U_i^- \right | \geq u$, the vertex $u$ has at least $d + \left | U_i^- \right |$ in-neighbours in $D_i'$.

\item[$(\rm T9)$] $|O^*|$ is small enough; $|O^*|$ is at most $\frac{2mu}{t_1} + \frac{10 \overline{\Delta} m}{t_2}$, and if $t_2 \geq \overline\Delta$ then $|O^*| \leq \frac{2mu}{t_1}$.
\end{itemize}
\end{DEFN}

The following lemma guarantees a $(t_1,t_2,d,m,u)$-trio for dense digraphs.

\begin{LEM}\label{lem:trioexist}
Let $d,k,n,m,t_1 , t_2 \geq 1$, $\overline{\Delta} \geq 0$ be integers with $m \geq k$, and $u > 0$ be a real number. Let $D$ be an $n$-vertex directed multigraph with $\overline{\Delta}(D) \leq \overline{\Delta}$. If $n \geq 10m$ and $u \geq \frac{d}{15}$, then $D$ contains a $(t_1 , t_2,d,m,u)$-trio $(\mathcal{A},\mathcal{B},O^*)$.
\end{LEM}
\begin{proof}
First of all, we construct $m$ distinct 5-indominators satisfying some properties.

\begin{CLAIM}\label{claim:indom*}
There exist a collection $\mathcal{A}$ of $m$ distinct 5-indominators $\left \{ (D_i , A_i , x_i , a_i ) \right \}_{i=1}^{m}$ satisfying the following. For every $1 \leq i \leq m$,
\begin{itemize} 
\item[$(1)$] $D_i := D - \bigcup_{j=1}^{i-1}{A_j}$.
\item[$(2)$] $x_i$ is a vertex in $D_i$ with the smallest number of out-neighbours in $V(D_i)$.
\item[$(3)$] $(D_i , A_i , x_i , a_i )$ is a 5-indominator.
\end{itemize}
\end{CLAIM}
\begin{proof}[Proof of Claim~\ref{claim:indom*}]
Since $|V(D)|=n \geq 5m$, the claim follows by successively applying Lemma~\ref{lem:dom}.
\end{proof}

Let us define $A := \bigcup_{i=1}^{m}A_i$. Now we construct $m$ distinct 5-outdominators satisfying some properties.

\begin{CLAIM}\label{claim:outdom*}
There exist a collection $\mathcal{B}$ of $m$ distinct 5-outdominators $\left \{ (D_i' , B_i , x_i' , b_i  ) \right \}_{i=1}^{m}$ satisfying the following. For every $1 \leq i \leq m$,
\begin{itemize} 
\item[$(1)$] $D_i' := D - (A \cup \bigcup_{j=1}^{i-1}{B_j})$.
\item[$(2)$] $x_i'$ is a vertex in $D_i'$ with the smallest number of in-neighbours in $V(D_i')$.
\item[$(3)$] $(D_i' , B_i , x_i' , b_i )$ is a 5-outdominator.
\end{itemize}
\end{CLAIM}
\begin{proof}[Proof of Claim~\ref{claim:outdom*}]
Since $|V(D) \setminus A| \geq n - 5m \geq 5m$, the claim follows by successively applying Lemma~\ref{lem:dom}.
\end{proof}

Let us define $B := \bigcup_{i=1}^{m}{B_i}$, and for every $i \in [m]$, let us define 
$$U_i^+ := \bigcap_{v \in A_i}N_{D_i}^+(v) \setminus \bigcup_{v \in A_i}N_{D_i}^-(v) \:\:, \: \:U_i^- := \bigcap_{v \in B_i}N_{D_i'}^-(v) \setminus \bigcup_{v \in B_i}N_{D_i'}^+(v).$$

By $(\rm ID2)$ and $(\rm OD2)$, for every $i \in [m]$ we have 
\begin{align}\label{eqn:deg}
|N_{D_i}^+(x_i)| \geq 16 |U_i^+|\:\: , \:\:|N_{D_i'}^-(x_i)| \geq 16 |U_i^-|
\end{align}

Since both $D_i$ and $D_i'$ contain $D-(A \cup B)$ as a subgraph for $1 \leq i \leq m$, this proves $(\rm T1)$ and $(\rm T2)$ of Definition~\ref{def:trio}. From the construction of $\mathcal{A}$ and $\mathcal{B}$, $(\rm T3)$ is clear.

By Lemma~\ref{lem:manydeg} and permuting indices, we may assume that for every $i \in [k]$, 
$$|N_{D[\left \{a_1 , \dots , a_m \right \}]}^- (a_i)| \geq \frac{m-k-\overline{\Delta}}{2} \:\:, \:\:|N_{D[\left \{b_1 , \dots , b_m \right \}]}^+ (b_i)| \geq \frac{m-k-\overline{\Delta}}{2}.$$ 
which proves $(\rm T4)$ of Definition~\ref{def:trio}.

For $1 \leq i \leq m$, let
\begin{align*}
F_i^+ := V(D_i) \setminus (A_i \cup U_i^+ \cup \bigcup_{v \in A_i}{N_{D_i}^- (v)}) \:\:, \:\: F_i^- := V(D_i') \setminus (B_i \cup U_i^- \cup \bigcup_{v \in B_i}{N_{D_i'}^+ (v)}),
\end{align*}
where $F_i^+$ is the set of vertices $v$ in $V(D_i) \setminus A_i$ that are not in-dominated by $A_i$ and are non-neighbours of some vertices in $A_i$, and $F_i^-$ is the set of vertices $v$ in $V(D_i') \setminus B_i$ that are not out-dominated by $B_i$ and are non-neighbours of some vertices in $B_i$.

Since every vertex in $D$ has at most $\overline{\Delta}$ other non-neighbour vertices and $|A_i|,|B_i| \leq 5$ for $i \in [m]$, it follows that 
\begin{align}
|F_i^+|, |F_i^-| \leq 5 \overline{\Delta}.\label{eqn:f_i}
\end{align}

It is easy to observe the following, from the definitions of $U_i^+$, $F_i^+$, $U_i^-$, and $F_i^-$.

\begin{OBS}\label{obs:trivial}
For every vertex $v \in V(D) \setminus (A \cup B)$ and $i \in [m]$, the following hold.
\begin{itemize} 
\item Either $v$ is in-dominated by $A_i$, or $v$ is in $U_i^+$, or $v$ is in $F_i^+$.
\item Either $v$ is out-dominated by $B_i$, or $v$ is in $U_i^-$, or $v$ is in $F_i^-$.
\end{itemize}
\end{OBS}

Let us define
\begin{align*}
I^+ &:= \left \{ i \in [m] \: \colon \: |U_i^{+}| < u \right \} , \: \: I^- := \left \{ i \in [m] \: \colon \: |U_i^{-}| < u \right \},\label{def:i}\\
O^+ &:= \left \{ v \in V(D) \:\colon\: \lvert \left \{ i \in I^+ \: : \: v \in U_i^{+} \right \} \rvert > t_1 \right \}, \\
F^+ &:= \left \{ v \in V(D) \:\colon\: \lvert \left \{ i \in [m] \: : \: v \in F_i^{+} \right \} \rvert > t_2 \right \}, \\
O^- &:= \left \{ v \in V(D) \:\colon\: \lvert \left \{ i \in I^- \: : \: v \in U_i^{-} \right \} \rvert > t_1 \right \},\\
F^- &:= \left \{ v \in V(D) \:\colon\: \lvert \left \{ i \in [m] \: : \: v \in F_i^{-} \right \} \rvert > t_2 \right \},\\
O &:= O^+ \cup O^-,\\
F &:= F^+ \cup F^-.
\end{align*}

Let $O^* := O \cup F$. By Observation~\ref{obs:trivial} and the definition of $O^*$, both $(\rm T5)$ and $(\rm T6)$ of Definition~\ref{def:trio} are satisfied.

\begin{CLAIM}\label{claim:sizebdd}
The following hold.
\begin{itemize} 
\item[$(\rm 1)$] For every $i \in [m] \setminus I^+$ and $v \in V(D_i) \setminus A_i$, $|N_{D_i}^{+}(v)| \geq d + |U_i^{+}|$.
\item[$(\rm 2)$] For every $i \in [m] \setminus I^-$ and $w \in V(D_i') \setminus B_i$, $|N_{D_i'}^{-}(w)| \geq d + |U_i^{-}|$.
\item[$(\rm 3)$] $|O| \leq \frac{2mu}{t_1}$.
\item[$(\rm 4)$] $|F| \leq \frac{10 \overline{\Delta} m}{t_2}$. Moreover, if $t_2 \geq \overline{\Delta}$, then $F = \emptyset$.
\end{itemize}
\end{CLAIM}
\begin{proof}[Proof of Claim~\ref{claim:sizebdd}]
For every $i \in [m]$, we have $|N_{D_i}^+(x_i)| \geq 16|U_i^+|$ and $|N_{D_i'}^-(x_i')| \geq 16|U_i^-|$ by~\eqref{eqn:deg}. From the definition of $x_i$ and $x_i'$, it follows that for every $v \in V(D_i) \setminus A_i$ and $w \in V(D_i') \setminus B_i$,
\begin{align*}
|N_D^+ (v)| & \geq |N_{D_i}^+(x_i)| \geq 16|U_i^+|\\
|N_{D-A}^- (w)| & \geq |N_{D_i'}^-(x_i')| \geq 16|U_i^-|.
\end{align*}
by Claims~\ref{claim:indom*} and~\ref{claim:outdom*}.

For every $i \in [m] \setminus I^+$ and $v \in V(D_i) \setminus A_i$, since $|U_i^{+}| \geq u$ it follows that $|N_{D_i}^{+}(v)| \geq 16|U_i^{+}| \geq d + |U_i^{+}|$ since $u \geq d/15$. Similarly, for every $i \in [m] \setminus I^-$ and $w \in V(D_i') \setminus B_i$, we have $|N_{D_i'}^{-}(w)| \geq d + |U_i^{-}|$. This proves (1) and (2).

Since every vertex in $O^+$ is in $U_i^{+}$ for more than $t_1$ indices $i \in I^+$,
$$ t_1 |O^+| \leq \sum_{i \in I^+}{|U_i^{+}|} \leq |I^+| \cdot u \leq m \cdot u$$
and $|O^+| \leq \frac{m \cdot u}{t_1}$. Similarly, $|O^-| \leq \frac{m \cdot u}{t_1}$, implying that $|O| \leq \frac{2mu}{t_1}$. This proves (3).

If $\overline{\Delta} = 0$, then $(\rm 4)$ is trivial. We may assume that $\overline{\Delta} > 0$. Since every vertex in $F^+$ is in $F_i^{+}$ for more than $t_2$ indices $i \in [m]$ and by~\eqref{eqn:f_i},
$$ t_2 |F^+| \leq \sum_{i \in [m]}{|F_i^{+}|} \leq m \cdot 5 \overline{\Delta}$$
and $|F^+| \leq \frac{5 \overline{\Delta} m}{t_2}$. Similarly, $|F^-| \leq \frac{5 \overline{\Delta} m}{t_2}$, implying that $|F| \leq \frac{10 \overline{\Delta} m}{t_2}$.

If $t_2 \geq \overline{\Delta}$, then for every $v \in F^+$, there are more than $\overline{\Delta}$ indices $i \in [m]$ such that $v \in F_i^+$ and there is $w \in A_i$ with $(v,w),(w,v) \notin E(D)$, implying that $v$ has more than $\overline{\Delta}$ non-neighbours. Hence $F^+ = \emptyset$. Similarly, we have $F^- = \emptyset$. This proves (4).
\end{proof}

Since $O^* = O \cup F$, $|O^*| \leq |O| + |F| \leq \frac{2mu}{t_1} + \frac{10 \overline{\Delta} m}{t_2}$ by Claim~\ref{claim:sizebdd}. If $t_2 \geq \overline{\Delta}$, then $F = \emptyset$ and thus $|O^*| \leq |O| \leq \frac{2mu}{t_1}$. Hence $(\mathcal{A},\mathcal{B},O^*)$ is a $(t_1 , t_2 , d , m , u)$-trio since $(\rm T7)$--$(\rm T9)$ hold by Claim~\ref{claim:sizebdd}.
\end{proof}

\subsection{Escapers}
In this subsection, we consider objects called escapers. Roughly speaking, given a directed multigraph $D$ and a small set $U \subseteq V(D)$, a $k$-escaper is a set of edges such that every vertex in $U$ can escape from $U$ to $V(D) \setminus U$ by a path, after we remove less than $k$ vertices of $D$. Finding $k$-escapers with few edges is one of the most crucial parts in constructing a sparse strongly $k$-connected subgraph of $D$. 

\begin{DEFN}
Let $k \geq 1$ be an integer and $D$ be a digraph. A \emph{$k$-escaper} in $D$ is a triple $(E_{\rm escape} , U , U_{\rm out})$ of a subset $E_{\rm escape}$ of $E(D)$ and subsets $U$ and $U_{\rm out}$ of $V(D)$ such that 
\begin{itemize} 
\item[$(\rm E1)$] $U_{\rm out} \subseteq V(D) \setminus U$, 
\item[$(\rm E2)$] For every $S \subseteq V(D)$ with $|S| \leq k-1$ and any vertex $u \in U \setminus S$, a subgraph $D-S$ contains a path from $u$ to a vertex in $U_{\rm out}$ only using edges in $E_{\rm escape}$, and
\item[$(\rm E3)$] For every $S \subseteq V(D)$ with $|S| \leq k-1$ and any vertex $v \in U \setminus S$, a subgraph $D-S$ contains a path from a vertex in $U_{\rm out}$ to $v$ only using edges in $E_{\rm escape}$.
\end{itemize}
\end{DEFN}

The following lemma is the main lemma of this subsection, which allows us to find a sparse $k$-escaper of a set $U$ of vertices.

\begin{LEM}\label{lem:escape}
Let $k,n \geq 1$ be integers. Let $D$ be a strongly $k$-connected digraph, and $U \subseteq V(D)$. If $|U| \leq |V(D)|-k$, then there is a $k$-escaper $(E_{\rm escape} , U , U_{\rm out})$ in $D$ such that $|E_{\rm escape}| \leq 4k|U|$ and $|U_{\rm out}| \leq 2k|U|$. %Moreover, $(E_{\rm escape} , U , U_{\rm out})$ can be found in $O(k \cdot (|E(D)|^2 + |U| \cdot |E(D)|))$.
\end{LEM}
\begin{proof}
Let $D'$ be a minimally strongly $k$-connected spanning subgraph of $D$. Since $|V(D)\setminus U| \geq k$, we can apply Proposition~\ref{prop:fan} as follows. For every $u \in U$, there are a $k$-fan $\{P^+_{u,i} \}_{i=1}^{k}$ from $u$ to $V(D) \setminus U$ and a $k$-fan $\{P^-_{u,i}\}_{i=1}^{k}$ from $V(D) \setminus U$ to $u$. 

Let us define
\begin{align}
E_{\rm escape} &:= \bigcup_{u \in U} \left ( \bigcup_{i=1}^{k} E(P^+_{u,i}) \cup \bigcup_{i=1}^{k} E(P^-_{u,i}) \right ),\\
U_{\rm out} &:= \bigcup_{u \in U} \left ( \bigcup_{i=1}^{k} V(P^+_{u,i}) \cup \bigcup_{i=1}^{k} V(P^-_{u,i}) \right ) \setminus U,
\end{align}
which proves $(\rm E1)$. %Note that both $E_{\rm escape}$ and $U_{\rm out}$ can be found in $O(k \cdot (|E(D)|^2 + |U| \cdot |E(D)|))$.

For every $u \in U$, it follows that 
\begin{align}
\left | U_{\rm out} \cap \bigcup_{i=1}^{k}V(P^+_{u,i}) \right | = k, \: \: \left | U_{\rm out} \cap \bigcup_{i=1}^{k}V(P^-_{u,i}) \right | = k.
\end{align}
and thus $|U_{\rm out}| \leq 2k|U|$ and $|E_{\rm escape}| \leq |E(D'[U])| + |U_{\rm out}| \leq 4k|U|$ by Proposition~\ref{prop:fewedges}. 

Since $E_{\rm escape} \subseteq E(D')$, it is a subset of $E(D)$. Now we claim that $(E_{\rm escape}, U , U_{\rm out})$ is a $k$-escaper. For every $S \subseteq V(D)$ with $|S| \leq k-1$ and $u \in U \setminus S$, there is $i \in [k]$ with $V(P^+_{u,i}) \cap S = \emptyset$. Since $E(P^+_{u,i}) \subseteq E_{\rm escape}$ and by the definition of $U_{\rm out}$, this proves $(\rm E2)$. Similarly $(\rm E3)$ holds by the same proof.
\end{proof}

We also define an edge-version of escapers.

\begin{DEFN}
Let $k \geq 1$ be an integer and $D$ be a directed multigraph. A \emph{$k$-arc-escaper} in $D$ is a 3-tuple $(E_{\rm escape} , U , U_{\rm out})$ satisfying the following.
\begin{itemize} 
\item[$(\rm E1')$] $U_{\rm out} \subseteq V(D) \setminus U$.
\item[$(\rm E2')$] For every $F \subseteq E(D)$ with $|F| \leq k-1$ and any vertex $u \in U$, a subgraph $D-F$ contains a path from $u$ to a vertex in $U_{\rm out}$ only using edges in $E_{\rm escape}$.
\item[$(\rm E3')$] For every $F \subseteq E(D)$ with $|F| \leq k-1$ and any vertex $v \in U$, a subgraph $D-F$ contains a path from a vertex in $U_{\rm out}$ to $v$ only using edges in $E_{\rm escape}$.
\end{itemize}
\end{DEFN}

Replacing Proposition~\ref{prop:fewedges} by Corollary~\ref{cor:minimal_arc} in the proof of Lemma~\ref{lem:escape}, the following lemma easily follows.

\begin{LEM}\label{lem:escape_arc}
Let $k,n \geq 1$ be integers. Let $D$ be an $n$-vertex strongly $k$-arc-connected directed multigraph, and $U \subsetneq V(D)$. Then there is a $k$-arc-escaper $(E_{\rm escape} , U , U_{\rm out})$ in $D$ such that $|E_{\rm escape}| \leq 4k|U|$ and $|U_{\rm out}| \leq 2k|U|$.
\end{LEM}

\subsection{Hubs}
In this subsection, we consider objects called \emph{hubs}, which allow us to connect a set of vertices with the vertices of dominators. Hubs are one of the main parts in constructing highly connected sparse spanning subgraphs of dense digraphs.

\begin{DEFN}\label{def:hub}
Let $k$ be an integer and $D$ be a digraph. A \emph{$k$-hub} $\mathcal{H}$ in $D$ is a 5-tuple $(E_{\rm hub}, A_0 , B_0 , U_o, U_i)$ that consists of a set $E_{\rm hub} \subseteq E(D)$, two sets $A_0 , B_0 \subseteq V(D)$ with $|A_0|=|B_0|=k$, and subsets $U_o , U_i \subseteq V(D)$ satisfying the following.

\begin{itemize} 
\item[$(\rm H1)$] $A_0 =: \left \{a_1 , \dots , a_k \right \}$, $B_0 =: \left \{b_1 , \dots , b_k \right \}$ and $A_0 \cap B_0 = \emptyset$.
\item[$(\rm H2)$] For every $t \in [k]$ and $S \subseteq V(D)$ with $|S| \leq k-1$, if $u \in U_o \setminus S$ and $a_t \notin S$, then $D-S$ contains a path from $u$ to $a_t$ only using edges in $E_{\rm hub}$.
\item[$(\rm H3)$] For every $t \in [k]$ and $S \subseteq V(D)$ with $|S| \leq k-1$, if $v \in U_i \setminus S$ and $b_t \notin S$, then $D-S$ contains a path from $b_t$ to $v$ only using edges in $E_{\rm hub}$.
\end{itemize}
\end{DEFN}

We also define an edge-version of hubs.

\begin{DEFN}\label{def:hub_arc}
Let $k$ be an integer and $D$ be a directed multigraph. A \emph{$k$-arc-hub} $\mathcal{H}$ in $D$ is a 5-tuple $(E_{\rm hub}, A_0 , B_0 , U_o, U_i)$ that consists of a set $E_{\rm hub} \subseteq E(D)$, two sets $A_0 , B_0 \subseteq V(D)$ with $|A_0|=|B_0|=k$, and subsets $U_o , U_i \subseteq V(D)$ satisfying the following.

\begin{itemize} 
\item[$(\rm H1')$] $A_0 =: \left \{a_1 , \dots , a_k \right \}$, $B_0 =: \left \{b_1 , \dots , b_k \right \}$ and $A_0 \cap B_0 = \emptyset$.
\item[$(\rm H2')$] For every $t \in [k]$, $u \in U_o$ and $F \subseteq E(D)$ with $|F| \leq k-1$, the subgraph $D-F$ contains a path from $u$ to $a_t$ only using edges in $E_{\rm hub}$.
\item[$(\rm H3')$] For every $t \in [k]$, $v \in U_i$ and $F \subseteq E(D)$ with $|F| \leq k-1$, the subgraph $D-F$ contains a path from $b_t$ to $v$ only using edges in $E_{\rm hub}$.
\end{itemize}
\end{DEFN}

The following lemma guarantees the existence of a $k$-hub under some conditions for dense digraphs.
 
\begin{LEM}\label{lem:hubexist}
Let $d,k,m,t_1 , t_2 \geq 1$, $\overline{\Delta},w \geq 0$ be integers with $d \geq 6m + 5\overline{\Delta}$ and a real number $u \geq \frac{d}{15}$. Let $D$ be a  digraph with $\overline{\Delta}(D) \leq \overline{\Delta}$ and at least $10m$ vertices. If $D$ contains a $(t_1 , t_2 , d, m, u)$-trio $(\mathcal{A},\mathcal{B},O^*)$ such that
\begin{itemize} 
\item $(\mathcal{A},\mathcal{B},O^*)$ satisfies the assumptions in Lemma~\ref{lem:trioexist},
\item $\mathcal{A}$ consists of 5-indominators $\left \{(D_i , A_i , x_i , a_i ) \right \}_{i=1}^{m}$, and
\item $\mathcal{B}$ consists of 5-outdominators $\left \{(D_i' , B_i , x_i' , b_i ) \right \}_{i=1}^{m}$.
\end{itemize}
then for every $W_o, W_i \subseteq V(D) \setminus (\bigcup_{i=1}^{m}{A_i} \cup \bigcup_{i=1}^{m}{B_i} \cup O^*)$ with $|W_o|,|W_i| \leq w$, then $D$ satisfies the following.

\begin{itemize} 
\item[$(\rm 1)$] If $m \geq t_1 + t_2 + k$, then there is $E_{\rm conn} \subseteq E(D)$ with $|E_{\rm conn}| \leq 6w(m-t_1-t_2)$ such that  for every $S \subseteq V(D)$ with $|S| \leq k-1$, if $u \in W_o \setminus S$, then there is $t \in [m]$ such that $D-S$ contains a path from $u$ to $a_t$ only using edges in $E_{\rm conn}$, and if $v \in W_i \setminus S$ then there is $t' \in [m]$ such that $D-S$ contains a path from $b_{t'}$ to $v$ only using edges in $E_{\rm conn}$. %Moreover, $E_{\rm conn}$ can be found in $O(mnw)$.

\item[$(\rm 2)$] If $m > 2t_1 + 2t_2 + 3k + \overline{\Delta} - 2$, then $D$ contains a $k$-hub 
$$\mathcal{H} := (E_{\rm hub} , \left \{a_1 , \dots , a_k \right \} , \left \{b_1 , \dots , b_k \right \} , W_o , W_i)$$ 
with $|E_{\rm hub}| \leq 2km + 6w(m-t_1-t_2)$. 
%Moreover, $E_{\rm hub}$ can be found in $O(mnw + km)$.
\end{itemize}
\end{LEM}
\begin{proof}
Since $D$ is a digraph with $\overline{\Delta}(D) \leq \overline{\Delta}$ and $|V(D)| \geq 10m$, there is $(\mathcal{A},\mathcal{B},O^*)$ such that 
\begin{align}\label{def:trio'}
\textrm{$(\mathcal{A},\mathcal{B},O^*)$ is a $\left (t_1 , \: t_2 ,\: d ,\: m ,\: u \right )$-trio in $D$,}
\end{align}
by Lemma~\ref{lem:trioexist}, where $\mathcal{A}$ consists of $m$ distinct 5-indominators $\left \{(D_i , A_i , x_i , a_i ) \right \}_{i=1}^{m}$, $\mathcal{B}$ consists of $m$ distinct 5-outdominators $\left \{(D_i' , B_i , x_i' , b_i ) \right \}_{i=1}^{m}$, $|O^*| < \frac{2mu}{t_1}$ if $t_2 \geq \overline{\Delta}$ and otherwise $|O^*| \leq \frac{2mu}{t_1} + \frac{10\overline{\Delta} m}{t_2}$.

Let $A := \bigcup_{i=1}^{m}A_i$ and $B := \bigcup_{i=1}^{m}{B_i}$. For $1 \leq i \leq m$, let 
$$U_i^+ := \bigcap_{v \in A_i}N_{D_i}^+(v) \setminus \bigcup_{v \in A_i}N_{D_i}^-(v) \:\:, \:\: U_i^- := \bigcap_{v \in B_i}N_{D_i'}^-(v) \setminus \bigcup_{v \in B_i}N_{D_i'}^+(v).$$

For each $i \in [m]$, let $F_i^+ \subseteq V(D_i) \setminus A_i$ be the set of vertices in $V(D_i) \setminus A_i$ that are not in-dominated by $A_i$ and not in $U_i^+$, and $F_i^- \subseteq V(D_i') \setminus B_i$ be the set of vertices in $V(D_i') \setminus B_i$ that are not out-dominated by $B_i$ and not in $U_i^-$. Since every vertex of $D$ has at most $\overline{\Delta}$ non-neighbours and each $|A_i|,|B_i| \leq 5$ for $i \in [m]$, we have
\begin{align}
|A|,|B| &\leq 5m.\label{eqn:AB}\\
|F_i^+|,|F_i^-| &\leq 5 \overline{\Delta}.\label{eqn:ni}
\end{align}

Let $W_o$ and $W_i$ be any subsets of $V(D) \setminus (A \cup B \cup O^*)$ with $|W_o|,|W_i| \leq w$. For each $u \in W_o$, let $I_0^+(u)$ be the set of indices $i \in [m]$ such that $A_i$ in-dominates $u$, and $I_1^+(u) \subseteq [m] \setminus I_0^+(u)$ be the set of indices $i \in [m] \setminus I_0^+(u)$ such that $u \in U_i^+$ and $|U_i^+| \geq u$. Let $S^+(u) := \left \{a_i \: : \: i \in I_0^+(u) \cup I_1^+(u) \right \}$. By $(\rm T5)$, we have $|S^+(u)| \geq m - t_1 - t_2$. By removing some elements in $I_0^+(u)$ and $I_1^+(u)$, we may assume that
\begin{align}\label{eqn:sizes^+}
|S^+(u)| = m - t_1 - t_2.
\end{align}

Now we construct a $|S^+(u)|$-fan $\{P^+_{u,i} \}_{i \in I_0^+(u) \cup I_1^+(u)}$ from $u$ to $S^+(u)$ as follows. For each $i \in I_0^+(u)$, since $A_i$ in-dominates $u$ we pick any vertex $u_i \in A_i \cap N_D^{+}(u)$. If $u_i \ne a_i$, then we can define $P^+_{u,i}$ to be the path $(u , u_i , a_i)$ since $D[A_i]$ contains a spanning transitive tournament by ($\rm ID1$) and~\eqref{def:trio'}. If $u_i = a_i$, then we define $P^+_{u,i}$ to be the path $(u , a_i)$. 

For each $i \in I_1^+(u)$, we have $d \geq 6m + 5\overline{\Delta}$ by the assumption of the lemma. By $(\rm T7)$,~\eqref{eqn:AB} and~\eqref{eqn:ni},
\begin{align*}
|N_{D_i}^{+}(u)| \geq d + |U_i^+| \geq 6m + 5\overline{\Delta} + |U_i^+| \geq m + |U_i^+| + |A| + |F_i^+|
\end{align*}
Thus we may choose $u_i \in N_{D_i}^+(u) \setminus (A \cup U_i^+ \cup F_i^+)$ for each $i \in I_1^+(u)$, so that $u_i \ne u_j$ for two distinct $i,j \in I_0^+(u) \cup I_1^+(u)$ as $|I_0^+(u) \cup I_1^+(u)| \leq m$. 

For each $i \in I_1^+(u)$, $u_i \in V(D_i) \setminus (A_i \cup U_i^+ \cup F_i^+)$ by $(\rm T1)$. This shows that $u_i$ is in-dominated by $A_i$ in $D_i$ and thus we can pick any $u_i' \in N_{D_i}^{+}(u_i) \cap A_i$. If $u_i' \ne a_i$, then we define $P^+_{u,i}$ to be the path $(u , u_i , u_i' , a_i)$, otherwise we define $P^+_{u,i}$ to be the path $(u , u_i , a_i)$. Since $u_i \notin A$, $\left \{P^+_{u,i} \right \}_{i \in I_0^+(u) \cup I_1^+(u)}$ is an $(m-t_1-t_2)$-fan from $u$ to $S^+(u)$. Note that each path in the $|S^+(u)|$-fan is of length at most 3. %One can easily check that $\left \{P^+_{u,i} \right \}_{i \in S^+(u)}$ can be found in $O(mn)$.

Similarly, for each $v \in W_i$, let $I_0^-(v)$ be the set of $i \in [m]$ such that $B_i$ out-dominates $v$, and $I_1^-(v) := [m] \setminus I_0^-(v)$ be the set of indices $i \in [m] \setminus I_0^-(v)$ such that $v \in U_i^-$ and $|U_i^-| \geq u$. Let $S^-(v) := \left \{b_i \: : \: i \in I_0^-(v) \cup I_1^-(v) \right \}$. By $(\rm T6)$, we have $|S^-(v)| \geq m-t_1 - t_2$. By removing some elements in $I_0^-(v)$ and $I_1^-(v)$, we may assume that
\begin{align}\label{eqn:sizes^-}
|S^-(v)| = m - t_1 - t_2.
\end{align}

Now we construct a $|S^-(v)|$-fan $\{P^-_{v,i} \}_{i \in I_0^-(v) \cup I_1^-(v)}$ from $S^-(v)$ to $v$. For each $i \in I_0^-(v)$, since $B_i$ out-dominates $u$ we pick any vertex $v_i \in B_i \cap N_D^{-}(u)$. If $v_i \ne b_i$, then we can define $P^-_{v,i}$ to be the path $(b_i , v_i , u)$ since $D[B_i]$ contains a spanning transitive tournament by ($\rm OD1$) and~\eqref{def:trio'}. If $v_i = b_i$, then we define $P^-_{v,i}$ to be the path $(b_i , v)$. 

For each $i \in I_1^-(v)$, we have $d \geq 6m + 5\overline{\Delta}$ by the assumption of the lemma. By $(\rm T8)$,~\eqref{eqn:AB} and~\eqref{eqn:ni},
\begin{align*}
|N_{D_i'}^{-}(v)| \geq d + |U_i^-| \geq 6m + 5\overline{\Delta} + |U_i^-| \geq m + |U_i^-| + |B| + |F_i^-|
\end{align*}
Thus we may choose $v_i \in N_{D_i'}^-(u) \setminus (B \cup U_i^- \cup F_i^-)$ for each $i \in I_1^-(v)$, so that $v_i \ne v_j$ for two distinct $i,j \in I_0^-(v) \cup I_1^-(v)$ as $|I_0^-(v) \cup I_1^-(v)| \leq m$. 

For each $i \in I_1^-(v)$, $v_i \in V(D_i') \setminus (B_i \cup U_i^- \cup F_i^-)$ by $(\rm T2)$. This shows that $v_i$ is out-dominated by $B_i$ in $D_i'$ and thus we can pick any $v_i' \in N_{D_i'}^{-}(v_i) \cap B_i$. If $v_i' \ne b_i$, then we define $P^-_{v,i}$ to be the path $(b_i , v_i' , v_i , v)$, otherwise we define $P^-_{v,i}$ to be the path $(b_i , v_i , v)$. Since $v_i \notin A \cup B$, $\left \{P^-_{v,i} \right \}_{i \in I_0^-(v) \cup I_1^-(v)}$ is an $(m-t_1-t_2)$-fan from $S^-(v)$ to $v$. Note that each path in the $|S^-(v)|$-fan is of length at most 3.

Now we prove $(\rm 1)$. For $m \geq t_1 + t_2 + k$, let us define 
\begin{align*}
E_{\rm conn} := \bigcup_{u \in W_o} \bigcup_{i \in S^+(u)}{E(P^+_{u,i})} \cup \bigcup_{v \in W_i} \bigcup_{i \in S^-(v)}{E(P^-_{v,i})}.
\end{align*}
%which can be constructed in $O(mnw)$ since $\left \{P^+_{u,i} \right \}_{i \in S^+(u)}$ and $\left \{P^-_{v,i} \right \}_{i \in S^-(v)}$ can be constructed in $O(mn)$ for each $u \in W_o$ and $v \in W_i$.

By $|W_o|,|W_i| \leq w$,~\eqref{eqn:sizes^+}, and~\eqref{eqn:sizes^-}, we have 
\begin{align}\label{eqn:e_conn}
|E_{\rm conn}| \leq 6w(m-t_1-t_2). 
\end{align}

For every $S \subseteq V(D)$ with $|S| \leq k-1$, since for $u \in W_o$, $|S^+(u)| \geq m-t_1 -t_2 \geq k$ and for $v \in W_i$, $|S^-(v)| \geq m-t_1 - t_2 \geq k$, there are $t \in I_0^+(u) \cup I_1^+(u)$ with $V(P^+_{u,t}) \cap S = \emptyset$. Similarly, there is $t' \in I_0^-(v) \cup I_1^-(v)$ with $V(P^-_{v,t'}) \cap S = \emptyset$. This proves $(\rm 1)$.

Now we prove $(\rm 2)$. Let us assume that $m \geq 2t_1 + 2t_2 + 3k + \overline{\Delta} - 2$. Note that $m \geq t_1 + t_2 + k$ and thus $(\rm 1)$ is satisfied. Let us define
\begin{align*}
E_{\rm hub} := E_D (\left \{a_1 , \dots , a_k \right \} , \left \{a_1 , \dots , a_{m} \right \}) \cup E_D (\left \{b_1 , \dots , b_k \right \} , \left \{b_1 , \dots , b_{m} \right \}) \cup E_{\rm conn}.
\end{align*}
%which can be constructed in $O(mnw + km)$.

By~\eqref{eqn:e_conn}, we have
\begin{align*}
|E_{\rm hub}| \leq 2km + |E_{\rm conn}| \leq 2km + 6w(m-t_1-t_2).
\end{align*}

We prove that $(E_{\rm hub} , \left \{a_1 , \dots , a_k \right \} , \left \{b_1 , \dots , b_k \right \} , W_o , W_i)$ satisfies $(\rm H2)$. Let $S \subseteq V(D)$ be a set of at most $k-1$ vertices. For $t \in [k]$ with $a_t \notin S$ and $u \in W_o \setminus S$, it follows that $a_t$ has at least $\frac{m-k-\overline{\Delta}}{2}$ in-neighbours in $D[\left \{a_1 , \dots , a_m \right \}]$ by ($\rm T4$) and~\eqref{def:trio'}. There is a $|S^+(u)|$-fan from $u$ to $S^+(u) \subseteq A_0$ and $|S^+(u)| = m - t_1 - t_2$ by~\eqref{eqn:sizes^+}, it follows that there are at least $m - t_1 - t_2 - k + 1$ $i$'s with $i \in I_0^+(u) \cup I_1^+(u)$ and $V(P^+_{u,i}) \cap S = \emptyset$. Since $m > 2t_1 + 2t_2 + 3k + \overline{\Delta} - 2$ by the assumption of the lemma, we have 
\begin{align*}
|N_{D[\left \{a_1 , \dots , a_m \right \}]}^- (a_t)| + |S^+(u)| - |S| \geq \frac{m-k-\overline{\Delta}}{2} + (m-t_1-t_2) - (k-1) > m
\end{align*}
and by pigeonhole principle, there is $i \in I_0^+(u) \cup I_1^+(u)$ with $V(P^+_{u,i}) \cap S = \emptyset$ and $a_i \in  N_D^-(a_t)$. Then $P := P^+_{u,i} \cup (a_i , a_t)$ is a path from $u$ to $a_t$ that does not intersect with $S$. Note that $E(P) \subseteq E_{\rm hub}$, as $P^+_{u,i} \subseteq E_{\rm hub}$ and $a_i a_t \in E_{\rm hub}$. The proof of $(\rm H3)$ is similar.
\end{proof}

The following lemma guarantees a $k$-arc-hub for dense digraphs under some conditions. Since the proof is almost identical to the proof of Lemma~\ref{lem:hubexist} except for a few parts, we only sketch the proof. The proof differs from the proof of Lemma~\ref{lem:hubexist} for two parts: for every $i \in I_1^+(u)$, we choose each $u_i \in N_{D_i}^+(u) \setminus (U_i^+ \cup F_i^+)$ which may be in $A$, since the paths in $|S^+(u)|$-fan are not necessarily vertex-disjoint. Similarly, for $i \in I_1^-(v)$, we choose $v_i \in N_{D_i'}^-(u) \setminus (U_i^- \cup F_i^-)$ which may be in $B$, since the paths in $|S^-(v)|$-fan are not necessarily vertex-disjoint. Therefore, we only need $d \geq m + 5 \overline{\Delta}$ in the assumption. As the rest of the proof is identical, and we omit the  proof. 

\begin{LEM}\label{lem:hubexist_arc}
Let $d,k,m,t_1 , t_2 \geq 1$, $\overline{\Delta},w \geq 0$ be integers with $d \geq m + 5\overline{\Delta}$ and a real number $u \geq \frac{d}{15}$. Let $D$ be a directed multigraph with $\overline{\Delta}(D) \leq \overline{\Delta}$ and at least $10m$ vertices. If $D$ contains a $(t_1 , t_2 , d, m, u)$-trio $(\mathcal{A},\mathcal{B},O^*)$ such that
\begin{itemize} 
\item $(\mathcal{A},\mathcal{B},O^*)$ satisfies the assumptions in Lemma~\ref{lem:trioexist},
\item $\mathcal{A}$ consists of 5-indominators $\left \{(D_i , A_i , x_i , a_i ) \right \}_{i=1}^{m}$, and
\item $\mathcal{B}$ consists of 5-outdominators $\left \{(D_i' , B_i , x_i' , b_i ) \right \}_{i=1}^{m}$.
\end{itemize}
then for any $W_o, W_i \subseteq V(D) \setminus (\bigcup_{i=1}^{m}{A_i} \cup \bigcup_{i=1}^{m}{B_i} \cup O^*)$ with $|W_o|,|W_i| \leq w$, then $D$ satisfies the following.

\begin{itemize} 
\item[$(\rm 1)$] If $m \geq t_1 + t_2 + k$, then there is $E_{\rm conn} \subseteq E(D)$ with $|E_{\rm conn}| \leq 6w(m-t_1-t_2)$ such that  for every $F \subseteq E(D)$ with $|F| \leq k-1$, if $u \in W_o$ then there is $t \in [m]$ such that $D-F$ contains a path from $u$ to $a_t$ only using edges in $E_{\rm conn}$, and if $v \in W_i$ then there is $t' \in [m]$ such that $D-F$ contains a path from $b_{t'}$ to $v$ only using edges in $E_{\rm conn}$. %Moreover, $E_{\rm conn}$ can be found in $O(mnw)$.

\item[$(\rm 2)$] If $m > 2t_1 + 2t_2 + 3k + \overline{\Delta} - 2$, then $D$ contains a $k$-arc-hub 
$$\mathcal{H} := (E_{\rm hub} , \left \{a_1 , \dots , a_k \right \} , \left \{b_1 , \dots , b_k \right \} , W_o , W_i)$$ 
with $|E_{\rm hub}| \leq 2km + 6w(m-t_1-t_2)$. 
%Moreover, $E_{\rm hub}$ can be found in $O(mnw + km)$.
\end{itemize}
\end{LEM}

\subsection{Absorbers}
In this subsection, we consider objects called absorbers. Roughly speaking, even though we remove few vertices from a digraph, we can connect vertices to a small set of vertices by a path in an absorber. This plays an important role in preserving the vertex-connectivity in a spanning subgraph, and finding sparse absorbers are directly related to finding highly connected sparse spanning subgraphs.

\begin{DEFN}\label{def:abs}
Let $k\geq 1$ be an integer and $D$ be a digraph. A \emph{$k$-absorber} is a 5-tuple $(E_{\rm abs} , V_{\rm ex} , \mathcal{P} , W_i , W_o)$ that consists of a set $E_{\rm abs} \subseteq E(D)$, a set $V_{\rm ex} \subseteq V(D)$, a collection $\mathcal{P} = \left \{P_i \right \}_{i=1}^{k}$ of $k$ vertex-disjoint paths, and sets $W_i,W_o \subseteq V(D)$ satisfying the following.

\begin{itemize} 
\item[$(\rm A1)$] For every $t \in [k]$, both endvertices of $P_t$ are in $V_{\rm ex}$.
\item[$(\rm A2)$] $\bigcup_{t=1}^{k}E(P_t) \subseteq E_{\rm abs}$.
\item[$(\rm A3)$] For every $S \subseteq V(D)$ with $|S| \leq k-1$ and $u \in V(D) \setminus S$, the subgraph $D-S$ has a path from $u$ to a vertex in $W_o \setminus S$ only using edges in $E_{\rm abs}$.
\item[$(\rm A4)$] For every $S \subseteq V(D)$ with $|S| \leq k-1$ and $v \in V(D) \setminus S$, the subgraph $D-S$ has a path from a vertex in $W_i \setminus S$ to $v$ only using edges in $E_{\rm abs}$.
\end{itemize}
\end{DEFN}

We also define an edge-version of absorbers.

\begin{DEFN}\label{def:abs_arc}
Let $k\geq 1$ be an integer and $D$ be a directed multigraph. A \emph{$k$-arc-absorber} is a 5-tuple $(E_{\rm abs} , V_{\rm ex} , \mathcal{P} , W_i , W_o)$ that consists of a set $E_{\rm abs} \subseteq E(D)$, a set $V_{\rm ex} \subseteq V(D)$, a collection $\mathcal{P} = \left \{P_i \right \}_{i=1}^{k}$ of $k$ edge-disjoint paths, and sets $W_i,W_o \subseteq V(D)$ satisfying the following.

\begin{itemize} 
\item[$(\rm A1')$] For each $t \in [t]$, both endvertices of $P_t$ are in $V_{\rm ex}$.
\item[$(\rm A2')$] $\bigcup_{t=1}^{k}E(P_t) \subseteq E_{\rm abs}$.
\item[$(\rm A3')$] For every $F \subseteq E(D)$ with $|F| \leq k-1$ and $u \in V(D)$, the subgraph $D-F$ has a path from $u$ to a vertex in $W_o$ using only edges in $E_{\rm abs}$.
\item[$(\rm A4')$] For every $F \subseteq E(D)$ with $|F| \leq k-1$ and $v \in V(D)$, the subgraph $D-F$ has a path from a vertex in $W_i$ to $v$ using only edges in $E_{\rm abs}$.
\end{itemize}
\end{DEFN}

The following lemma guarantees the existence of a $k$-absorber that uses only few edges in dense digraphs.

\begin{LEM}\label{lem:abs}
Let $k,n \geq 1$ and $\overline{\Delta} \geq 0$ be integers, and $D$ be a strongly $k$-connected $n$-vertex digraph with $\overline{\Delta}(D) \leq \overline{\Delta}$. Let $V_{\rm ex} \subseteq V(D)$ with $|V(D) \setminus V_{\rm ex}| \geq 39k + 38\overline{\Delta}$, and $\mathcal{P}$ be a collection of $k$ vertex-disjoint paths $\left \{P_1 , \dots , P_k \right \}$ such that $P_i$ is a minimal path with endvertices in $V_{\rm ex}$ for every $i \in [k]$. 

Then $D$ has a $k$-absorber $\mathcal{D} = (E_{\rm abs} , V_{\rm ex} , \mathcal{P} , W_i , W_o)$ satisfying the following.

\begin{itemize} 
\item[$(\rm 1)$] $W_i,W_o \subseteq V(D) \setminus V_{\rm ex}$ and $|W_i|,|W_o| = 3k$.
\item[$(\rm 2)$] $|E_{\rm abs}| \leq kn + 226k(k+\overline{\Delta}) + 38(k+\overline{\Delta}) + (5k+1)|V_{\rm ex}|$. 
\end{itemize}
%Moreover, $\mathcal{D}$ can be found in $O(n^3 + kn^{2.5} + k(k+\overline{\Delta})n)$.
\end{LEM}
\begin{proof}
For $t \in [k]$, let us define $E_\textrm{path} := \bigcup_{t=1}^{k} E(P_t)$ and $D' := D - V_{\rm ex}$.

Since $|V(D')| \geq 39k + 38\overline{\Delta} \geq 10 \cdot 3k$, by applying Lemma~\ref{lem:trioexist} to $D'$ we deduce that
\begin{align}\label{def:trio'''}
\textrm{there is a $\left (k,\:k,\:18k + 5\overline{\Delta},\:3k,\:\frac{18k + 5\overline{\Delta}}{15} \right )$-trio $(\mathcal{A}',\mathcal{B}',S^*)$ in $D'$,}
\end{align}
where $\mathcal{A}'$ consists of $3k$ distinct 5-indominators $\left \{(D_i , A_i' , y_i , a_i' ) \right \}_{i=1}^{3k}$, $\mathcal{B}$ consists of $3k$ distinct 5-outdominators $\left \{(D_i' , B_i' , y_i' , b_i' ) \right \}_{i=1}^{3k}$, and $|S^*| \leq 8k + 32\overline{\Delta}$. %Note that $(\mathcal{A}',\mathcal{B}',S^*)$ can be found in $O(kn)$ by Lemma~\ref{lem:trioexist}.

Let us define 
$$A' := \bigcup_{i=1}^{3k}A_i', \:\:B' := \bigcup_{i=1}^{3k}B_i', \:\:V_{\rm ex}' := V_{\rm ex} \cup A' \cup B' \cup S^*,$$
and $V_i^+ := \bigcap_{v \in A_i'}N_{D_i}^+(v) \setminus \bigcup_{v \in A_i'}N_{D_i}^-(v)$, $V_i^- := \bigcap_{v \in B_i'}N_{D_i'}^-(v) \setminus \bigcup_{v \in B_i'}N_{D_i'}^+(v)$ for every $i \in [3k]$.

Since $|A'|,|B'| \leq 5 \cdot 3k$, it follows that
\begin{align}
|A' \cup B'| &\leq 30k,\label{eqn:a'b'}\\
|A' \cup B' \cup S^*| &\leq 38(k+\overline{\Delta}),\label{eqn:a'b's}\\
|V_{\rm ex}'| &\leq |V_{\rm ex}| + 38(k+\overline{\Delta}). \label{eqn:v_ex'}
\end{align}

Since 
\begin{align*}
|V(D) \setminus V_{\rm ex}'| &\geq |V(D) \setminus V_{\rm ex}| - |A' \cup B' \cup S^*| \\
& \geq  39k + 38\overline{\Delta} - 38(k+\overline{\Delta}) \geq k, 
\end{align*}
by applying  Lemma~\ref{lem:escape} to a set $V_{\rm ex}'$, there is a $k$-escaper $(E_{\rm escape}, V_{\rm ex}' , V_{\rm out})$ with $V_{\rm out} \subseteq V(D) \setminus V_{\rm ex}'$ such that
\begin{align}
|V_{\rm out}| &\leq  2k|V_{\rm ex}'| \leq 2k|V_{\rm ex}| + 76k(k+\overline{\Delta}) \label{eqn:v_out}\\
|E_{\rm escape}| &\leq 4k|V_{\rm ex}'| \leq 4k|V_{\rm ex}| + 152k(k+\overline{\Delta}) \label{eqn:e_escape}.
\end{align}

Let us define
\begin{align}
X_1' &:= \bigcup_{i=1}^{k}V^\textrm{int}(P_i) \setminus (V_{\rm ex}' \cup V_{\rm out})\label{def:x_1}\\
X_1 &:= V(D) \setminus (V_{\rm ex}' \cup V_{\rm out} \cup X_1')\label{def:x_1'}
\end{align}

\begin{CLAIM}\label{claim:conn2'}
There exist sets $U_i^0 , U_o^0 \subseteq V_{\rm out}$, a set $E_0 \subseteq E(D)$, sets $U_i^1,U_o^1 \subseteq X_1$, a set $E_1 \subseteq E(D)$, sets $U_i'^1 , U_o'^1 \subseteq X_1'$ and a set $E_1'\subseteq E(D)$ satisfying the following.

\begin{itemize} 
\item[$(\rm 1)$] $|E_0| \leq k |V_{\rm out}| - k + k \overline{\Delta}$.
\item[$(\rm 2)$] There are $U_i^0, U_o^0 \subseteq V_{\rm out}$ such that $|U_i^0|, |U_o^0| \leq 2k+\overline{\Delta}-1$ and for every $S \subseteq V(D)$ with $|S| \leq k-1$ and for every $u,v \in V_{\rm out} \setminus S$, the subgraph $D-S$ has a path from $u$ to a vertex in $U_o^0 \setminus S$, and a path from a vertex in $U_i^0 \setminus S$ to $v$ such that both paths only use edges in $E_0$.

\item[$(\rm 3)$] $|E_1| \leq k |X_1| - k + k \overline{\Delta}$.
\item[$(\rm 4)$] There are $U_i^1, U_o^1 \subseteq X_1$ such that $|U_i^1|, |U_o^1| \leq 2k+\overline{\Delta}-1$ and for every $S \subseteq V(D)$ with $|S| \leq k-1$ and for every $u,v \in X_1 \setminus S$, the subgraph $D-S$ has a path from $u$ to a vertex in $U_o^1 \setminus S$, and a path from a vertex in $U_i^1 \setminus S$ to $v$ such that both paths only use edges in $E_1$.

\item[$(\rm 5)$] $|E_1'| \leq (k-1)|X_1'| + (\overline{\Delta}+1)(k-1)$.
\item[$(\rm 6)$] There are $U_i'^1, U_o'^1 \subseteq X_1'$ such that $|U_i'^1|,|U_o'^1| \leq 2k+\overline{\Delta}-1$ and for every $S \subseteq V(D)$ with $|S| \leq k-1$ and for every $u,v \in X_1' \setminus S$, the subgraph $D-S$ has a path from $u$ to a vertex in $(U_o'^1 \cup V_\textrm{ex}) \setminus S$ , and a path from a vertex in $(U_i'^1 \cup V_\textrm{ex}) \setminus S$ to $v$ such that both paths only use edges in $E_\textrm{path} \cup E_1'$.
\end{itemize}
%Moreover, these sets can be found in $O(n^3 + kn^{2.5})$.
\end{CLAIM}
\begin{proof}[Proof of Claim~\ref{claim:conn2'}]
By applying Lemma~\ref{lem:digraph1'} to $D[V_{\rm out}]$ and $D[X_1]$, (1),(2),(3), and (4) follows. Similarly, applying Lemma~\ref{lem:digraph2'} to $D[X_1']$, (5) and (6) follows.
\end{proof}

Let us define 
\begin{align}
U_o := U_o^0 \cup U_o^1 \cup U_o'^1 , \:\:\: U_i := U_i^0 \cup U_i^1 \cup U_i'^1. \label{def:w_1}
\end{align}

Then $|U_i|,|U_o| \leq 3(2k+\overline{\Delta})$.

\begin{CLAIM}\label{claim:e_conn}
There is a set $E_{\rm conn} \subseteq E(D')$ of edges satisfying the following.
\begin{itemize} 
\item[$(1)$] $|E_{\rm conn}| \leq 18k(2k+\overline{\Delta})$.
\item[$(2)$] For every $S \subseteq V(D)$ with $|S| \leq k-1$ and $u \in U_o \setminus S$, there is $t \in [3k]$ such that $D' - S$ contains a path from $u$ to $a_t'$, only using edges in $E_{\rm conn}$.
\item[$(3)$] For every $S \subseteq V(D)$ with $|S| \leq k-1$ and $v \in U_i \setminus S$, there is $t \in [3k]$ such that $D' - S$ contains a path from $b_t'$ to $v$, only using edges in $E_{\rm conn}$.
\end{itemize}
%Moreover, $E_{\rm conn}$ can be found in $O(k(k+\overline{\Delta})n)$.
\end{CLAIM}
\begin{proof}
Note that $U_o , U_i \subseteq V(D) \setminus V_{\rm ex}' \subseteq V(D')$. By~\eqref{def:trio'''}, $(\mathcal{A}',\mathcal{B}',S^*)$ satisfies the requirements of Lemma~\ref{lem:hubexist}, hence the claim follows by $(\rm 1)$ of Lemma~\ref{lem:hubexist}.
\end{proof}

Now let us define
\begin{align}
E_{\rm abs} &:= E_{\rm path} \cup E_{\rm escape} \cup E_0 \cup E_1 \cup E_1' \cup E_{\rm conn}, \label{def:e_abs} \\
W_o &:= \left \{a_1' , \dots , a_{3k}' \right \}\\
W_i &:= \left \{b_1' , \dots , b_{3k}' \right \}.
\end{align}

Then $W_o,W_i \subseteq V(D') = V(D) \setminus V_{\rm ex}$. Since $\bigcup_{t=1}^{k}{\rm Int}(P_t) \subseteq V_{\rm ex}' \cup V_{\rm out} \cup X_1'$ , we have $|E_{\rm path}| \leq |V_{\rm ex}'| + |V_{\rm out}| + |X_1'| + k$ by~\eqref{eqn:v_ex'}. 

Note that $V(D) = V_{\rm ex}' \cup V_{\rm out} \cup X_1 \cup X_1'$ by~\eqref{def:x_1'}. By~\eqref{eqn:v_ex'},~\eqref{eqn:v_out},~\eqref{eqn:e_escape}, Claim~\ref{claim:conn2'}, Claim~\ref{claim:e_conn}, and $V(D) = V_{\rm ex}' \cup V_{\rm out} \cup X_1 \cup X_1'$ we have
\begin{align}
|E_{\rm abs}| &\leq  |E_{\rm escape}| + |E_{\rm path}| + |E_0| + |E_1| + |E_1'| + |E_{\rm conn}|\nonumber \\
& \leq  4k|V_{\rm ex}'| + (|V_{\rm ex}'| + |V_{\rm out}| + |X_1'| + k) + (k|V_{\rm out}| - k + k \overline{\Delta}) + (k|X_1| - k + k \overline{\Delta})\nonumber\\
&\qquad  + ((k-1)|X_1'| + k-1 + k \overline{\Delta}) + 18k(2k+\overline{\Delta})\nonumber \\
& \leq  k(|V_{\rm ex}'| + |V_{\rm out}| + |X_1| + |X_1'|) + (3k+1)|V_{\rm ex}'| + |V_{\rm out}| + 3k\overline{\Delta} + 18k(2k+\overline{\Delta})\nonumber\\
& \leq  kn + (3k+1)|V_{\rm ex}| + 114k(k+\overline{\Delta}) + 38(k+\overline{\Delta}) + |V_{\rm out}| + 36k(k+\overline{\Delta}) \nonumber\\
& \leq  kn + 226k(k+\overline{\Delta}) + 38(k+\overline{\Delta}) + (5k+1)|V_{\rm ex}|. \label{eqn:e_abs*}
\end{align}

Let us define
$$\mathcal{D} := (E_{\rm abs} , V_{\rm ex} , \mathcal{P} , W_i , W_o).$$

\begin{CLAIM}\label{claim:abs_end}
$\mathcal{D}$ is a $k$-absorber in $D$.
\end{CLAIM}
\begin{proof}
Both $(\rm A1)$ and $(\rm A2)$ are clear. Let $S \subseteq V(D)$ with $|S| \leq k-1$, and $u,v \in V(D) \setminus S$ be two distinct vertices.

\begin{itemize} 
\item[$(\rm a)$] If $u \in V_{\rm ex}'$, then since $(E_{\rm escape} , V_{\rm ex}' , V_{\rm out})$ is a $k$-escaper, there is a path from $u$ to $u' \in V_{\rm out}$ in $D-S$ using only edges in $E_{\rm escape}$, and there is a path from $u'$ to a vertex $u'' \in U_o$ in $D-S$ only using edges in $E_0$ by Claim~\ref{claim:conn2'}. By Claim~\ref{claim:e_conn}, there is a path from $u''$ to a vertex $u^+ \in W_o$ in $D-S$ only using edges in $E_{\rm conn}$.

\item[$(\rm b)$] If $u \in X_1'$, then there is a path from $u$ to $u' \in U_o \cup V_{\rm ex}$ in $D-S$ only using edges in $E_{\rm path} \cup E_1'$ by Claim~\ref{claim:conn2'}. If $u' \in U_o$, then there is a path from $u'$ to a vertex $u^+ \in W_o$ in $D-S$ only using edges in $E_{\rm conn}$ by Claim~\ref{claim:e_conn}. Otherwise if $u' \in V_{\rm ex} \setminus S$, then there is a path from $u'$ to a vertex $u^+ \in W_o$ in $D-S$ only using edges in $E_{\rm abs}$ by $(\rm a)$.

\item[$(\rm c)$] If $u \in V_{\rm out} \cup X_1$, then there is a path from $u$ to a vertex $u' \in U_o$ in $D-S$ using only edges in $E_0 \cup E_1$ by Claim~\ref{claim:conn2'}. By Claim~\ref{claim:e_conn}, there is a path from $u'$ to a vertex $u^+ \in W_o$ in $D-S$ only using edges in $E_{\rm conn}$.
\end{itemize}

Hence there is a path in $D-S$ from $u$ to $u^+ \in W_o$ only using edges in $E_{\rm abs}$, proving $(\rm A3)$. Similarly, there is a path in $D-S$ from a vertex $v^+ \in W_i$ to $v$ only using edges in $E_{\rm abs}$, proving $(\rm A4)$. This proves the claim.
\end{proof}

By Claim~\ref{claim:abs_end} and~\eqref{eqn:e_abs*}, this completes the proof of the lemma.
\end{proof}

Similarly, the following lemma guarantees the existence of a $k$-arc-absorber that uses only few edges in dense digraphs.

\begin{LEM}\label{lem:abs_arc}
Let $k,n \geq 1$ and $\overline{\Delta} \geq 0$ be integers, and $D$ be a strongly $k$-connected $n$-vertex directed multigraph with $\overline{\Delta}(D) \leq \overline{\Delta}$. Let $V_{\rm ex} \subseteq V(D)$ with $|V(D) \setminus V_{\rm ex}| \geq 33k + 32\overline{\Delta}$, and $\mathcal{P}$ be a collection of $k$ edge-disjoint paths $\left \{P_1 , \dots , P_k \right \}$ such that $P_i$ is a path with endvertices in $V_{\rm ex}$ for every $i \in [k]$. 

Then $D$ has a $k$-arc-absorber $\mathcal{D} = (E_{\rm abs} , V_{\rm ex} , \mathcal{P} , W_i , W_o)$ satisfying the following.

\begin{itemize} 
\item[$(\rm 1)$] $W_i,W_o \subseteq V(D) \setminus V_{\rm ex}$ and $|W_i|,|W_o| = 3k$.
\item[$(\rm 2)$] $|E_{\rm abs}| \leq kn + 210k(k+\overline{\Delta}) + 32(k+\overline{\Delta}) + (5k+1)|V_{\rm ex}|$. 
\end{itemize}
%Moreover, $\mathcal{D}$ can be found in $O(n^3 + kn^{2.5} + k(k+\overline{\Delta})n)$.
\end{LEM}
\begin{proof}
For $t \in [k]$, let us define $E_\textrm{path} := \bigcup_{t=1}^{k} E(P_t)$ and $D' := D - V_{\rm ex}$.

Since $|V(D')| \geq 33k + 32\overline{\Delta} \geq 10 \cdot 3k$, by applying Lemma~\ref{lem:trioexist} to $D'$ we deduce that 
\begin{align}\label{def:trio''}
\textrm{there is a $\left (k,\:k,\:3k + 5\overline{\Delta},\:3k,\:\frac{3k + 5\overline{\Delta}}{15} \right )$-trio $(\mathcal{A}',\mathcal{B}',S^*)$ in $D'$,}
\end{align}
where $\mathcal{A}'$ consists of $3k$ distinct 5-indominators $\left \{(D_i , A_i' , y_i , a_i') \right \}_{i=1}^{3k}$, $\mathcal{B}$ consists of $3k$ distinct 5-outdominators $\left \{(D_i' , B_i' , y_i' , b_i') \right \}_{i=1}^{3k}$, and $|S^*| \leq 1.2k + 32\overline{\Delta}$. %Note that $(\mathcal{A}',\mathcal{B}',S^*)$ can be found in $O(kn)$ by Lemma~\ref{lem:trioexist}.

Let us define 
$$A' := \bigcup_{i=1}^{3k}A_i', \:\:B' := \bigcup_{i=1}^{3k}B_i', \:\:V_{\rm ex}' := V_{\rm ex} \cup A' \cup B' \cup S^*,$$
and $V_i^+ := \bigcap_{v \in A_i'}N_{D_i}^+(v) \setminus \bigcup_{v \in A_i'}N_{D_i}^-(v)$, $V_i^- := \bigcap_{v \in B_i'}N_{D_i'}^-(v) \setminus \bigcup_{v \in B_i'}N_{D_i'}^+(v)$ for every $i \in [3k]$.

Since $|A'|,|B'| \leq 5 \cdot 3k$, it follows that
\begin{align}
|A' \cup B'| &\leq 30k,\label{eqn:a'b'_arc}\\
|A' \cup B' \cup S^*| &\leq 32(k+\overline{\Delta}),\label{eqn:a'b's_arc}\\
|V_{\rm ex}'| &\leq |V_{\rm ex}| + 32(k+\overline{\Delta}).\label{eqn:v_ex'_arc}
\end{align}

The rest of the proof is almost identical to the proof of Lemma~\ref{lem:abs}, except for a few parts: we use Lemma~\ref{lem:escape_arc} for $k$-arc-escapers instead of Lemma~\ref{lem:escape} for $k$-escapers. Since the paths in $\left \{P_1 , \dots , P_k \right \}$ are edge-disjoint and each $P_i$ is not necessarily minimal, as we use Lemmas~\ref{lem:digraph3'} and~\ref{lem:digraph4'} instead of Lemmas~\ref{lem:digraph1'} and~\ref{lem:digraph2'} respectively. Note that Lemma~\ref{lem:digraph4'} has a slightly worse bound than Lemma~\ref{lem:digraph2'}. Finally, we replace Lemma~\ref{lem:hubexist} by Lemma~\ref{lem:hubexist_arc} in the proof of Claim~\ref{claim:e_conn}. 

Since 
\begin{align*}
|V(D) \setminus V_{\rm ex}'| &\geq |V(D) \setminus V_{\rm ex}| - |A' \cup B' \cup S^*| \\
& \geq  33k + 32\overline{\Delta} - 32(k+\overline{\Delta}) \geq k, 
\end{align*}
by applying  Lemma~\ref{lem:escape_arc} to a set $V_{\rm ex}'$, there is a $k$-arc-escaper $(E_{\rm escape}, V_{\rm ex}' , V_{\rm out})$ with $V_{\rm out} \subseteq V(D) \setminus V_{\rm ex}'$ such that
\begin{align}
|V_{\rm out}| &\leq  2k|V_{\rm ex}'| \leq 2k|V_{\rm ex}| + 64k(k+\overline{\Delta}) \label{eqn:v_out_arc}\\
|E_{\rm escape}| &\leq 4k|V_{\rm ex}'| \leq 4k|V_{\rm ex}| + 128k(k+\overline{\Delta}) \label{eqn:e_escape_arc}.
\end{align}

Let us define
\begin{align}
X_1' &:= \bigcup_{i=1}^{k}V^\textrm{int}(P_i) \setminus (V_{\rm ex}' \cup V_{\rm out})\label{def:x_1_arc}\\
X_1 &:= V(D) \setminus (V_{\rm ex}' \cup V_{\rm out} \cup X_1')\label{def:x_1'_arc}
\end{align}

%%%%%%%

\begin{CLAIM}\label{claim:conn2'_arc}
There exist sets $U_i^0 , U_o^0 \subseteq V_{\rm out}$, a set $E_0 \subseteq E(D)$, sets $U_i^1,U_o^1 \subseteq X_1$, a set $E_1 \subseteq E(D)$, sets $U_i'^1 , U_o'^1 \subseteq X_1'$ and a set $E_1'\subseteq E(D)$ satisfying the following.

\begin{itemize} 
\item[$(\rm 1)$] $|E_0| \leq k |V_{\rm out}| - k + k \overline{\Delta}$.
\item[$(\rm 2)$] There are $U_i^0, U_o^0 \subseteq V_{\rm out}$ such that $|U_i^0|, |U_o^0| \leq 2k+\overline{\Delta}-1$ and for every $F \subseteq E(D)$ with $|S| \leq k-1$ and for every $u,v \in V_{\rm out}$, the subgraph $D-F$ has a path from $u$ to a vertex in $U_o^0$, and a path from a vertex in $U_i^0$ to $v$ such that both paths only use edges in $E_0$.

\item[$(\rm 3)$] $|E_1| \leq k |X_1| - k + k \overline{\Delta}$.
\item[$(\rm 4)$] There are $U_i^1, U_o^1 \subseteq X_1$ such that $|U_i^1|, |U_o^1| \leq 2k+\overline{\Delta}-1$ and for every $F \subseteq E(D)$ with $|F| \leq k-1$ and for every $u,v \in X_1$, the subgraph $D-F$ has a path from $u$ to a vertex in $U_o^1$, and a path from a vertex in $U_i^1 $ to $v$ such that both paths only use edges in $E_1$.

\item[$(\rm 5)$] $|E_1'| \leq (k-1)|X_1'| + (k-1)(\overline{\Delta}+2k-1)$.
\item[$(\rm 6)$] There are $U_i'^1, U_o'^1 \subseteq X_1'$ such that $|U_i'^1|,|U_o'^1| \leq 4k+\overline{\Delta}-3$ and for every $F \subseteq E(D)$ with $|F| \leq k-1$ and for every $u,v \in X_1' \setminus S$, the subgraph $D-F$ has a path from $u$ to a vertex in $U_o'^1 \cup V_\textrm{ex}$ , and a path from a vertex in $U_i'^1 \cup V_\textrm{ex}$ to $v$ such that both paths only use edges in $E_\textrm{path} \cup E_1'$.
\end{itemize}
%Moreover, these sets can be found in $O(n^3 + kn^{2.5})$.
\end{CLAIM}
\begin{proof}[Proof of Claim~\ref{claim:conn2'_arc}]
By applying Lemma~\ref{lem:digraph3'} to $D[V_{\rm out}]$ and $D[X_1]$, (1),(2),(3), and (4) follows. Similarly, applying Lemma~\ref{lem:digraph4'} to $D[X_1']$, (5) and (6) follows.
\end{proof}

Let us define 
\begin{align}
U_o := U_o^0 \cup U_o^1 \cup U_o'^1 , \:\:\: U_i := U_i^0 \cup U_i^1 \cup U_i'^1. \label{def:w_1}
\end{align}

Then $|U_i|,|U_o| \leq 8k+3\overline{\Delta}$.

\begin{CLAIM}\label{claim:e_conn_arc}
There is a set $E_{\rm conn} \subseteq E(D')$ of edges satisfying the following.
\begin{itemize} 
\item[$(1)$] $|E_{\rm conn}| \leq 6k(8k+3\overline\Delta)$.
\item[$(2)$] For every $S \subseteq V(D)$ with $|S| \leq k-1$ and $u \in U_o \setminus S$, there is $t \in [3k]$ such that $D' - S$ contains a path from $u$ to $a_t'$, only using edges in $E_{\rm conn}$.
\item[$(3)$] For every $S \subseteq V(D)$ with $|S| \leq k-1$ and $v \in U_i \setminus S$, there is $t \in [3k]$ such that $D' - S$ contains a path from $b_t'$ to $v$, only using edges in $E_{\rm conn}$.
\end{itemize}
%Moreover, $E_{\rm conn}$ can be found in $O(k(k+\overline{\Delta})n)$.
\end{CLAIM}
\begin{proof}
Note that $U_o , U_i \subseteq V(D) \setminus V_{\rm ex}' \subseteq V(D')$. By~\eqref{def:trio''}, $(\mathcal{A}',\mathcal{B}',S^*)$ satisfies the requirements of Lemma~\ref{lem:hubexist_arc}, hence the claim follows by $(\rm 1)$ of Lemma~\ref{lem:hubexist_arc}.
\end{proof}

Now let us define
\begin{align}
E_{\rm abs} &:= E_{\rm path} \cup E_{\rm escape} \cup E_0 \cup E_1 \cup E_1' \cup E_{\rm conn}, \label{def:e_abs} \\
W_o &:= \left \{a_1' , \dots , a_{3k}' \right \}\\
W_i &:= \left \{b_1' , \dots , b_{3k}' \right \}.
\end{align}

Then $W_o,W_i \subseteq V(D') = V(D) \setminus V_{\rm ex}$. Since $\bigcup_{t=1}^{k}{\rm Int}(P_t) \subseteq V_{\rm ex}' \cup V_{\rm out} \cup X_1'$ , we have $|E_{\rm path}| \leq |V_{\rm ex}'| + |V_{\rm out}| + |X_1'| + k$ by~\eqref{eqn:v_ex'_arc}. 

Note that $V(D) = V_{\rm ex}' \cup V_{\rm out} \cup X_1 \cup X_1'$ by~\eqref{def:x_1'_arc}. By~\eqref{eqn:v_ex'_arc},~\eqref{eqn:v_out_arc},~\eqref{eqn:e_escape_arc}, Claim~\ref{claim:conn2'_arc}, Claim~\ref{claim:e_conn_arc}, and $V(D) = V_{\rm ex}' \cup V_{\rm out} \cup X_1 \cup X_1'$ we have
\begin{align}
|E_{\rm abs}| &\leq  |E_{\rm escape}| + |E_{\rm path}| + |E_0| + |E_1| + |E_1'| + |E_{\rm conn}|\nonumber \\
& \leq  4k|V_{\rm ex}'| + (|V_{\rm ex}'| + |V_{\rm out}| + |X_1'| + k) + (k|V_{\rm out}| - k + k \overline{\Delta}) + (k|X_1| - k + k \overline{\Delta}) \nonumber\\
&\qquad + ((k-1)|X_1'| + (k-1) \overline{\Delta} + 2k^2 - 3k + 1) + 6k(8k+3\overline{\Delta})\nonumber \\
& \leq  k(|V_{\rm ex}'| + |V_{\rm out}| + |X_1| + |X_1'|) + (3k+1)|V_{\rm ex}'| + |V_{\rm out}| + 50k^2 + 21k\overline{\Delta}\nonumber\\
& \leq  kn + (3k+1)|V_{\rm ex}| + 96k(k+\overline{\Delta}) + 32(k+\overline{\Delta}) + |V_{\rm out}| + 50k^2 + 21k \overline\Delta \nonumber\\
& \leq  kn + (210k^2 + 181k \overline\Delta) + 32(k+\overline{\Delta}) + (5k+1)|V_{\rm ex}|. \label{eqn:e_abs*_arc}
\end{align}

Let us define
$$\mathcal{D} := (E_{\rm abs} , V_{\rm ex} , \mathcal{P} , W_i , W_o).$$

\begin{CLAIM}\label{claim:abs_end_arc}
$\mathcal{D}$ is a $k$-arc-absorber in $D$.
\end{CLAIM}
\begin{proof}
Both $(\rm A1')$ and $(\rm A2')$ are clear. Let $F \subseteq E(D)$ with $|F| \leq k-1$, and $u,v \in V(D)$ be two distinct vertices.

\begin{itemize} 
\item[$(\rm a)$] If $u \in V_{\rm ex}'$, then since $(E_{\rm escape} , V_{\rm ex}' , V_{\rm out})$ is a $k$-arc-escaper, there is a path from $u$ to $u' \in V_{\rm out}$ in $D-F$ using only edges in $E_{\rm escape}$, and there is a path from $u'$ to a vertex $u'' \in U_o$ in $D-F$ only using edges in $E_0$ by Claim~\ref{claim:conn2'_arc}. By Claim~\ref{claim:e_conn_arc}, there is a path from $u''$ to a vertex $u^+ \in W_o$ in $D-F$ only using edges in $E_{\rm conn}$.

\item[$(\rm b)$] If $u \in X_1'$, then there is a path from $u$ to $u' \in U_o \cup V_{\rm ex}$ in $D-F$ only using edges in $E_{\rm path} \cup E_1'$ by Claim~\ref{claim:conn2'_arc}. If $u' \in U_o$, then there is a path from $u'$ to a vertex $u^+ \in W_o$ in $D-F$ only using edges in $E_{\rm conn}$ by Claim~\ref{claim:e_conn_arc}. Otherwise if $u' \in V_{\rm ex}$, then there is a path from $u'$ to a vertex $u^+ \in W_o$ in $D-F$ only using edges in $E_{\rm abs}$ by $(\rm a)$.

\item[$(\rm c)$] If $u \in V_{\rm out} \cup X_1$, then there is a path from $u$ to a vertex $u' \in U_o$ in $D-F$ using only edges in $E_0 \cup E_1$ by Claim~\ref{claim:conn2'_arc}. By Claim~\ref{claim:e_conn_arc}, there is a path from $u'$ to a vertex $u^+ \in W_o$ in $D-F$ only using edges in $E_{\rm conn}$.
\end{itemize}

Hence there is a path in $D-F$ from $u$ to $u^+ \in W_o$ only using edges in $E_{\rm abs}$, proving $(\rm A3')$. Similarly, there is a path in $D-F$ from a vertex $v^+ \in W_i$ to $v$ only using edges in $E_{\rm abs}$, proving $(\rm A4')$.
\end{proof}

By Claim~\ref{claim:abs_end_arc} and~\eqref{eqn:e_abs*_arc}, this completes the proof of the lemma.
\end{proof}

\section{Proof of the main result}\label{sec:proof}
We divide Theorem~\ref{thm:main} into two parts as follows. First of all, the following theorem establishes the upper bound of the minimum number of edges in a strongly $k$-connected spanning subgraph.

\begin{THM}\label{thm:main_1}
For all integers $k,n \geq 1$ and $\overline{\Delta} \geq 0$, every strongly $k$-connected $n$-vertex digraph $D$ with $\overline{\Delta}(D) \leq \overline{\Delta}$ contains a strongly $k$-connected spanning subgraph with at most $kn + 790k\overline{\Delta} + 790k^2$ edges.
\end{THM}

Secondly, the following theorem establishes the upper bound of the minimum number of edges in a strongly $k$-arc-connected spanning subgraph.

\begin{THM}\label{thm:main_2}
For all integers $k,n \geq 1$ and $\overline{\Delta} \geq 0$, every strongly $k$-arc-connected $n$-vertex directed multigraph $D$ with $\overline{\Delta}(D) \leq \overline{\Delta}$ contains a strongly $k$-arc-connected spanning subgraph with at most $kn + 666k\overline{\Delta} + 666k^2$ edges.
\end{THM}

Both Theorems~\ref{thm:main_1} and~\ref{thm:main_2} prove Theorem~\ref{thm:main}. Now we are ready to prove Theorem~\ref{thm:main_1}.

\begin{proof}[Proof of Theorem~\ref{thm:main_1}]
%Let $D$ be an $n$-vertex digraph. For two distinct $u,v \in V(D)$, one can check $\kappa(u,v) \geq k$ in $O(kn^2)$ time by applying Ford-Fulkerson algorithm~\cite{ahuja1993}. Hence one can verify whether $D$ is strongly $k$-connected in $O(kn^4)$ by checking whether  $\kappa(u,v) \geq k$ for all $u,v \in V(D)$ with $u \ne v$ (for a better algorithm, see~\cite{gabow2006}).

Let $D$ be a strongly $k$-connected $n$-vertex digraph with $\overline{\Delta}(D) \leq \overline{\Delta}$. For $n < 4k+3$, we have $|E(D)| \leq 2\binom{n}{2} < 16k^2 + 20k + 6 \leq 790k(k+\overline{\Delta})$. For $4k+3 \leq n < 200(k + \overline{\Delta})$, let $D'$ be a minimally strongly $k$-connected spanning subgraph of $D$. By the result of Mader~\cite{mader1985}, we have $|E(D')| \leq 2kn \leq 400k(k+\overline{\Delta}) \leq 790k(k+\overline{\Delta})$. 

We may assume that $n \geq 200(k + \overline{\Delta})$. By Lemma~\ref{lem:trioexist}, $D$ contains a 3-tuple $(\mathcal{A},\mathcal{B},O^*)$ such that
\begin{align}\label{def:trio_main_1}
\textrm{$(\mathcal{A},\mathcal{B},O^*)$ is a $\left ( k+\overline{\Delta} ,\: \overline{\Delta} ,\: 30k + 35\overline{\Delta} ,\: 5(k+\overline{\Delta}) ,\: \frac{7(k+\overline{\Delta})}{3} \right )$-trio,}
\end{align}
where $\mathcal{A}$ consists of $5(k+\overline{\Delta})$ distinct 5-indominators $\left \{(D_i , A_i , x_i , a_i ) \right \}_{i=1}^{5(k+\overline{\Delta})}$, $\mathcal{B}$ consists of $5(k+\overline{\Delta})$ distinct 5-outdominators $\left \{(D_i' , B_i , x_i' , b_i) \right \}_{i=1}^{5(k+\overline{\Delta})}$, and $|O^*| \leq 24(k+\overline{\Delta})$. %Note that $(\mathcal{A},\mathcal{B},O^*)$ can be found in $O((k+\overline{\Delta})n)$.

Let $A := \bigcup_{i=1}^{5(k+\overline{\Delta})}A_i$ and $B := \bigcup_{i=1}^{5(k+\overline{\Delta})}B_i$. For $i \in [5(k+\overline{\Delta})]$, let $U_i^+ := \bigcap_{v \in A_i}N_{D_i}^+(v) \setminus \bigcup_{v \in A_i}N_{D_i}^-(v)$ and $U_i^- := \bigcap_{v \in B_i}N_{D_i'}^-(v) \setminus \bigcup_{v \in B_i}N_{D_i'}^+(v)$.

Since $|A|,|B| \leq 5 \cdot 5(k+\overline{\Delta})$ and $|O^*| < 24(k+\overline{\Delta})$, it follows that
$$ |A \cup B \cup O^*| \leq 74(k+\overline{\Delta}). $$

By Menger's theorem, let $P_1 , \dots , P_k$ be $k$ vertex-disjoint paths from $\left \{a_1 , \dots , a_k \right \}$ to $\left \{b_1 , \dots , b_k \right \}$ such that there is a permutation $\sigma : [k] \to [k]$ and for $i \in [k]$, $P_i$ is a path from $a_i$ to $b_{\sigma(i)}$. Without loss of generality, we may assume that $P_i$ is a minimal path from $a_i$ to $b_{\sigma(i)}$ for $i \in [k]$. Let $\mathcal{P} := \left \{P_1 , \dots , P_k \right \}$.

Since $|V(D)| - |A \cup B \cup O^*| \geq 200(k+\overline{\Delta}) - 74(k+\overline{\Delta}) \geq 39k + 38\overline{\Delta}$, we apply Lemma~\ref{lem:abs} so that $D$ contains a $k$-absorber 
$$\mathcal{D} := (E_{\rm abs} , A \cup B \cup O^* , \mathcal{P} , W_i , W_o)$$ 
with $W_i,W_o \subseteq V(D) \setminus (A \cup B \cup O^*)$, $|W_i|,|W_o| = 3k$, and 
\begin{align}
|E_{\rm abs}| & \leq kn + 226k(k+\overline{\Delta}) + 38(k+\overline{\Delta}) + (5k+1)|A \cup B \cup O^*| \nonumber\\
& \leq kn + 596k(k+\overline{\Delta}) + 112(k+\overline{\Delta}), \label{eqn:e_abs''}
\end{align}
since $|A \cup B \cup O^*| \leq 74(k+\overline{\Delta})$. %Note that $E_{\rm abs}$ can be found in $O(n^3 + kn^{2.5} + k(k+\overline{\Delta})n)$ by Lemma~\ref{lem:abs}.

Since $W_i,W_o \subseteq V(D) \setminus (A \cup B \cup O^*)$ with $|W_i|,|W_o| = 3k$ and~\eqref{def:trio_main_1}, we apply Lemma~\ref{lem:hubexist} with $3k$ playing the role of $w$. By $(\rm 2)$ of Lemma~\ref{lem:hubexist}, $D$ has a $k$-hub 
$$\mathcal{H} := (E_{\rm hub}, \left \{a_1 , \dots , a_k \right \} , \left \{b_1 , \dots , b_k \right \} , W_o , W_i)$$ 
such that 
\begin{align}\label{eqn:e_hub''}
|E_{\rm hub}| \leq 82k(k+\overline{\Delta}).
\end{align}

%Note that $E_{\rm hub}$ can be found in $O(k(k+\overline{\Delta})n + k^2)$. 
Let $E_L := E_{\rm abs} \cup E_{\rm hub}$. By~\eqref{eqn:e_abs''} and~\eqref{eqn:e_hub''},
\begin{align*}
|E_L| &\leq |E_{\rm abs}| + |E_{\rm hub}|\\
& \leq kn + 596k(k+\overline{\Delta}) + 82k(k+\overline{\Delta}) + 112(k+\overline{\Delta})\\
& \leq kn + 678k(k+\overline{\Delta}) + 112(k+\overline{\Delta})\\
& \leq kn + 790k(k+\overline{\Delta}).
\end{align*}

%In summary, a spanning subgraph $D' := (V(D),E_L)$ of $D$ can be found in $O(n^3 + kn^4 + k(k+\overline{\Delta})n + k^2)$. 
Let $D' := (V(D), E_L)$ be a spanning subgraph of $D$. Now it remains to prove that $D'$ is strongly $k$-connected. Let $S \subseteq V(D')$ with $|S| \leq k-1$ and $u,v \in V(D') \setminus S$. Let $i \in [k]$ be an index such that $V(P_i) \cap S = \emptyset$. If $u \in W_o$, then $u' := u$. Otherwise, $D'-S$ contains a path from $u$ to a vertex $u' \in W_o \setminus S$ since $\mathcal{D}$ is a $k$-absorber in $D$. Since $\mathcal{H}$ is a $k$-hub, $D'-S$ contains a path from $u'$ to $a_i$, showing that $D'-S$ contains a path from $u$ to $a_i$. Similarly, $D'-S$ contains a path from $b_{\sigma(i)}$ to $v$. Connecting from $a_i$ to $b_{\sigma(i)}$ by $P_i$, we deduce that $D'-S$ contains a path from $u$ to $v$, as desired.
\end{proof}

Now we prove Theorem~\ref{thm:main_2}, and the proof is analogous to the proof of Theorem~\ref{thm:main_1}.

\begin{proof}[Proof of Theorem~\ref{thm:main_2}]
Let $D$ be a strongly $k$-arc-connected $n$-vertex digraph with $\overline{\Delta}(D) \leq \overline{\Delta}$. For $n < 100(k + \overline{\Delta})$, let $D'$ be a minimally strongly $k$-arc-connected spanning subgraph of $D$. By the result of Dalmazzo~\cite{dalmazzo1977}, we have $|E(D')| \leq 2kn \leq 200k(k+\overline{\Delta}) \leq 666k(k+\overline{\Delta})$. 

We may assume that $n \geq 100(k + \overline{\Delta})$. By Lemma~\ref{lem:trioexist}, $D$ contains a 3-tuple $(\mathcal{A},\mathcal{B},O^*)$ such that
\begin{align}\label{def:trio_main_2}
\textrm{$(\mathcal{A},\mathcal{B},O^*)$ is a $\left ( k+\overline{\Delta} ,\: \overline{\Delta} ,\: 5k + 10\overline{\Delta} ,\: 5(k+\overline{\Delta}) ,\: \frac{k+2\overline{\Delta}}{3} \right )$-trio,}
\end{align}
where $\mathcal{A}$ consists of $5(k+\overline{\Delta})$ distinct 5-indominators $\left \{(D_i , A_i , x_i , a_i ) \right \}_{i=1}^{5(k+\overline{\Delta})}$, $\mathcal{B}$ consists of $5(k+\overline{\Delta})$ distinct 5-outdominators $\left \{(D_i' , B_i , x_i' , b_i) \right \}_{i=1}^{5(k+\overline{\Delta})}$, and $|O^*| \leq \frac{10k + 20\overline{\Delta}}{3} \leq 4k + 7\overline{\Delta}$. %Note that $(\mathcal{A},\mathcal{B},O^*)$ can be found in $O((k+\overline{\Delta})n)$.

Let $A := \bigcup_{i=1}^{5(k+\overline{\Delta})}A_i$ and $B := \bigcup_{i=1}^{5(k+\overline{\Delta})}B_i$. For $i \in [5(k+\overline{\Delta})]$, let $U_i^+ := \bigcap_{v \in A_i}N_{D_i}^+(v) \setminus \bigcup_{v \in A_i}N_{D_i}^-(v)$ and $U_i^- := \bigcap_{v \in B_i}N_{D_i'}^-(v) \setminus \bigcup_{v \in B_i}N_{D_i'}^+(v)$.

Since $|A|,|B| \leq 5 \cdot 5(k+\overline{\Delta})$ and $|O^*| < 4k + 7\overline{\Delta}$, it follows that
$$ |A \cup B \cup O^*| \leq 57(k+\overline{\Delta}). $$

By Menger's theorem, let $P_1 , \dots , P_k$ be $k$ edge-disjoint paths from $\left \{a_1 , \dots , a_k \right \}$ to $\left \{b_1 , \dots , b_k \right \}$ such that there is a permutation $\sigma : [k] \to [k]$ where for $i \in [k]$, $P_i$ is a path from $a_i$ to $b_{\sigma(i)}$. Let $\mathcal{P} := \left \{P_1 , \dots , P_k \right \}$.

The rest of the proof is analogous to the proof of Theorem~\ref{thm:main_1}. As $|V(D)| - |A \cup B \cup O^*| \geq 100(k+\overline{\Delta}) - 57(k+\overline{\Delta}) \geq 33k + 32\overline{\Delta}$, we apply Lemma~\ref{lem:abs_arc} so that $D$ contains a $k$-arc-absorber 
$$\mathcal{D_{\rm arc}} := (E_{\rm abs} , A \cup B \cup O^* , \mathcal{P} , W_i , W_o)$$ 
with $W_i,W_o \subseteq V(D) \setminus (A \cup B \cup O^*)$, $|W_i|,|W_o| = 3k$, and 
\begin{align}
|E_{\rm abs}| & \leq kn + 210k(k+\overline{\Delta}) + 32(k+\overline{\Delta}) + (5k+1)|A \cup B \cup O^*| \nonumber\\
& \leq kn + 495k(k+\overline{\Delta}) + 89(k+\overline{\Delta}), \label{eqn:e_abs'''}
\end{align}
since $|A \cup B \cup O^*| \leq 57(k+\overline{\Delta})$.

Since $W_i,W_o \subseteq V(D) \setminus (A \cup B \cup O^*)$ with $|W_i|,|W_o| = 3k$ and~\eqref{def:trio_main_2}, we apply Lemma~\ref{lem:hubexist_arc} with $3k$ playing the role of $w$. By $(\rm 2)$ of Lemma~\ref{lem:hubexist_arc}, $D$ has a $k$-arc-hub 
$$\mathcal{H_{\rm arc}} := (E_{\rm hub}, \left \{a_1 , \dots , a_k \right \} , \left \{b_1 , \dots , b_k \right \} , W_o , W_i)$$ 
such that 
\begin{align}\label{eqn:e_hub'''}
|E_{\rm hub}| \leq 82k(k+\overline{\Delta}).
\end{align}

%Note that $E_{\rm hub}$ can be found in $O(k(k+\overline{\Delta})n + k^2)$. 
Let $E_L := E_{\rm abs} \cup E_{\rm hub}$. By~\eqref{eqn:e_abs'''} and~\eqref{eqn:e_hub'''},
\begin{align*}
|E_L| &\leq |E_{\rm abs}| + |E_{\rm hub}|\\
& \leq kn + 495k(k+\overline{\Delta}) + 82k(k+\overline{\Delta}) + 89(k+\overline{\Delta})\\
& \leq kn + 577k(k+\overline{\Delta}) + 89(k+\overline{\Delta})\\
& \leq kn + 666k(k+\overline{\Delta}).
\end{align*}

Let $D' := (V(D), E_L)$ be a spanning subgraph of $D$. Now it remains to prove that $D'$ is strongly $k$-arc-connected. Let $F \subseteq E(D')$ with $|F| \leq k-1$ and $u,v \in V(D')$. Let $i \in [k]$ be an index such that $E(P_i) \cap F = \emptyset$. If $u \in W_o$, then $u' := u$. Otherwise, $D'-F$ contains a path from $u$ to a vertex $u' \in W_o$ since $\mathcal{D_{\rm arc}}$ is a $k$-arc-absorber in $D$. Since $\mathcal{H_{\rm arc}}$ is a $k$-arc-hub, $D'-F$ contains a path from $u'$ to $a_i$, showing that $D'-F$ contains a path from $u$ to $a_i$. Similarly, $D'-F$ contains a path from $b_{\sigma(i)}$ to $v$. Connecting from $a_i$ to $b_{\sigma(i)}$ by $P_i$, we deduce that $D'-F$ contains a path from $u$ to $v$, as desired.
\end{proof}

%%%%%%%%%%%%%%%%%%%%%%%%%%%%%%%%%%%%%%%%%%%%%%%%%%%%%%%%%%%%%%%%%%%%%%%%

\section{Concluding Remarks}\label{sec:conclude}
\subsection{Improving the upper bound}
For any integer $k \geq 1$ and a digraph $D$, let $h(k,D)$ be the minimum number of edges in a spanning subgraph $D'$ of $D$ with $\delta^+(D'),\delta^-(D') \geq k$. Bang-Jensen, Huang, and Yeo~\cite{bang2004spanning} proved that $h(k,T) \leq k|V(T)| + \frac{k(k+1)}{2}$ for every tournament $T$ with $\delta^+(T),\delta^-(T) \geq k$, and $h(k,T) \leq k|V(T)| + \frac{k(k-1)}{2}$ if the tournament $T$ is strongly $k$-arc-connected (see~\cite[Proposition 2.1]{bang2004spanning}). They also conjectured that $h(k,T)$ is equal to the minimum number of edges in a strongly $k$-arc-connected spanning subgraph of $T$, for every strongly $k$-arc-connected tournament $T$. Using the ideas of the proof of~\cite[Proposition 2.1]{bang2004spanning}, we prove the following.

\begin{PROP}\label{prop:opt}
For integers $k,n \geq 1$ and an integer $\overline{\Delta} \geq 2k-1$, $h(k,D) \leq kn + k\overline{\Delta}$ for every strongly $k$-arc-connected $n$-vertex digraph $D$ with $\overline{\Delta}(D) \leq \overline{\Delta}$.
\end{PROP}
\begin{proof}
Let $V_1 := \left \{v_1 \: : \: v \in V(D) \right \}$ and $V_2 := \left \{v_2 \: : \: v \in V(D) \right \}$ be two disjoint copies of $V(D)$. Let $\mathcal{N}$ be a network with a vertex-set $\left \{s,t \right \} \cup V_1 \cup V_2$ and an edge-set 
$$\left \{sv_1 \: : \: v \in V(D) \right \} \cup \left \{v_2 t \: : \: v \in V(D) \right \} \cup \left \{u_1 v_2 \: : \: uv \in E(D) \right \}.$$

We may assume that $s,t \notin V_1 \cup V_2$. Let $\ell : E(\mathcal{N}) \to \mathbb{R}_{\geq 0}$ be a lower bound function such that $\ell(sv_1) = \ell(v_2 t) = k$ for every $v \in V(D)$, and $\ell(e) = 0$ for the other edges $e \in E(\mathcal{N})$. Let $c : E(\mathcal{N}) \to \mathbb{R}_{\geq 0} \cup \left \{\infty \right \}$ be a capacity function such that $c(sv_1) = c(v_2 t) = \infty$ for every $v \in V(D)$ and $c(u_1 v_2)=1$ for every $uv \in E(D)$. One can easily check that the minimum $(s,t)$-flow of $\mathcal{N}$ is equal to $h(k,D)$. By Min-Flow Max-Demand Theorem (see~\cite[Theorem 4.9.1]{bang2008digraphs}), the minimum $(s,t)$-flow is equal to the maximum of $\ell(S,T)-c(T,S)$, where $\left \{S,T \right \}$ is a partition of $V(\mathcal{N})$ with $s \in S$ and $t \in T$.

Let $\left \{S,T \right \}$ be a partition of $V(\mathcal{N})$ with $s \in S$ and $t \in T$. For $A,B \in \left \{S,T \right \}$, let $V_{A,B} := \left \{v \in V(D) \: : \: v_1 \in A, \: v_2 \in B \right \}$. Then
\begin{align*}
\ell(S,T) &= k(V_{T,S} + V_{T,T}) + k(V_{S,S} + V_{T,S}) = k|V(D)| + k|V_{T,S}| - k|V_{S,T}|\\
c(T,S) &= e_D (V_{T,S} \cup V_{T,T} , V_{S,S} \cup V_{T,S}) = |E(D[V_{T,S}])| + e_D(V_{T,S} , V_{S,S}) + e_D(V_{T,T} , V_{S,S} \cup V_{T,S}).
\end{align*}

Now we aim to prove $\ell(S,T)- c(T,S) \leq kn + k\overline\Delta$. If there are at least three empty sets in $\left \{V_{S,S},V_{S,T},V_{T,S},V_{T,T} \right \}$, then it is easy to check that $\ell(S,T)-c(T,S) \leq kn$. Hence we may assume that there are at least two nonempty sets in $\left \{V_{S,S},V_{S,T},V_{T,S},V_{T,T} \right \}$. We claim that 
$$\ell(S,T)-c(T,S) \leq kn + k|V_{T,S}| - |E(D[V_{T,S}])| - k \leq kn + k \overline\Delta.$$

If $V_{S,T}=V_{T,T}=\emptyset$ then $e_D(V_{T,T},V_{S,S} \cup V_{T,S}) = 0$ and $e_D(V_{T,S},V_{S,S}) \geq k$, implying $\ell(S,T)-c(T,S) \leq kn + k|V_{T,S}| - |E(D[V_{T,S}])| - k$. If $V_{S,T} = \emptyset$ and $V_{T,T} \ne \emptyset$ then $e_D(V_{T,T} , V_{S,S} \cup V_{T,S}) \geq k$ since $D$ is strongly $k$-arc-connected. Therefore, either $|V_{S,T}| \geq 1$ or $e_D(V_{T,T} , V_{S,S} \cup V_{T,S}) \geq k$. In either case, it follows that $\ell(S,T) - c(T,S) \leq kn + k|V_{T,S}| - |E(D[V_{T,S}])| - k.$

Since $|E(D[V_{T,S}])| \geq \max (0 , |V_{T,S}|(|V_{T,S}| - 1 - \overline{\Delta}) / 2 )$, we have
\[k|V_{T,S}| - |E(D[V_{T,S}])| \leq
  \begin{cases}
k|V_{T,S}|& \text{if }|V_{T,S}| < \overline{\Delta}+1, \\
k|V_{T,S}| - \frac{|V_{T,S}|}{2}(|V_{T,S}|-\overline{\Delta}-1)& \text{otherwise}.
  \end{cases}
\]

If $|V_{T,S}| < \overline{\Delta}+1$, then $k|V_{T,S}| - |E(D[V_{T,S}])| < k\overline{\Delta} + k$ and thus $\ell(S,T) - c(T,S) < kn + k\overline{\Delta}$. Let us assume that $|V_{T,S}| \geq \overline{\Delta}+1$. Since the function $f(x) = \frac{x(2k + \overline{\Delta} + 1 - x)}{2}$ is a decreasing function for $x \geq k + \frac{\overline{\Delta} + 1}{2}$ and $|V_{T,S}| \geq \overline{\Delta}+1 \geq k + \frac{\overline{\Delta} + 1}{2}$, we have
\begin{align*}
k|V_{T,S}| - |E(D[V_{T,S}])| &\leq k|V_{T,S}| - \frac{|V_{T,S}|}{2}(|V_{T,S}|-\overline{\Delta}-1) = f(|V_{T,S}|)\\
& \leq f(\overline{\Delta}+1) = k\overline{\Delta} + k.
\end{align*}
and thus $\ell(S,T) - c(T,S) \leq kn + k|V_{T,S}| - |E(D[V_{T,S}])| - k \leq kn + k\overline{\Delta}$. This completes the proof.
\end{proof}

Since the oriented graph $G_{n_1 , n_2 , k , \overline{\Delta}}$ in Section~\ref{sec:lowerbdd} with $n = n_1 + n_2 + \overline{\Delta} + 1$ satisfies $h(k , G_{n_1 , n_2 , k , \overline{\Delta}}) \geq kn + k \overline{\Delta}$ if $\overline{\Delta} \geq 2k-1$, Proposition~\ref{prop:opt} implies that $h(k , G_{n_1 , n_2 , k , \overline{\Delta}}) = kn + k \overline{\Delta}$ when $\overline{\Delta} \geq 2k-1$. 

For $k=1$, Bang-Jensen, Huang, and Yeo~\cite[Theorem 8.3]{bang2004spanning} proved that every strongly connected $n$-vertex digraph $D$ with $\overline{\Delta}(D) \leq \overline{\Delta}$ contains a spanning strongly connected subgraph with at most $n + \overline{\Delta}$ edges. We conjecture that the multiplicative constant of $k\overline{\Delta}$ of Theorem~\ref{thm:main} can be improved to 1, which is best possible.

\begin{CONJ}\label{que:lowerbdd}
~\
\begin{itemize} 
\item[$(1)$] There is $C>0$ such that for integers $k,n \geq 1$ and $\overline{\Delta} \geq 0$, every strongly $k$-connected $n$-vertex digraph $D$ with $\overline{\Delta}(D) \leq \overline{\Delta}$ contains a strongly $k$-connected spanning subgraph with at most $kn + k\overline{\Delta} + Ck^2$ edges.

\item[$(2)$] There is $C'>0$ such that for integers $k,n \geq 1$ and $\overline{\Delta} \geq 0$, every strongly $k$-arc-connected $n$-vertex directed multigraph $D$ with $\overline{\Delta}(D) \leq \overline{\Delta}$ contains a strongly $k$-arc-connected spanning subgraph with at most $kn + k\overline{\Delta} + C'k^2$ edges.
\end{itemize}
\end{CONJ}

Since Mader~\cite{mader1985} proved that every strongly $k$-connected $n$-vertex digraph contains a strongly $k$-connected spanning subgraph with at most $2kn - k(k+1)$ edges, Conjecture~\ref{que:lowerbdd} is true for $\overline{\Delta} \geq n-k-1$.

% Improving $800k\overline{\Delta}$ to $k\overline{\Delta}$, which is the common generalisation of Mader and KKKS.

\subsection{Almost-regular spanning subgraphs}
There are many studies regarding finding spanning regular subgraphs in tournaments. One of the typical examples of spanning regular subgraphs is a union of edge-disjoint Hamiltonian cycles, and there are some results relating edge-disjoint Hamiltonian cycles and the vertex-connectivity of tournaments. Thomassen~\cite{thomassen1982} conjectured that there is a function $f : \mathbb{N} \to \mathbb{N}$ such that every strongly $f(k)$-connected tournament contains $k$ edge-disjoint Hamiltonian cycles, and K\"uhn, Lapinskas, Osthus, and Patel~\cite{kuhn2014proof} proved that $f(k) = O(k^2 (\log k)^2)$ suffices and constructed a strongly $\frac{(k-1)^2}{4}$-connected tournament with no $k$ edge-disjoint Hamiltonian cycles. Recently, Pokrovskiy~\cite{pokrovskiy2016edge} proved that $f(k) = O(k^2)$ suffices, which is asymptotically sharp. 

As a variation of the problem, one may ask the minimum $m=m(k)$ such that every strongly $mk$-connected tournament $T$ contains a spanning $k$-regular subgraph. The following lemma proves that $m \geq \frac{k+1}{2}$, and the result of Pokrovskiy~\cite{pokrovskiy2016edge} is asymptotically best possible even if we relax the condition of existence of $k$ edge-disjoint Hamiltonian cycles to the existence of spanning $k$-regular subgraph. Recall that $T_{n_1 , n_2 , k}$ is a strongly $k$-connected $(n_1+n_2+k)$-vertex tournament defined in Section~\ref{sec:lowerbdd}. We remark that an almost identical construction can be found in~\cite[Proposition 5.1]{kuhn2014proof}.

\begin{LEM}\label{lem:lowerbdd2}
Let $m,k \geq 1$ be integers. For a $(5mk+2)$-vertex tournament $T_{2mk+1 , 2mk+1 , mk}$, every spanning subgraph $D$ of $T_{2mk+1 , 2mk+1 , mk}$ with $\delta^{+}(D) , \delta^{-}(D) \geq k$ contains at least $\frac{k-2m+1}{5m}$ vertices of either in-degree or out-degree more than $k$ in $D$.
\end{LEM}
\begin{proof}
Let $T_{2mk+1,2mk+1,mk}$ be the tournament with subtournaments $T_1$, $T_2$ and $T_3$ defined in Section~\ref{sec:lowerbdd}. 

Let $D$ be any spanning subgraph of $T$ such that $\delta^{+}(D),\delta^{-}(D) \geq k$. Let $S^+ \subseteq V(T_2)$ be the set of vertices $v$ in $V(T_2)$ such that $d_D^{+}(v) > k$. 

Since $d_T^{+}(v) \leq 5mk+1$ for any $v \in V(T_2)$ and every vertex in $V(T_1)$ has in-degree at least $k$ in $D$, it follows that $(5mk+1)|S^+| + k(2mk+1 - |S^+|) \geq \sum_{v \in V(T_2)}d_D^{+}(v)$ and $e_D(V(T_2),V(T_1))$ is at least $\frac{k(k+1)}{2}$. Hence
\begin{align*}
(5mk+1)|S^+| + k(2mk+1 - |S^+|) &\geq \sum_{v \in V(T_2)}d_D^{+}(v) \\
&\geq e_D(V(T_2),V(T_1)) - e_D(V(T_3),V(T_2)) + \sum_{w \in V(T_2)}d_D^{-}(w)\\
&\geq \frac{k(k+1)}{2} - mk + k(2mk+1),
\end{align*}
implying that $|S^+| \geq \frac{k(k+1-2m)}{2(5mk-k+1)} \geq \frac{k+1-2m}{10m}$. Let $S^- \subseteq V(T_3)$ be the set of vertices $v$ in $V(T_3)$ such that $d_D^{-}(v) > k$. Similarly, $|S^-| \geq \frac{k+1-2m}{10m}$, and it follows that $D$ contains at least $\frac{k-2m+1}{5m}$ vertices with either in-degree or out-degree more than $k$ in $D$.
\end{proof}

Rather than finding spanning regular subgraphs in semicomplete digraphs, we may consider finding \emph{almost} regular spanning subgraph (all vertices except few vertices have the same in/out-degrees) in semicomplete digraphs. Corollary~\ref{cor:tournaments} implies that every strongly $k$-connected semicomplete digraph contains a strongly $k$-connected spanning subgraph such that all vertices except for $O(k^2)$ vertices have both in-degree and out-degree exactly $k$. We conjecture the following.

\begin{CONJ}\label{que:almostreg}
~\
\begin{itemize} 
\item[$(\rm 1)$] For integers $k,n \geq 1$ and given a strongly $k$-connected semicomplete digraph $D$, there exists a set $S \subseteq V(D)$ with $|S| = O(k)$ such that there is a strongly $k$-connected spanning subgraph $D'$ of $D$ with $d_{D'}^{+}(v) = d_{D'}^{-}(v) = k$ for every $v \in V(D) \setminus S$, and $d_{D'}^{+}(w) = d_{D'}^{-}(w) = O(k)$ for every $w \in V(D)$.

\item[$(\rm 2)$] For integers $k,n \geq 1$ and given a strongly $k$-arc-connected semicomplete directed multigraph $D$, there exists a set $S \subseteq V(D)$ with $|S| = O(k)$ such that there is a strongly $k$-arc-connected spanning subgraph $D'$ of $D$ with $d_{D'}^{+}(v) = d_{D'}^{-}(v) = k$ for every $v \in V(D) \setminus S$, and $d_{D'}^{+}(w) = d_{D'}^{-}(w) = O(k)$ for every $w \in V(D)$.
\end{itemize}
\end{CONJ}

Note that the statements in Conjecture~\ref{que:almostreg} imply that $|E(D')| \leq k|V(D)| + O(k^2)$, strengthening Corollary~\ref{cor:tournaments}. By Lemma~\ref{lem:lowerbdd2}, we remark that the size $O(k)$ of $S$ cannot be improved further, since every spanning subgraph $D$ of a tournament $T_{2k+1,2k+1,k}$ with $\delta^{+}(D),\delta^{-}(D) \geq k$ contains at least $\frac{k-1}{4}$ vertices of either in-degree or out-degree more than $k$ in $D$.

\section{Acknowledgements}
The author would like to thank Sang-il Oum and Jaehoon Kim for their valuable comments and suggestions. The author would also like to thank anonymous referees for their careful reading and suggestions.

%\bibliographystyle{abbrv}
%\bibliography{ref-ecd}

\end{document}